\documentclass[11pt,a4paper]{article}

\newcommand{\graphF}[1]{}
\newcommand{\graphR}[1]{\framebox[5cm][c]{\rule{0pt}{2cm}}}

\usepackage[utf8]{inputenc}
\usepackage[T1]{fontenc}
\usepackage{lmodern}
\usepackage[english]{babel}
\usepackage{amsmath, amssymb, amsthm, mathtools, enumitem, xcolor, tcolorbox, pifont, multirow}
\usepackage{subcaption}
\usepackage{mathrsfs}
\usepackage{geometry}
\usepackage{fontawesome}
\usepackage{microtype}
\usepackage[export]{adjustbox}
\usepackage[normalem]{ulem} 

\usepackage{booktabs}
\usepackage{siunitx}
\newcommand{\ma}{\alpha}

\newcommand{\ms}{\sigma}

\newcommand{\wX}{\widetilde{X}}

\def\argmax{\mathop{\rm argmax}}
\def\argmin{\mathop{\rm argmin}}

\DeclareMathOperator{\diam}{diam}
\newcommand{\be}{\begin{eqnarray}} 
\newcommand{\ee}{\end{eqnarray}}
\newcommand{\beno}{\begin{eqnarray*}}
\newcommand{\eeno}{\end{eqnarray*}}

\usepackage{algorithm2e}
\RestyleAlgo{ruled}
\makeatletter
\renewcommand{\@algocf@capt@plain}{above}
\makeatother

\usepackage[numbers]{natbib}
\usepackage{hyperref}
\usepackage{aliascnt}
\usepackage[capitalise,noabbrev]{cleveref}

\hypersetup{
  colorlinks,
  citecolor=green!80!black,
  linkcolor=blue!90!black,
  urlcolor=blue!90!black
}

\theoremstyle{plain}

\newtheorem{theorem}{Theorem}[section]

\newaliascnt{lemma}{theorem}
\newtheorem{lemma}[lemma]{Lemma}
\aliascntresetthe{lemma}

\newaliascnt{proposition}{theorem}
\newtheorem{proposition}[proposition]{Proposition}
\aliascntresetthe{proposition}

\newaliascnt{corollary}{theorem}

\aliascntresetthe{corollary}

\theoremstyle{definition}

\newaliascnt{definition}{theorem}

\aliascntresetthe{definition}

\newaliascnt{assumption}{theorem}
\newtheorem{assumption}[assumption]{Assumption}
\aliascntresetthe{assumption}

\newaliascnt{example}{theorem}

\aliascntresetthe{example}

\newaliascnt{algo}{theorem}
\newtheorem{algo}[algo]{Algorithm}
\aliascntresetthe{algo}

\theoremstyle{remark}

\newaliascnt{remark}{theorem}
\newtheorem{remark}[remark]{Remark}
\aliascntresetthe{remark}

\crefname{theorem}{theorem}{theorems}
\Crefname{theorem}{Theorem}{Theorems}

\crefname{lemma}{lemma}{lemmas}
\Crefname{lemma}{Lemma}{Lemmas}

\crefname{proposition}{proposition}{propositions}
\Crefname{proposition}{Proposition}{Propositions}

\crefname{corollary}{corollary}{corollaries}
\Crefname{corollary}{Corollary}{Corollaries}

\crefname{assumption}{assumption}{assumptions}
\Crefname{assumption}{Assumption}{Assumptions}

\crefname{definition}{definition}{definitions}
\Crefname{definition}{Definition}{Definitions}

\crefname{remark}{remark}{remarks}
\Crefname{remark}{Remark}{Remarks}

\crefname{example}{example}{examples}
\Crefname{example}{Example}{Examples}

\crefname{algo}{algorithm}{algorithms}
\Crefname{algo}{Algorithm}{Algorithms}

\Crefname{algocfline}{Algorithm}{Algorithms}

\newcommand{\R}{\mathbb{R}}
\newcommand{\N}{\mathbb{N}}

\newcommand{\F}{\mathbb{F}}
\newcommand{\E}{\mathbb{E}}
\newcommand{\PP}{\mathbb{P}}
\newcommand{\Fcal}{\mathcal{F}}
\newcommand{\Acal}{\mathcal{A}}
\newcommand{\Mcal}{\mathcal{M}}
\newcommand{\Lcal}{\mathcal{L}}
\newcommand{\Gcal}{\mathcal{G}}
\newcommand{\Ecal}{\mathcal{E}}
\newcommand{\Zcal}{\mathcal{Z}}

\newcommand{\cG}{{\mathcal G}}
\newcommand{\eps}{\varepsilon}
\newcommand{\pr}[1]{{#1}^{\prime}}
\newcommand{\de}{\mathrm{d}}

\title{Neural feedback approximation for stochastic control with degenerate diffusions: error estimates and numerical analysis}

\author{
Olivier Bokanowski\thanks{Universit{\'e} Paris Cit{\'e}, Laboratoire Jacques-Louis Lions,
\href{mailto:olivier.bokanowski@u-paris.fr}{olivier.bokanowski@u-paris.fr}.
This research benefited from the support of the FMJH Program PGMO and from the support to this program from EDF.}
\and
Jean-Fran{\c c}ois Chassagneux\thanks{ENSAE, CREST and Institut Polytechnique de Paris,
\href{mailto:jean-francois.chassagneux@ensae.fr}{jean-francois.chassagneux@ensae.fr}. This research has benefited from the support of the ANR Project ReLISCoP (ANR-21-CE40-0001).}
\and
Marco Scaratti\thanks{Corresponding author:
\href{mailto:marco.scaratti@univr.it}{marco.scaratti@univr.it}.}\ 
\thanks{University of Verona and Universit{\'e} Paris Cit{\'e}, Laboratoire Jacques-Louis Lions.}
\and
Xavier Warin\thanks{EDF Lab Paris-Saclay and FiME, Laboratoire de Finance des March{\'e}s de l'Energie, 91120 Palaiseau, France,
\href{mailto:xavier.warin@edf.fr}{xavier.warin@edf.fr}.}
}

\date{\today}

\begin{document}

\maketitle

\begin{abstract}
We study finite-horizon stochastic optimal control problems and approximate the resulting time-discrete formulation by a direct policy-learning problem over neural-network feedback maps. 
We prove a quantitative convergence estimate, in an averaged sense, for the error between the time-discrete value and the value induced by an approximately optimized neural policy. The bound separates the approximation of near-optimal feedback policies, the localization of stochastic trajectories on compact sets, and the optimization tolerance in training. The analysis does not require transition-density assumptions and covers possibly degenerate diffusions and deterministic controlled dynamics in a unified framework. Numerical experiments are provided for a degenerate stochastic radial target problem, a Hamilton--Jacobi--Bellman benchmark, and a gas storage problem, illustrating the approach and separating the main error sources: time discretization, restriction to piecewise-constant policies, neural-network approximation, and Monte Carlo evaluation.
\end{abstract}

\bigskip
\noindent\textbf{Keywords.}
Stochastic optimal control; degenerate diffusions; neural networks; feedback control approximation; error estimates; time-discrete approximation.

\medskip
\noindent\textbf{MSC 2020.} Primary: 93E20, 49M25, 65C30. Secondary: 68T07, 60H10.

\section{Introduction}
\label{sec:introduction}

Stochastic optimal control problems arise in many applications in which decisions must be made dynamically under uncertainty, including finance, economics, energy management, operations research, engineering, and models of large interacting populations. In continuous time, these problems are usually formulated through controlled stochastic differential equations and have been studied through several classical approaches: among others, the dynamic programming principle (DPP) and the associated Hamilton--Jacobi--Bellman (HJB) equation \cite{fleming_soner_2006,pham2009continuous}, weak formulations based on martingale methods, changes of measure, backward stochastic differential equations (BSDEs) \cite{elkaroui_peng_quenez_1997}, and stochastic maximum principle methods based on Hamiltonians and adjoint equations \cite{peng1990general,yong1999stochastic}.
Despite this well-established theory, the numerical solution of stochastic control problems remains challenging in moderate and high dimensions. Grid-based approximations of HJB equations suffer from the curse of dimensionality, while computing accurate feedback controls is particularly difficult when the dynamics are degenerate or when the diffusion coefficients are controlled.

\medskip
In recent years, neural-network methods have provided a flexible computational tool for high-dimensional stochastic control and related nonlinear PDEs. Broadly speaking, the existing literature can be divided into two main families. The first one is \emph{value-based}: the neural network is used to approximate the value function, its gradient, a BSDE component, or the residual of a PDE. The second one is \emph{policy-based}: the feedback control itself is parameterized by a neural network and optimized directly through simulated trajectories. These two points of view are related, but they lead to different algorithms and to different theoretical questions.

A first major line of work concerns neural-network solvers for high-dimensional PDEs and BSDEs. The deep BSDE method of \cite{han-jen-e-18} reformulates semilinear parabolic PDEs through BSDEs and approximates the unknown gradient, or equivalently the BSDE control process, by neural networks. Several extensions and variants of this probabilistic PDE viewpoint have since been developed; see, for instance, \cite{han2017deep,hpw2021,beck2019machine,pham2021neural,germain2021neural}. The convergence of the deep BSDE method was analyzed in \cite{han_convergence_2020}, where a posteriori error estimates are proved for coupled FBSDEs under assumptions ensuring stability and neural-network approximation. More recent theoretical analyses have also addressed stochastic-control formulations based on the stochastic maximum principle; see, for instance, \cite{huang_negyesi_oosterlee_2025}. These works are closely related to stochastic control, but their primary object is the PDE/FBSDE representation: neural networks approximate the value function, its gradient, or adjoint variables, rather than directly approximating a Markov feedback policy.
Other equation-based approaches include residual or collocation methods, such as physics-informed neural networks (PINNs) \cite{raissi_physics_informed_2019} and the deep Galerkin method \cite{sirignano_spiliopoulos_dgm_2018}, as well as actor-critic and generalized policy-iteration schemes for HJB equations \cite{zhou_han_lu_2021,cohen_hebner_jiang_sirignano_2025,lefebvre2023differential}. These methods are complementary to the direct feedback-policy approximation studied here.

A second line of work, closer to the present paper, is based on dynamic programming in discrete time. \cite{hure_deep_2021,bachouch_deep_2022,abbas2026stochastic} combine deep neural networks with the dynamic programming principle to solve finite-horizon stochastic control problems.
In particular, in \cite{hure_deep_2021,bachouch_deep_2022}
the proposed algorithms first approximate optimal policies by neural networks and then approximate value functions through Monte Carlo regression, using performance iteration, hybrid iteration, and regress-now methods.
The convergence analysis is expressed in terms of neural-network approximation errors and statistical regression errors, while the optimization error is left aside.
The two works are particularly relevant for the present paper because they also approximate feedback controls by neural networks and provide theoretical guarantees. The main difference is methodological. Their algorithms are backward, DPP-based, and local in time, whereas the method studied here is a global direct-policy method: all feedback policies on the time grid are optimized simultaneously by minimizing the expected cost induced by the simulated controlled dynamics.
A further difference concerns the treatment of degenerate dynamics. 
The error analysis in \cite{hure_deep_2021} relies on a bounded-density assumption for the controlled transition kernels with respect to the training distribution.
This assumption rules out singular transition kernels, and therefore excludes deterministic dynamics and degenerate stochastic schemes whose transition laws are not absolutely continuous with respect to the training distribution.
By contrast, our estimates are based on moment stability estimates for controlled trajectories and localization arguments and do not require any transition-density assumption. This allows us to treat possibly degenerate diffusions as well as deterministic controlled dynamics.

\medskip
The direct policy-learning approach for high-dimensional stochastic control was proposed in \cite{han_deep_2016}. There, the time-dependent controls are represented by feedforward neural networks, stacked through the discretized dynamics, and the objective functional of the stochastic control problem is used directly as the training loss. Note that the underlying idea of optimizing a parametrized policy directly through simulated costs was already used, e.g., in \cite{barrera2006numerical} for gas storage and swing option problems, with low-dimensional parametric families of strategies in place of neural networks. In both cases, this global forward approach is algorithmically close to the method analyzed in
the present paper.

Related policy-based and continuous-time reinforcement-learning approaches have been developed in several directions, including continuous-time portfolio selection \cite{wang_zhou_2020}, actor-critic methods for mean-field control \cite{frikha2023actor,pham2023actor}, and simulation-free methods for stochastic control \cite{hua_lauriere_vandeneijnden_2024}.
These works demonstrate the relevance of policy parameterization and trajectory-based optimization. However, many of them focus primarily on algorithm design, numerical performance, or the training procedure itself. They do not provide the type of quantitative feedback-policy approximation estimate developed here, in which the value error is decomposed into explicit policy approximation, localization, and optimization terms.

Our error analysis is closely related to recent weak error estimates for neural-network approximations of feedback strategies in deterministic control problems. In \cite{bok-pro-war-23}, neural networks are used to approximate feedback controls for first-order HJB equations, with error estimates in an averaged norm. The subsequent work \cite{bok-war-25} develops representation results and weak error estimates for differential games. In the spirit of these two works, our error estimate is measured in a weak or averaged sense. However, the stochastic setting requires new stability estimates for controlled stochastic numerical schemes and a localization argument to handle trajectories that may leave every compact subset of the state space with positive probability.
Further related literature concerns neural approximation of feedback controls for stochastic partial differential equations (SPDEs); for instance, the works \cite{stannat_vogler_wessels_2023,stannat_vogler_2025} develop approximation results for optimal feedback controls in stochastic reaction-diffusion settings, including neural-network approximations.

Although the present paper is written in a classical finite-dimensional setting, the analysis is also relevant for high-dimensional control problems obtained from particle or finite-dimensional approximations of infinite-dimensional models, such as mean field games and mean field control; see, e.g., \cite{hu_lauriere_2024}. Finite-particle and symmetric finite-dimensional approximations provide one way to reduce such problems to high-dimensional stochastic control problems, where neural-network methods can then be applied; see, for example, \cite{picarelli_extended_2025,soner2025learning, bokanowski2026numerical}.

\medskip
This paper develops a quantitative convergence analysis for direct neural-network approximations of feedback policies in finite-horizon stochastic control, allowing for degenerate diffusions and deterministic controlled dynamics. The analysis is carried out for the value function of a fixed time-discretized control problem, obtained by restricting controls to the time grid and discretizing the controlled dynamics.
The main result, \Cref{thm:error_estimate}, provides an explicit averaged error estimate for the value induced by an approximate neural feedback policy. The bound separates the approximation error of a near-optimal feedback policy by the neural-network class on compact sets, the localization error due to stochastic trajectories leaving these compact sets, and the optimization tolerance of the training problem. This decomposition is the central theoretical contribution of the paper: it identifies separately the roles of feedback approximation, stochastic propagation, and numerical optimization in the error of direct policy learning.
A distinctive feature of the analysis is that it does not require transition-density assumptions or non-degeneracy of the diffusion coefficient. It is instead based on stability estimates for controlled stochastic numerical schemes, approximation of near-optimal measurable feedback policies by Lipschitz feedbacks in value, and probabilistic localization via moment bounds. The resulting framework therefore covers non-degenerate diffusions, degenerate diffusions, and deterministic controlled dynamics in a unified way.
We stress that, while the error estimate of \Cref{thm:error_estimate} is
fully quantitative, the Lipschitz approximation of near-optimal measurable feedbacks
is not: the Lipschitz constant of the approximating policy is not quantified, so
that the resulting convergence result (\Cref{thm:convergence}) does not provide a
rate in terms of the size of the approximation class; see
\Cref{rem:lip_const_dependence,rem:non-quantitative-L} for a discussion.

The convergence theorem is deliberately formulated at the level of a fixed time grid. The discretization error between the continuous-time problem and the time-discrete problem is treated separately, using standard estimates for piecewise-constant controls and numerical schemes \cite{picarelli_probabilistic_2020,ass-bok-zid-2015,jakobsen_improved_2019,kloeden_numerical_1992}. This separation makes explicit the two layers of approximation: first, the discretization of the original stochastic control problem; second, the neural-network approximation of the resulting discrete feedback problem.

\medskip
The numerical experiments are designed to complement the theoretical analysis by isolating the main error mechanisms that arise in practice. We consider three examples. The first is a two-dimensional stochastic radial target problem with degenerate diffusion. The results show that the expected time-discretization orders are recovered when the neural-network and Monte Carlo errors remain below the discretization error.
The second example is an HJB benchmark for which semi-analytic formulas for both the value and the optimal feedback are available, allowing for comparison of both the learned values and controls. The experiment identifies the regime in which neural-network approximation and optimization errors become dominant after the time-discretization error has been sufficiently reduced.
In both examples, we additionally compare Euler--Maruyama with a weak second-order scheme.
The third example is a constrained gas storage problem motivated by energy applications, providing a realistic test case for which no closed-form benchmark is available. In this case, the exogenous price factors are simulated exactly on the grid, so the observed convergence mainly reflects the refinement of the decision grid, together with residual neural-network and training errors.

Taken together, the numerical results support the theoretical error decomposition: the experiments isolate the effects of time discretization, restriction to piecewise-constant policies, neural-network approximation and optimization, and Monte Carlo evaluation. To the best of our knowledge, a comparable combination of quantitative error estimates for global direct neural feedback learning and numerical separation of the main error sources within a framework that covers degenerate diffusions and deterministic controlled dynamics is not available in the existing literature.

\bigskip
The paper is organized as follows. \Cref{sec:continuous_time_setting,sec:discrete_time_setting} introduce the continuous-time control problem, the time-discrete approximation, and the feedback formulation. \Cref{sec:convergence_analysis} proves the stability, localization, and approximation estimates leading to the main convergence theorem. \Cref{sec:numerical_algorithm} describes the direct policy-learning algorithm used in the experiments, and \Cref{sec:numerical_examples} presents the numerical examples.

\medskip \noindent\textbf{Notation.}
Given two measurable spaces $(X,\mathcal{X})$ and $(Y,\mathcal{Y})$, we denote by $\Mcal(X;Y)$ the set of measurable functions from $X$ to $Y$. We adopt $\lvert\cdot\rvert$ for both the Euclidean norm and the Frobenius norm. Finally, for any $R\geq0$, $B_R:=B(0,R)$ denotes the (Euclidean) closed ball of radius $R$.
If, in addition, $(X,|\cdot|_X), (Y,|\cdot|_Y)$ are normed spaces and $\varphi: X \to Y$, we denote by
\begin{align*}
    [\varphi] := \sup_{x\neq\pr{x}}\frac{|\varphi(x)-\varphi(\pr{x})|_Y}{|x-\pr{x}|_X}    
\end{align*}
its Lipschitz constant, whenever finite.

\section{Continuous-time setting}
\label{sec:continuous_time_setting}

Let $(\Omega, \Fcal, \PP)$ be a complete probability space endowed with a filtration $\F=(\Fcal_t)_{t\geq0}$ satisfying the usual conditions ($\PP$-completeness and right-continuity) and let $w=(w_t)_{t\geq0}$ be an $m$-dimensional standard Brownian motion, adapted to $\F$. Finally, denote by $\E$ the expectation under $(\Omega, \Fcal, \PP)$.

Fix $T>0$. We consider a finite-horizon stochastic control problem. Let $A \subseteq \R^{k}$ and denote by $\Acal$ the set of admissible controls that are assumed to be $A$-valued $\F$-progressively measurable processes $\alpha: [0,T]\times \Omega \to A$. The evolution of a controlled $d$-dimensional state is governed by a stochastic differential equation, with measurable coefficients
$b : [0,T] \times \R^{d} \times A \to \R^{d}$
and
$\sigma : [0,T] \times \R^{d} \times A \to\R^{d \times m}$. 
For an initial condition $(t,x)\in[0,T]\times\R^d$, we consider the state equation given by:
\begin{align}\label{eq:sde} 
\begin{cases}
    \de y_s = b(s, y_s, \alpha_s)\, \de s + \sigma(s, y_s, \alpha_s)\, \de w_s,\quad s\in(t,T]\\
    y_t = x.
\end{cases}
\end{align}
We denote by $y^{t,x,\alpha}$ the corresponding state process starting from $(t,x)$ and driven by the control $\alpha\in\Acal$. Note that when the initial $t$ time is strictly positive, only the restriction of $\alpha$ to $[t,T]$ is relevant.

\begin{assumption}\label{ass:ex_uniq_SDE}
The following conditions hold.
    \begin{enumerate}[label=(\roman*)]
        \item $A$ is a non-empty compact and convex subset of $\R^k$.
        \item For any $t,s \in [0,T]$, $x,y\in\R^d$, and $a,\pr a \in A$,
        \begin{align*}
            \lvert b(t,x,a) - b(s,y,\pr a) \rvert
            \leq [b]_0 \lvert t-s \rvert^{1/2} + [b]_1 \lvert x-y \rvert + [b]_2 \lvert a-\pr a \rvert,
        \end{align*}
        for some non-negative constants $[b]_0, [b]_1, [b]_2$;
        
        \item For any $t,s \in [0,T]$, $x,y\in\R^d$, and $a,\pr a \in A$,
        \begin{align*}
            \lvert \sigma(t,x,a) - \sigma(s,y,\pr a) \rvert
            \leq [\sigma]_0 \lvert t-s \rvert^{1/2} + [\sigma]_1 \lvert x-y \rvert + [\sigma]_2 \lvert a-\pr a \rvert,
        \end{align*}
        for some non-negative constants
        $[\sigma]_0, [\sigma]_1, [\sigma]_2$.
    \end{enumerate}
\end{assumption}

Notice that, since the horizon is finite and $A$ is compact, the coefficients have linear growth uniformly in $(s,a)$. More precisely, there exists a positive constant $C$ such that
\begin{align*}
    |b(s,x,a)| + |\sigma(s,x,a)| \leq C(1+|x|), \qquad (s,x,a) \in [0,T] \times \R^d \times A.
\end{align*}
Consequently, for every $(t,x)\in[0,T]\times\mathbb R^d$ and every $\alpha \in \Acal$, equation~\eqref{eq:sde} admits a unique strong solution, as a standard consequence of the classical well-posedness theory for SDEs; see, for instance, \cite[Ch.s~1-2]{yong1999stochastic}.

\begin{remark}
    The Lipschitz continuity of $b$ and $\sigma$ with respect to the control variable is not needed for the existence and uniqueness of the solution to \eqref{eq:sde}. For that purpose, it is enough to assume uniform Lipschitz continuity in the state variable and a uniform linear growth condition. We impose Lipschitz continuity with respect to the control variable because it will later be used to derive quantitative error estimates.
\end{remark}

To introduce the control problem, let $f : [0,T] \times \R^{d} \times A \to \R$ and $g :\R^{d} \to \R$ be measurable functions satisfying the following assumptions.

\begin{assumption}\label{ass:cost}~\\ \vspace{-1em}
    \begin{enumerate}[label=(\roman*)]
        \item 
        There exist non-negative constants $[f]_0, [f]_1, [f]_2$ such that, for all $t,s\in[0,T]$, $x,y\in\mathbb R^d$, and $a,\pr a\in A$, 
        \begin{align*}
            \lvert f(t,x,a) - f(s,y,\pr a) \rvert &\leq [f]_0 \lvert t-s \rvert^{1/2} + [f]_1 \lvert x-y \rvert + [f]_2 \lvert a-\pr a \rvert.
        \end{align*}
        \item
        There exists a non-negative constant $[g]$ such that, for all $x,y\in \R^d$,
        \begin{align*}
            \lvert g(x) - g(y) \rvert &\leq [g] \lvert x-y \rvert.
        \end{align*}
    \end{enumerate}
\end{assumption}

For any $(t,x)\in[0,T]\times\R^d$ and any admissible control $\alpha\in\mathcal{A}$ we define the cost functional
\begin{align*}
    J(t,x,\alpha) := \E \bigg[ \int^T_t f(s, y^{t,x,\alpha}_s, \alpha_s)\, \de s + g(y^{t,x,\alpha}_T) \bigg],
\end{align*}
and the corresponding value function by
\begin{align*}
    v(t,x) := \inf_{\alpha\in\Acal}J(t,x,\alpha).
\end{align*}

We recall the following regularity estimate for the value function. It follows from the continuity property of the coefficients and standard estimates for stochastic differential equations (cf. \cite[Ch.~1, Thm. 6.16 and Ch.~4, Prop. 3.1]{yong1999stochastic}).

\begin{proposition}\label{prop:cont_value_function}
    Under \Cref{ass:ex_uniq_SDE,ass:cost}, the value function $v$ is well-defined on $[0,T]\times \R^d$ and satisfies the bound
    \begin{align*}
        |v(t,x)-v(s,y)|\leq C \big( |x-y| + (1+|x|+|y|)|t-s|^{1/2} \big), \hspace{10pt}\forall\ t,s\in[0,T],\ x,y \in \R^d,
    \end{align*}
    for some constant $C>0$, depending only on $T$, and the constants appearing in the above assumptions.
\end{proposition}

For simplicity of exposition, we assume $t=0$ from now on. Thus, the control problem is posed on the whole interval $[0,T]$; in this case, the process $(y_s)_{s\in[0,T]}$ is denoted by $y_s^{x,\alpha}:=y_s^{0,x,\alpha}$.

\section{Discrete-time setting and feedback control formulation}
\label{sec:discrete_time_setting}

\subsection{Discrete-time setting}
\label{sub:discrete_time_setting}

We approach the discretization of the control problem in two steps: given a time grid, we first restrict the admissible controls to piecewise-constant policies, and then discretize the dynamics.
Let $N \in \N \backslash \{0\}$ and define the time-step $\tau := T/N >0$, so that the resulting time mesh is given by $t_k := k\tau$ for $k=0,\dots,N$. We define
\begin{align*}
    \Acal^\tau := \bigg\{ \alpha\in\Acal \ &\bigg| \ \exists \, (a_0,\dots,a_{N-1})\in\Mcal(\Omega;A)^N, \ a_k \text{ are } \Fcal_{t_k}\text{-measurable},\\
    &\qquad \text{such that} \ \alpha_s(\omega) = \sum_{k=0}^{N-1} a_k(\omega) \mathbf{1}_{s\in[t_k, t_{k+1})} \bigg\}
\end{align*}
(with the convention that $\alpha_T(\omega) = a_{N-1}(\omega)$).
Thus, any control $\alpha\in\Acal^\tau$ can be written as
\begin{align*}
    \alpha(\omega) = (a_0(\omega),\dots,a_{N-1}(\omega)), \quad \text{where each $a_k$ 
    is $A$-valued, $\Fcal_{t_k}$-measurable.}
\end{align*}

We first define the value $v^\tau(t,x)$ by considering piecewise-constant policies, as follows:
\begin{align*}
    v^\tau(0,x) := \inf_{\alpha\in\Acal^\tau}J(0,x,\alpha) = \inf_{\alpha\in\Acal^\tau} \E \bigg[ \int^T_0 f(s, y^{x,\alpha}_s, \alpha_s)\, \de s + g(y^{x,\alpha}_T) \bigg].
\end{align*}
It is clear that $v(0,x) \leq v^\tau(0,x)$. The reverse bound quantifies the
approximation of general controls by piecewise-constant policies: this question
was first addressed by Krylov \mbox{\cite{krylov_piecewise_1999}}, who obtained, for
bounded coefficients, an error of order $\tau^{1/6}$, subsequently improved to
$C\tau^{1/4}$ in \cite{jakobsen_improved_2019}. In our setting, with possibly
unbounded (Lipschitz) coefficients, the arguments in \mbox{\cite{ass-bok-zid-2015}}
yield $v^\tau(0,x) \leq v(0,x) + C(1+|x|)^{q}\,\tau^{1/4}$ for a positive
constant $C$ and some $q>0$.

We then consider the discretization of the dynamics.
For any given $\alpha\in\Acal^\tau$, we discretize the dynamics \eqref{eq:sde} of $y^{x,\alpha}$ by means of the Euler--Maruyama recursive relation
\begin{align*}
\begin{cases}
    X^{x,\alpha}_{t_{k+1}} &= X^{x,\alpha}_{t_k} + b(t_k,X^{x,\alpha}_{t_k},a_k)\tau
      + \sigma(t_k,X^{x,\alpha}_{t_k},a_k)\Delta W_{k+1},\quad k=0,\dots,N-1,\\
    X^{x,\alpha}_{t_0} &= x,
\end{cases}
\end{align*}
where $\Delta W_{k+1}:= w_{t_{k+1}}-w_{t_k}$ are i.i.d. random variables with distribution $\mathcal{N}(0,\tau\,\mathrm{I}_m)$. We then obtain a second approximation of the value function, hereafter denoted by $V_0$:
\begin{align}\label{eq:def_V0}
    V_0(x) := \inf_{\alpha\in\Acal^\tau}\tilde{J}(0,x,\alpha)
\end{align}
where
\begin{align}\label{eq:def_J_tilde}
    \tilde{J}(0,x,\alpha) := \E \bigg[ \sum_{k=0}^{N-1}f(t_k, X^{x,\alpha}_{t_k}, a_k)\tau + g(X^{x,\alpha}_{t_N}) \bigg].
\end{align}
We also have that there exists a positive constant $C$ such that $\lvert v^\tau(0,x) - V_0(x) \rvert \leq C(1 + \lvert x \rvert)\tau^{1/2}$ 
(cf. \cite[Prop.~3.1]{picarelli_probabilistic_2020}; see also \cite[Thm.~10.2.2]{kloeden_numerical_1992} for the classical strong order $1/2$ convergence of the Euler--Maruyama scheme).
Therefore, from the triangular inequality, we conclude that there exists a positive constant $C$ such that, with $\pr q := \max\{q,1\}$,
\begin{align*}
    \lvert v(0,x) - V_0(x) \rvert &\leq C(1+|x|)^{\pr q}\tau^{1/4}.
\end{align*}

In the next sections, we deal with the approximation of the value $V_0$. Indeed, the focus of this work is the approximation of $V_0$ by deep neural networks: roughly speaking, denoting the approximation by $\hat{V}_0$, we are interested in providing explicit error estimates, in a weak sense, between $\hat{V}_0$ and $V_0$, allowing for possibly degenerate diffusions.

\subsection{Feedback control formulation and algorithm} 

We first recall that the discrete-time control problem associated with $V_0$ admits an equivalent Markov feedback formulation. We denote by $\Gcal$ the set of measurable functions from $\R^d$ to $A$, i.e., 
\begin{align*}
   \Gcal := \Mcal(\R^d;A).
\end{align*}

For a given feedback control policy $\alpha\in\Gcal^N$, that is, $\ma=(\ma_0,\dots,\ma_{N-1})$, where each  $\alpha_k$ belongs to $\Gcal$,
and given $x\in \R^d$,
we consider the process $\widetilde{X}_{t_k}^{x,\alpha}$ expressed by the recursive relation, for $k=0,\dots,N-1$,
\begin{align}\label{eq:EM_scheme}
\begin{cases}
    \wX^{x,\alpha}_{t_{k+1}} &= 
      \wX^{x,\alpha}_{t_k} + b(t_k,\widetilde{X}^{x,\alpha}_{t_k},\alpha_k(\widetilde{X}^{x,\alpha}_{t_k}))\tau
      + \sigma(t_k,\widetilde{X}^{x,\alpha}_{t_k},\alpha_k(\widetilde{X}^{x,\alpha}_{t_k}))\sqrt{\tau}Z_{k+1},\\
    \widetilde{X}^{x,\alpha}_{t_0} &= x,
\end{cases}
\end{align}
where $Z_{k+1}$ are i.i.d. standard normal random vectors in $\R^m$. By the finite-horizon dynamic programming principle (DPP), the value $V_0$ can be equivalently written by restricting the minimization to nonrandomized Markov feedback policies (i.e., $\alpha\in\Gcal^N$), then
\begin{align*}
    V_0(x) = \inf_{\alpha\in\mathcal G^N} \mathbb E \bigg[ \sum_{k=0}^{N-1} f(t_k,\widetilde X^{x,\alpha}_{t_k}, \alpha_k(\widetilde X^{x,\alpha}_{t_k}))\tau + g(\widetilde X^{x,\alpha}_{t_N}) \bigg].
\end{align*}
Moreover, under the compactness of $A$ and the above continuity assumptions on the coefficients, an optimal feedback policy exists (see \cite[Prop.~8.1, 8.2, 8.5 and Cor.~8.1.1]{bertsekas_shreve}).

\medskip
We recall the following commutation formula, also known as a ``randomization formula''. Let $(\Omega_0,\Fcal_0,\PP_0)$ be an additional probability space, 
under which the expectation is indicated by $\E_0$, and let
$\xi \in L^2((\Omega_0,\Fcal_0,\PP_0);\R^d)$,
that is $\E_0[\lvert \xi \rvert^2]<\infty$.
We consider $\widetilde X^{\xi,\alpha}$ on the product probability space $(\Omega_0\times\Omega, \Fcal_0\otimes\Fcal, \PP_0\otimes\PP)$, with initial condition $\widetilde X^{\xi,\alpha}_{t_0}=\xi$.
\begin{proposition}[{\bf Randomization formula}]\label{prop:comm_formula}
Let \Cref{ass:ex_uniq_SDE,ass:cost} hold true, and let $\xi \in L^2((\Omega_0,\Fcal_0,\PP_0);\R^d)$. Then,
\begin{enumerate}
    \item The following commutation formula holds:
    \begin{equation}\label{eq:comm_formula}
    \begin{aligned}
        \E_0[V_0(\xi)] 
            &= \inf_{\alpha\in\cG^N} \E_0[ \tilde{J}(0,\xi,\alpha)]\\
            &= \inf_{\alpha\in\cG^N} \E_0 \bigg[ \E \bigg[
    	   \sum_{k=0}^{N-1}f(t_k, \wX^{\xi,\alpha}_{t_k}, \alpha_k(\wX^{\xi,\alpha}_{t_k}))\tau +  g(\wX^{\xi,\alpha}_{t_N}) \bigg] \bigg].
    \end{aligned}
    \end{equation}
    \item
    There exists an optimal feedback policy
    $\bar\alpha\in\Gcal^N$ minimizing \eqref{eq:comm_formula}. Moreover, for any $\bar\alpha\in\Gcal^N$, the following two are equivalent:
    \begin{enumerate}
        \item $\bar{\alpha} \in \operatorname*{arginf}_{\alpha\in\Gcal^N} \E_0[\tilde{J}(0,\xi,\alpha)]$,
        \item $\tilde{J}(0,x,\bar\alpha) = V_0(x)$ (i.e., $\bar{\alpha} \in \operatorname*{arginf}_{\alpha\in\Gcal^N} \tilde{J}(0,x,\alpha)$), for $\mathcal{L}(\xi)$-a.e. $x \in \R^d$, where $\mathcal{L}(\xi)$ denotes the law of $\xi$.
    \end{enumerate}
\end{enumerate}
\end{proposition}
\begin{proof}
    By the DPP and a measurable selection argument, using the continuity of the
    coefficients and of the costs with respect to the control variable, together with
    the compactness of $A$ (see \cite{bertsekas_shreve}), there exists an optimal feedback policy $\alpha^*\in\Gcal^N$ such that
    $V_0(x)=\tilde J(0,x,\alpha^*)$, for every $x\in\R^d$. By the definition of $V_0$ as the infimum over piecewise-constant policies, we get $V_0(x) \leq \tilde J(0,x,\alpha)$, for $\alpha\in\Gcal^N$, and integrating with respect to $\mathcal{L}(\xi)$ gives $\E_0[V_0(\xi)] \leq
        \inf_{\alpha\in\Gcal^N}
        \E_0[\tilde J(0,\xi,\alpha)]$.
    The reverse inequality follows by choosing $\alpha=\alpha^*$.

    For the characterization of minimizers, note that $\tilde J(0,x,\alpha)-V_0(x) \geq 0$, thus $\E_0[\tilde J(0,\xi,\bar\alpha)-V_0(\xi)]=0$ if and only if $\tilde J(0,x,\bar\alpha)=V_0(x)$ for $\mathcal{L}(\xi)$-a.e. $x \in \R^d$.
\end{proof}

In this work, we are interested in approximating the value $\E_0[V_0(\xi)]$ and a corresponding optimal feedback control strategy.
Given an approximation space $\hat\cG$ for $\Gcal$ (with $\hat\cG\subseteq\cG$), such as a neural network class, we consider the following approximation procedure.

\begin{algo}\label{algo:general}
Let $\eta>0$ be a given margin of error.
\begin{enumerate}[label=(\roman*)]
    \item Compute an $\eta$-suboptimal minimizer $\hat\ma \in \hat\cG^N$ of the above expectation formula, in the sense that
    \begin{align}\label{eq:algo_suboptimal_policy}
        \E_0[\tilde{J}(0,\xi,\hat\ma)] \leq \inf_{\alpha\in\hat\Gcal^N} \E_0[\tilde{J}(0,\xi,\alpha)] + \eta.
    \end{align}
    \item Define the approximate value induced by $\hat\ma$ as
    \begin{align*}
        \hat V_0(x):= \tilde{J}(0,x,\hat\ma).
    \end{align*}
\end{enumerate}
\end{algo}

\begin{remark}
    The approximate control $\hat\alpha$ is not necessarily unique.
\end{remark} 

\begin{remark}
    Later on, we will choose $\hat\alpha\in\hat\Gcal_L^N$, where $\hat\Gcal_L$ is a class of neural networks with Lipschitz constant bounded by $L$.
    Step $(i)$ can be performed, up to the tolerance $\eta$, by using a stochastic gradient descent algorithm (SGD).
\end{remark}

This procedure reduces the approximation of $\E_0[V_0(\xi)]$ to an optimization problem over the restricted class $\hat\Gcal^N$. The goal of the next section is to estimate the error between the induced approximation $\hat V_0$ and the true value $V_0$ in some averaged norm, in terms of the expressive power of $\hat\Gcal$ (that is, how close $\hat \cG$ is to the true space $\cG$), and the optimization tolerance $\eta$.

\section{Convergence analysis}
\label{sec:convergence_analysis}

Let $L\geq0$ be a given constant and denote by $\Gcal_L \subseteq \Gcal$ the set of measurable functions $\alpha: \R^d \to A$ whose Lipschitz constant is bounded above by $L$, that is,
\begin{align*}
    \Gcal_L := \{ \alpha\in\Gcal \, : \, [\alpha] \leq L \}.
\end{align*}

The first result is a discrete Gronwall-type estimate for the second moment of the difference between two controlled trajectories.
\begin{lemma} [{\bf Gronwall estimate}] \label{lemma:2}
    Under \cref{ass:ex_uniq_SDE}, given $\bar{\alpha}\in\Gcal^N$, $\alpha\in\Gcal_L^N$, and $x,y\in\R^d$, it holds
    \begin{align}\label{eq:L1-1}
        \max_{0 \leq k \leq N}
        \E\Big[ \big|{\widetilde{X}^{y,\alpha}_{t_k}} - {\widetilde{X}^{x,\bar\alpha}_{t_k}}\big|^2\Big]
        \leq e^{C_{1,L}T}\bigg(|y-x|^2 + C_{2,L} \tau \sum_{k=0}^{N-1} \E \Big[ \big| \alpha_k(\widetilde{X}^{x,\bar\alpha}_{t_k}) - \bar\alpha_k(\widetilde{X}^{x,\bar\alpha}_{t_k}) \big|^2 \Big] \bigg),
    \end{align}
    with constants $C_{1,L},C_{2,L} \geq 0$ which depend only on $T$, $L$ and the data.
\end{lemma}
\begin{proof}
    For $0 \leq k \leq N$, let ${y_k := \widetilde{X}^{y,\alpha}_{t_k}}$ and ${x_k := \widetilde{X}^{x,\bar\alpha}_{t_k}}$, which satisfy the recursive equations
    \begin{align*}
        y_{k+1} &= y_k + b_k(y_k,\alpha_k(y_k))\tau + \sigma_k(y_k,\alpha_k(y_k))\sqrt{\tau}Z_{k+1},\\
        x_{k+1} &= x_k + b_k(x_k,\bar\alpha_k(x_k))\tau + \sigma_k(x_k,\bar\alpha_k(x_k))\sqrt{\tau}Z_{k+1},
    \end{align*}
    where $b_k(x,a):=b(t_k,x,a)$, $\sigma_k(x,a):=\sigma(t_k,x,a)$, and $\{Z_i\}_{i=1}^{N}$ are i.i.d. standard normal random vectors in $\R^m$.
    We first write 
    \begin{align*}
        y_{k+1}-x_{k+1} &= y_k-x_k +  \big(b_k(y_k,\alpha_k(y_k)) -  b_k(x_k,\bar\alpha_k(x_k))\big)\tau\\
        &\quad + \big(\sigma_k(y_k,\alpha_k(y_k)) - \sigma_k(x_k,\bar\alpha_k(x_k))\big)\sqrt{\tau}Z_{k+1}.
    \end{align*}
    Taking the conditional expectation with respect to $\Fcal_{t_k}$ and using $\E[Z_{k+1}]=0$ 
    and $\E[|\Pi Z_{k+1}|^2\mid\mathcal F_{t_k}]=|\Pi|^2$ 
    for every $\mathcal F_{t_k}$-measurable matrix $\Pi$, where $|\cdot|$ denotes Frobenius norm on matrices, we deduce
    \begin{multline}\label{eq:L1-2}
    \E[|y_{k+1}-x_{k+1}|^2 \mid \Fcal_{t_k}]
    = |y_k-x_k + (b_k(y_k,\alpha_k(y_k))-b_k(x_k,\bar\alpha_k(x_k)))\tau|^2\\
    + \tau|\sigma_k(y_k,\alpha_k(y_k))-\sigma_k(x_k,\bar\alpha_k(x_k))|^2
    \end{multline}
    Let $C_b:=[b]_1 + L[b]_2$ and $C_\ms:=[\sigma]_1 + L[\sigma]_2$.
    The last term of \cref{eq:L1-2} can be bounded as follows:
    \begin{align*}
        |\sigma_k(y_k,\alpha_k(y_k))-\sigma_k(x_k,\bar\alpha_k(x_k))| \\
	&\hspace{-3cm} \leq |\sigma_k(y_k,\alpha_k(y_k)) - \sigma_k(x_k,\alpha_k(x_k))|
               + |\sigma_k(x_k,\alpha_k(x_k)) - \sigma_k(x_k,\bar\alpha_k(x_k))|\\
	&\hspace{-3cm} \leq C_\ms |y_k-x_k| + [\sigma]_2|\alpha_k(x_k)-\bar\alpha_k(x_k)|.
    \end{align*}
    Hence
    \begin{align}\label{eq:L3-1}
        |\sigma_k(y_k,\alpha_k(y_k))-\sigma_k(x_k,\bar\alpha_k(x_k))|^2
          & \leq 2 C_\ms^2|y_k-x_k|^2 + 2[\sigma]_2^2|\alpha_k(x_k)-\bar\alpha_k(x_k)|^2.
    \end{align}
    Similarly,
    \begin{align*}\label{eq:L3-2-0}
        |y_k-x_k + (b_k(y_k,\alpha_k(y_k))-b_k(x_k,\bar\alpha_k(x_k)))\tau| 
          &\leq (1 + C_b \tau) |y_k-x_k| + [b]_2\tau|\alpha_k(x_k)-\bar\alpha_k(x_k)|,
    \end{align*}
    and we get
    \begin{equation*}
    \begin{aligned}
        &\hspace{1cm}|y_k-x_k + (b_k(y_k,\alpha_k(y_k))-b_k(x_k,\bar\alpha_k(x_k)))\tau|^2 \\
        &\leq (1 + C_b \tau)^2|y_k - x_k|^2 + [b]_2^2 \tau^2 |\alpha_k(x_k)-\bar\alpha_k(x_k)|^2 + [b]_2(1 + C_b\tau)\tau(|y_k-x_k|^2 + |\alpha_k(x_k)-\bar\alpha_k(x_k)|^2)\\
        &\leq \big( (1 + C_b \tau)^2 + [b]_2(1+C_b\tau)\tau \big)|y_k - x_k|^2 + \big( [b]_2^2\tau^2 + [b]_2(1+C_b\tau)\tau \big)|\alpha_k(x_k)-\bar\alpha_k(x_k)|^2\\
        &\leq \big( 1 + \tau(TC_b^2 + 2C_b + [b]_2(1 + C_bT)) \big)|y_k - x_k|^2 + \tau\big( [b]_2^2T + [b]_2(1 + C_bT) \big)|\alpha_k(x_k)-\bar\alpha_k(x_k)|^2,
    \end{aligned}
    \end{equation*}
    where, in the last inequality, we used the fact that $\tau \leq T$. We then have
    \begin{multline}\label{eq:L3-2-1}
    |y_k-x_k + (b_k(y_k,\alpha_k(y_k))-b_k(x_k,\bar\alpha_k(x_k)))\tau|^2\\
    \leq (1 + B_{1,L}\tau)|y_k - x_k|^2 + \tau B_{2,L}|\alpha_k(x_k)-\bar\alpha_k(x_k)|^2
    \end{multline}
    with constants ${B}_{1,L} := TC_b^2 + 2C_b + [b]_2(1 + C_bT)$ and ${B}_{2,L} := [b]_2^2T + [b]_2(1 + C_bT)$.
    
    Combining estimates \eqref{eq:L3-1} and \eqref{eq:L3-2-1} 
    into \cref{eq:L1-2}
    we obtain:
    \begin{align*}
        \E[|y_{k+1}-x_{k+1}|^2 \,|\, \Fcal_{t_k}] &\leq |y_k-x_k|^2 (1 + C_{1,L}\tau) + {C_{2,L}}\tau |\alpha_k(x_k)-\bar\alpha_k(x_k)|^2,
    \end{align*}
    with $C_{1,L} := {B}_{1,L} + 2C^2_\ms$ and 
    $C_{2,L} := {B}_{2,L} + 2[\sigma]_2^2$.
    By induction and the tower property of conditional expectation, for $0 < n \leq N$, we get 
    \begin{align*}
        \E\Big[\big|{\widetilde{X}^{y,\alpha}_{t_n}} - {\widetilde{X}^{x,\bar\alpha}_{t_n}}\big|^2\Big] &\leq (1+C_{1,L}\tau)^n \bigg( |y-x|^2
          + C_{2,L}\tau \sum_{k=0}^{n-1} \E\Big[|\alpha_k({\widetilde{X}^{x,\bar\alpha}_{t_k}})-\bar\alpha_k({\widetilde{X}^{x,\bar\alpha}_{t_k}})|^2\Big] \bigg).
    \end{align*}
    By noting that $(1+C_{1,L}\tau)^n \leq e^{C_{1,L}T}$, for $t_n = n\tau \leq T$, we obtain the desired result.
\end{proof}

\begin{remark}
    The estimate is asymmetric with respect to the two controls: only $\alpha$ is required to be Lipschitz, while the reference control $\bar\alpha$ is merely measurable. This comes from decomposing the difference of the controls along the two trajectories as
    \begin{align*}
        \alpha_k(y_k)-\bar\alpha_k(x_k)
        =
        \big(\alpha_k(y_k)-\alpha_k(x_k)\big)
        +
        \big(\alpha_k(x_k)-\bar\alpha_k(x_k)\big),
    \end{align*}
    so that only the first term is controlled by the Lipschitz constant of $\alpha_k$, whereas the second one is the approximation error along the trajectory driven by the reference control $\bar\alpha$.
\end{remark}

We next derive an estimate for the sum term appearing on the right-hand side of \cref{eq:L1-1}. We first have the following bound.

\begin{lemma}\label{lemma:4}
    Let \cref{ass:ex_uniq_SDE} hold and
    $x\in\R^d$.
    Let $\bar\alpha\in\Gcal^N$ and $\alpha\in\Gcal_L^N$. For every Borel set $D\subseteq\R^d$, we have
    \begin{align}\label{eq:P1}
        \max_{0 \leq k \leq N} \E\Big[ \big|{\widetilde{X}^{x,\alpha}_{t_k}} - {\widetilde{X}^{x,\bar\alpha}_{t_k}}\big|^2\Big] \leq C_{T,L} \bigg(
        \tau \sum_{k=0}^{N-1} \| \ma_k-\bar\ma_k\|^2_{\infty,D}
        + T\diam(A)^2 \max_{0\leq k\leq N-1} \PP[\widetilde{X}^{x,\bar\alpha}_{t_k} \notin D] \bigg),
    \end{align}
    where $C_{T,L}$ is a positive constant depending only on $T$, $L$ and the data, and $\diam(A)$ denotes the diameter of $A$, which is finite since $A$ is assumed compact. Moreover, by $\|\cdot\|_{\infty,D}$ we denote $\|\phi\|_{\infty,D}:=\sup_{z\in D}|\phi(z)|$.
\end{lemma}
\begin{proof}
    By \Cref{lemma:2}, we have
    \begin{align*}
        \max_{0 \leq k \leq N}
        \E\Big[ \big|{\widetilde{X}^{x,\alpha}_{t_k}} - {\widetilde{X}^{x,\bar\alpha}_{t_k}}\big|^2\Big]
        \leq C_{2,L} e^{C_{1,L}T}\ \tau \sum_{k=0}^{N-1}
          \E \Big[ \big| \alpha_k(\widetilde{X}^{x,\bar\alpha}_{t_k}) 
                        - \bar\alpha_k(\widetilde{X}^{x,\bar\alpha}_{t_k}) \big|^2 \Big].
    \end{align*} 
    Let $\beta_k := \alpha_k-\bar\alpha_k$, so that $|\beta_k(x)|\leq \diam(A)$ (because both $\ma_k$ and $\bar\ma_k$ take their
    values in $A$).
    By separating the regions $\{\widetilde{X}_{t_k}^{x,\bar\ma}\in D\}$ and $\{\widetilde{X}_{t_k}^{x,\bar\ma}\notin D\}$, we have
    \begin{align*}
        |\beta_k(\widetilde{X}_{t_k}^{x,\bar\ma})|^2 &\leq \| \beta_k \|^2_{\infty,D} \cdot \mathbf{1}_{\{ \widetilde{X}_{t_k}^{x,\bar\ma} \in D \}} + \diam(A)^2 \cdot \mathbf{1}_{\{ \widetilde{X}_{t_k}^{x,\bar\ma} \notin D \}}\\
        &\leq \| \beta_k \|^2_{\infty,D} + \diam(A)^2 \cdot \mathbf{1}_{\{ \widetilde{X}_{t_k}^{x,\bar\ma} \notin D \}}.
    \end{align*}
    Thus, taking the expectation
    \begin{align} \label{eq:lemma5.5proof}
        \E \big[ |\beta_k(\widetilde{X}_{t_k}^{x,\bar\ma})|^2 \big] &\leq \| \beta_k \|^2_{\infty,D} + \diam(A)^2 \cdot \PP \big[ \widetilde{X}_{t_k}^{x,\bar\ma} \notin D \big].
    \end{align}
    This proves the claim, with $C_{T,L}:= C_{2,L} e^{C_{1,L} T}$.
\end{proof}

We recall the following classical estimate, in the discrete-time case.
\begin{lemma}\label{lemma:3}
    Let \cref{ass:ex_uniq_SDE} hold true and $x\in\R^d$. For every $\alpha\in\Gcal^N$, and for any $p\geq2$,
    \begin{align*}
        \E\bigg[ \max_{0 \leq k \leq N} \big|\widetilde{X}_{t_k}^{x,\alpha} - x \big|^p \bigg]
          \leq  C_{T,p}(1+|x|^p),
    \end{align*}
    where $C_{T,p}$ is a constant that depends only on $T$, $p$, and on the data.
\end{lemma}

\begin{proof}
     The proof follows from the regularity assumptions on $b$ and $\sigma$, the discrete-time Burkholder-Davis-Gundy inequality, and a discrete Gronwall argument. See, for instance, \cite[Thm. 4.5.4]{kloeden_numerical_1992}, or \cite[Ch.~1, Thm. 6.16]{yong1999stochastic} in the continuous-time setting. For completeness, we report the proof in \Cref{appendix:subsec-1}.
\end{proof}

We then obtain the following intermediate estimate.
\begin{lemma}\label{lemma:5}
    Let \cref{ass:ex_uniq_SDE} hold. Let $R>0$ and let $\xi\in L^p((\Omega_0,\mathcal F_0,\mathbb P_0);\mathbb R^d)$ for some $p\geq2$ be such that $|\xi| \leq R$, $\mathbb P_0$-a.s.
    Let $M>R$ and set $B_M:=B(0,M)$. 
    Then, for every $\alpha\in\Gcal^N$, there exists a constant $C_{T,p}\geq0$, depending only on $T$, $p$, and the data, such that
    \begin{align*}
        \E_0 \bigg[ \max_{0 \leq k \leq N} \PP  \Big[ \widetilde{X}_{t_k}^{\xi,\alpha} \notin B_M \Big] \bigg]
          & \leq \frac{C_{T,p}}{(M-R)^p}(1+\E_0[|\xi|^p]).
    \end{align*}
\end{lemma}
\begin{proof}
    By the estimate in \Cref{lemma:3}, applied conditionally on $\xi$, for $\mathbb P_0$-a.e. realization of $\xi$,
    there exists a positive constant $C_{T,p}$ such that
    \begin{align*}
        {\E}\bigg[ \max_{0 \leq k \leq N} |\widetilde{X}_{t_k}^{\xi,\alpha} - \xi|^p \bigg] 
          & \leq C_{T,p} (1+|\xi|^p).
    \end{align*}
    Since $|\xi|\le R$ and $M>R$, the implication
    \begin{align*}
        \widetilde{X}_{t_k}^{\xi,\alpha}\notin B_M
        \quad\Longrightarrow\quad
        |\widetilde{X}_{t_k}^{\xi,\alpha}-\xi|
        \ge M-|\xi|
        \ge M-R
    \end{align*}
    holds for every $k$. Hence,
    \begin{align*}
        \max_{0\le k\le N} \PP\Big[
        \widetilde{X}_{t_k}^{\xi,\alpha}\notin B_M
        \Big]
        &\leq
        \PP\bigg[ \max_{0\le k\le N}
        |\widetilde{X}_{t_k}^{\xi,\alpha}-\xi| \ge M-R \bigg] \\
        &\leq
        \frac{\E\Big[ \max_{0\le k\le N}
        |\widetilde{X}_{t_k}^{\xi,\alpha}-\xi|^p \Big]}{(M-R)^p}
        \leq
        \frac{C_{T,p}}{(M-R)^p} (1+|\xi|^p),
    \end{align*}
    where we used Markov's inequality.
    The result follows by taking the $\E_0$-expectation.
\end{proof}

We can now state our main result.

\begin{theorem} [\bf{Error estimate}] \label{thm:error_estimate}
    Assume that \cref{ass:ex_uniq_SDE,ass:cost} hold. Let $\xi \in L^2((\Omega_0,\Fcal_0,\PP_0);\R^d)$, such that $|\xi|\le R$, $\PP_0$-a.s., for some $R\ge0$.
    Let $\varepsilon>0$ and let $\bar\ma^\varepsilon \in \cG^N$ be an $\varepsilon$-optimal feedback policy for the randomized problem related to $V_0(\xi)$, namely $\E_0[\tilde{J}(0,\xi,\bar\alpha^\varepsilon)] \leq \E_0[V_0(\xi)] + \varepsilon$. Let $L\geq 0$ and let $\hat\Gcal_L\subseteq\Gcal_L$ be the approximation class.
    For any $\eta>0$, consider an $\eta$-suboptimal control policy $\hat\ma\in \hat\cG_L^N$ obtained as in \Cref{algo:general}, and define $\hat V_0(x):=\tilde J(0,x,\hat\alpha)$.
    Then there exists a positive constant $C_{T,L}$, depending only on $T$, $L$ and the data, and independent of $N$, 
    such that, for any  $M>R$,
    \begin{align*}
        \E_0\big[|\hat{V}_0(\xi) - V_0(\xi)|\big] & \\
         & \hspace{-2cm}
        \leq  C_{T,L} \left[
          \left( \tau \sum_{0 \leq k \leq N-1} d^2_{\infty,B_M}(\bar\alpha^\varepsilon_k, \hat\Gcal_L) \right)^{1/2} 
          + \frac{\diam(A)}{M-R} (1+\E_0[|\xi|^2])^{1/2} \right] + \varepsilon + \eta,
    \end{align*}
    where $d_{\infty,B_M}(\phi,\hat\Gcal_L) := \inf_{\psi\in\hat\Gcal_L} \sup_{z\in B_M} |\phi(z)-\psi(z)|$.
\end{theorem}

\begin{remark}[Error terms interpretation]
\label{rem:error_terms}
    On the right-hand side of the above error estimate:
    
    \begin{itemize}
    \item 
    The first term measures how well the reference policy $\bar\alpha^\varepsilon$ can be approximated, on the localization ball $B_M$, by $L$-Lipschitz policies in the neural network class $\hat\Gcal_L\subseteq\Gcal_L$.

    \item The estimate holds for any $\varepsilon$-optimal $\bar\alpha^\varepsilon\in\Gcal^N$, but is informative only when $\bar\alpha^\varepsilon$ can be chosen with Lipschitz components: for a merely measurable policy, such as a discontinuous bang-bang feedback, the first term may remain of order $\diam(A)$. The existence of such Lipschitz $\varepsilon$-optimal policies is the object of \Cref{prop:intermediate_result_convergence} below.

    \item 
    The parameter $\varepsilon$ measures the suboptimality of the chosen reference policy $\bar\alpha^\varepsilon$. \Cref{prop:intermediate_result_convergence} below shows that, for every $\varepsilon>0$, one can choose an $\varepsilon$-optimal reference policy in $\Gcal_L^N$, for a suitable $L$. In the convergence proof, $\varepsilon$ will be used precisely in this sense: it measures the cost of replacing measurable feedback policies by Lipschitz feedback policies. In particular, if an optimal feedback policy belongs to $\Gcal_L^N$, then one may take $\varepsilon=0$.
    
    \item
    The term proportional to $(M-R)^{-1}$ accounts for the probability that the stochastic process, starting in $B_R$, leaves the localization ball $B_M$ over the time horizon $T$.
    In deterministic settings, such as in \cite{bok-pro-war-23,bok-war-25}, this localization term can often be made zero by choosing $B_M$ large enough to contain the reachable set over the time horizon. In the stochastic case, one instead controls the probability of leaving $B_M$.

    \item
    The term $\eta$ accounts for the optimization error in the computation of the approximate policy $\hat\alpha$, (e.g., the convergence of an SGD algorithm).
    
    \end{itemize}
\end{remark}

\begin{remark}[Dependence on the Lipschitz constant]
\label{rem:lip_const_dependence}
The constants in the stability estimate have the form $C_{T,L}=C_{2,L}\exp(C_{1,L}T)$, with $C_{1,L}$ containing quadratic terms in $L$, hence $C_{T,L}$ may grow exponentially as $L$ increases.
This dependence should be kept in mind when interpreting the approximation term and the parameter $\varepsilon$. 
If an optimal, or near-optimal, feedback policy is Lipschitz with a moderate Lipschitz constant, then one may choose the reference policy with fixed $L$,
and the estimate yields a meaningful quantitative bound.
For example, in the linear--convex setting of~\cite{guo_hu_zhang_2023}, under additional structural assumptions, 
the existence of Lipschitz continuous optimal feedback controls is proved, together with stability properties. 
On the other hand, when the optimal feedback is discontinuous, for instance of bang--bang type, making $\varepsilon$
small may require Lipschitz approximations with increasingly large Lipschitz constants (cf.\ \Cref{rem:non-quantitative-L}).
In that case, the decrease of the approximation or suboptimality error must be balanced against the growth of $C_{T,L}$.
\end{remark}

\begin{remark}[Degenerate dynamics and pointwise approximation]
\label{rem:supremum_distance_degeneracy}
    The distance $d_{\infty,B_M}$ is defined through the pointwise supremum norm on $B_M$, rather than a Lebesgue essential supremum or an $L^p$ norm. This is not a technical detail: since the controlled dynamics may be degenerate or deterministic, the law of the discrete-time state process may be singular with respect to the Lebesgue measure (e.g., supported on a lower-dimensional set). Two policies that coincide Lebesgue-a.e. may then differ precisely on the support of the state process and induce different values, so that only a pointwise uniform approximation of $\bar\alpha^\varepsilon_k$ guarantees a control of the error, whatever the law of the state process.
This is also why our estimate does not rely on the existence of a transition probability (from step $\widetilde{X}_{t_n}$ to $\widetilde{X}_{t_{n+1}}$) associated with an explicit bounded density, unlike density-based estimates such as those used in~\cite{hure_deep_2021}.
The framework therefore covers possibly degenerate diffusions and includes the purely deterministic case, where the transition law is generally not associated with a density.
\end{remark}

\begin{proof}[Proof of Theorem~\ref{thm:error_estimate}]
    Since $\hat{\Gcal}_L\subseteq\Gcal$, the policy $\hat\alpha$ is admissible for the discrete feedback problem, hence $\hat V_0(\xi)-V_0(\xi)\geq0$, $\mathbb P_0$-a.s., and therefore $\E_0\big[|\hat V_0(\xi)-V_0(\xi)|\big] = \E_0\big[\hat V_0(\xi)-V_0(\xi)\big]$.
    The argument reduces this $L^1(\mathbb P_0)$ error to an $L^2(\mathbb P_0)$ estimate for differences of cost functionals, by the Cauchy--Schwarz inequality.

    Throughout the proof, we omit the superscript $\varepsilon$ and simply write
    $\bar\alpha$ for the $\varepsilon$-optimal feedback control
    $\bar\alpha^\varepsilon$. By its $\varepsilon$-suboptimality, we have
    \begin{align*}
        \E_0[\hat{V}_0(\xi) - V_0(\xi)] \leq \E_0[\hat{V}_0(\xi) - \tilde{J}(0,\xi,\bar\alpha)] + \varepsilon,
    \end{align*}
    which, together with the $\eta$-suboptimality of the control policy $\hat\ma\in \hat\cG_L^N$, leads to
    \be
        \E_0[\hat{V}_0(\xi) - V_0(\xi)] 
          & \leq &  \inf_{\ma\in \hat\cG_L^N}  \E_0[\tilde J(0,\xi, \ma) - \tilde J(0,\xi,\bar \ma)] + \varepsilon + \eta \nonumber\\
          & \leq & \inf_{\ma\in \hat\cG_L^N} (\E_0[|\tilde J(0,\xi, \ma) - \tilde J(0,\xi,\bar \ma)|^2] )^{1/2} + \varepsilon + \eta.
          \label{eq:bound3}
    \ee
    Let $\Ecal_\alpha:= |\tilde{J}(0,\xi,\alpha) - \tilde{J}(0,\xi,\bar\alpha)|^2$,
    where $\alpha \in \hat{\Gcal}_L^N$. We aim to bound $\inf_{\ma\in \hat\cG_L^N} (\E_0[\Ecal_\ma])^{1/2}$.
    By Jensen’s and Cauchy–Schwarz's inequalities:
    \beno 
        \Ecal_\alpha & \leq &
         \E\bigg[ \Big| 
         \sum_{k=0}^{N-1} \tau  
         \big( f(t_k, \widetilde{X}_{t_k}^{\xi,\alpha}, \alpha_k(\widetilde{X}_{t_k}^{\xi,\alpha})) 
             - f(t_k, \widetilde{X}_{t_k}^{\xi,\bar\alpha}, \bar\alpha_k(\widetilde{X}_{t_k}^{\xi,\bar\alpha})) \big) 
             + \big( g(\widetilde{X}_{t_N}^{\xi,\alpha}) - g(\widetilde{X}_{t_N}^{\xi,\bar\alpha}) \big) \Big|^2 \bigg] 
             \\
        &\leq & 2 \E\bigg[ N\tau^2 \sum_{k=0}^{N-1} 
          | f(t_k, \widetilde{X}_{t_k}^{\xi,\alpha}, \alpha_k(\widetilde{X}_{t_k}^{\xi,\alpha}))
          - f(t_k, \widetilde{X}_{t_k}^{\xi,\bar\alpha}, \bar\alpha_k(\widetilde{X}_{t_k}^{\xi,\bar\alpha})) |^2 \bigg] 
          + 2 \E \Big[ | g(\widetilde{X}_{t_N}^{\xi,\alpha}) - g(\widetilde{X}_{t_N}^{\xi,\bar\alpha}) |^2 \Big]
          \\
        &\leq & 2 T \E\bigg[ \tau \sum_{k=0}^{N-1} 
          | f(t_k, \widetilde{X}_{t_k}^{\xi,\alpha}, \alpha_k(\widetilde{X}_{t_k}^{\xi,\alpha}))
          - f(t_k, \widetilde{X}_{t_k}^{\xi,\bar\alpha}, \bar\alpha_k(\widetilde{X}_{t_k}^{\xi,\bar\alpha})) |^2 \bigg] 
          + 2 \E \Big[ | g(\widetilde{X}_{t_N}^{\xi,\alpha}) - g(\widetilde{X}_{t_N}^{\xi,\bar\alpha}) |^2 \Big].
    \eeno
    By considering the Lipschitz continuity assumptions of $f,g$ and the control policies, we get
    \beno
          | f(t_k, \widetilde{X}_{t_k}^{\xi,\alpha}, \alpha_k(\wX_{t_k}^{\xi,\alpha}))
          - f(t_k, \widetilde{X}_{t_k}^{\xi,\bar\alpha}, \bar\alpha_k(\wX_{t_k}^{\xi,\bar\alpha})) |
          &  \leq
             C_1 \ |\wX_{t_k}^{\xi,\alpha} - \wX_{t_k}^{\xi,\bar\alpha}|
         + [f]_2 \ |\alpha_k(\wX_{t_k}^{\xi,\bar\alpha}) - \bar\alpha_k(\wX_{t_k}^{\xi,\bar\alpha})|,
    \eeno
    where $C_1:=[f]_1 + L[f]_2$,
    and then
    \begin{align}
        \Ecal_\ma &\leq 2T \cdot \E \bigg[  \tau
        \sum_{k=0}^{N-1} \Big( 
           2 C_{1}^2\ |\wX_{t_k}^{\xi,\alpha} - \wX_{t_k}^{\xi,\bar\alpha}|^2 
         + 2[f]_2^2 \ | \alpha_k(\wX_{t_k}^{\xi,\bar\alpha}) - \bar\alpha_k(\wX_{t_k}^{\xi,\bar\alpha})|^2 
         \Big) \bigg] 
         + 2[g]^2 \E\Big[ |\widetilde{X}_{t_N}^{\xi,\alpha} - \widetilde{X}_{t_N}^{\xi,\bar\alpha}|^2 \Big]  \nonumber \\
        & \leq (4T^2 C_{1}^2 + 2[g]^2) \max_{0 \leq k \leq N} \E \Big[ |\widetilde{X}_{t_k}^{\xi,\alpha} - \widetilde{X}_{t_k}^{\xi,\bar\alpha}|^2 \Big] 
        + 4T[f]_2^2 \ \tau \sum_{k=0}^{N-1}
         \E \Big[ | \alpha_k(\widetilde{X}_{t_k}^{\xi,\bar\alpha}) - \bar\alpha_k(\widetilde{X}_{t_k}^{\xi,\bar\alpha}) |^2 \Big].
         \label{eq:bound1}
    \end{align}
    We bound the term
    $\E [ | \alpha_k(\widetilde{X}_{t_k}^{\xi,\bar\alpha}) - \bar\alpha_k(\widetilde{X}_{t_k}^{\xi,\bar\alpha}) |^2 ]$ as in 
    \eqref{eq:lemma5.5proof}
    and obtain
    \begin{align}
        \tau \sum_{k=0}^{N-1} \E
           \Big[ | \alpha_k(\widetilde{X}_{t_k}^{\xi,\bar\alpha}) - \bar\alpha_k(\widetilde{X}_{t_k}^{\xi,\bar\alpha}) |^2 \Big] 
        \leq \tau \sum_{k=0}^{N-1}  \| \alpha_k - \bar\alpha_k \|^2_{\infty,B_M} 
          + T \diam(A)^2 \max_{0 \leq k \leq N-1} \PP \big[ \widetilde{X}_{t_k}^{\xi,\bar\alpha} \notin B_M \big].
        \label{eq:bound2}
    \end{align}
    Combining \eqref{eq:bound1}, \eqref{eq:bound2} and \Cref{lemma:4} applied conditionally on $\xi$,
    for $\mathbb P_0$-a.e. realization of $\xi$, yields
    \begin{align*}
        \Ecal_\ma &\leq C_{T,L} \ \bigg(
        \tau \sum_{k=0}^{N-1} \| \ma_k-\bar\ma_k\|^2_{\infty,B_M}
        + T\diam(A)^2 \max_{0\leq k\leq N-1} \PP[\widetilde{X}^{\xi,\bar\alpha}_{t_k} \notin B_M] \bigg),
    \end{align*}
    for a constant $C_{T,L}$ depending only on $T$, $L$, and the data.
    By integrating with respect to $\PP_0$, we get from \Cref{lemma:5}
    \begin{align*}
        \E_0[\Ecal_\ma] & \leq C_{T,L} \ \bigg(
        \tau \sum_{k=0}^{N-1} \| \ma_k-\bar\ma_k\|^2_{\infty,B_M}
        + T\diam(A)^2 \frac{C_{T,2}}{(M-R)^2}(1+\E_0[|\xi|^2]) \bigg).
    \end{align*}
    Finally, in the bound \eqref{eq:bound3}, 
    taking the infimum over $\ma\in \hat\cG_L^N$ corresponds to
    \begin{align*}
        \inf_{\ma\in \hat\cG_L^N} \tau \sum_{k=0}^{N-1}\| \ma_k - \bar{\ma}_k \|^2_{\infty,B_M} = \tau\sum_{k=0}^{N-1} d^2_{\infty,B_M}(\bar{\ma}_k,\hat\Gcal_L).
    \end{align*}
    Using $\sqrt{u+v}\leq\sqrt u+\sqrt v$ for $u,v\ge0$ to bound $(\E_0[\Ecal_\alpha])^{1/2}$, we obtain the stated estimate (upon enlarging the $C_{T,L}$ constant if necessary).
\end{proof}

The previous error estimate now allows us to prove
the convergence of the approximated value $\hat V_0$ (computed with \Cref{algo:general}), towards $V_0$, in $L^1(\PP_0)$ for compactly supported initial distributions.

\begin{theorem} [\bf{Convergence}] \label{thm:convergence}
    Assume that \cref{ass:ex_uniq_SDE,ass:cost} hold. Let $\xi \in L^2((\Omega_0,\Fcal_0,\PP_0);\R^d)$, compactly supported.
    Let the time grid, equivalently $N\geq1$, be fixed.
    Let $(\hat\Gcal_{L}^{\Theta})_{\Theta\ge1}$, with $\hat\Gcal_L^\Theta\subseteq\Gcal_L$ for every $\Theta$, be a family of approximation classes such that, for every $L\ge0$, every $M\ge0$, and every $L$-Lipschitz function $\ma:B_M\rightarrow A$,
    \begin{align} \label{eq:nn-convergence-hypo}
        \lim_{\Theta \to \infty} d_{\infty,B_M}(\alpha, \hat\Gcal^\Theta_L) = 0.
    \end{align}
    Then, for every $\delta>0$, there exist $L\ge0$, $\bar\Theta\ge1$, and $\bar\eta>0$ such that, for every $\Theta \geq \bar\Theta$ and every $\eta\in(0,\bar\eta]$, the $\eta$-suboptimal approximation $\hat V_0^{\Theta,\eta}$ computed by \Cref{algo:general} over $(\hat\Gcal_{L}^\Theta)^N$ (that is, satisfying, as in \eqref{eq:algo_suboptimal_policy}, $\E_0[\hat{V}_0^{\Theta,\eta}] \leq \inf_{\alpha\in(\hat\Gcal^\Theta_L)^N} \E_0[\tilde{J}(0,\xi,\alpha)] + \eta$) satisfies
    \begin{align*}
        \E_0\big[ |\hat V_0^{\Theta,\eta}(\xi)-V_0(\xi)| \big] \leq \delta.
    \end{align*}
\end{theorem}
In other words, there exist sequences $L_j\ge0$, $\Theta_j\to\infty$, and $\eta_j\downarrow0$ such that
\begin{align*}
    \E_0\big[ |\hat V_0^{\Theta_j,\eta_j}(\xi)-V_0(\xi)| \big] \xrightarrow[j \to \infty]{} 0.
\end{align*}

In order to prove this convergence result, we first need an intermediate result
showing that the value of the discrete-time control problem can be approximated
arbitrarily well by using Lipschitz-continuous feedback controls. 

\begin{proposition}\label{prop:intermediate_result_convergence}
    Assume that \cref{ass:ex_uniq_SDE,ass:cost} hold. Let $\xi \in L^2((\Omega_0,\Fcal_0,\PP_0);\R^d)$. Then, for any $\varepsilon>0$, there exists a constant $L\geq0$ and $\bar{\alpha}\in\Gcal_L^N$, such that
    \begin{align*}
        \E_0[\tilde{J}(0,\xi,\bar{\alpha})] \leq \E_0[V_0(\xi)] + \varepsilon.
    \end{align*}
\end{proposition}

\begin{remark}
    Equivalently, for every $\varepsilon>0$ there exists $L\ge0$ such that
    \begin{align*}
        \E_0[V_0(\xi)] \geq \inf_{\alpha\in\Gcal_L^N}\E_0[\tilde{J}(0,\xi,\alpha)] - \varepsilon.
    \end{align*}
    In other words, Lipschitz feedback policies are dense at the level of the value $V_0(\xi)$, in the class of measurable feedback policies.
\end{remark}

\begin{proof}
    \newcommand{\cM}{\mathcal{M}}
    We prove the result by a density argument, presented in Appendix \Cref{lemma:density_argument_law} with respect to the laws of the discretized controlled process. In this proof, the convexity of $A\subset\R^k$ is needed to ensure that the approximation can be chosen to be $A$-valued.

    By the randomization formula (\Cref{prop:comm_formula}), there exists a measurable feedback policy $\alpha=(\alpha_0,\dots,\alpha_{N-1})\in\Gcal^N$ such that
    \begin{align}\label{eq:value_approx}
        \E_0[V_0(\xi)] 
        \leq
        \E_0[\tilde J(0,\xi,\alpha)]
        \leq
        \E_0[V_0(\xi)] + \frac{\varepsilon}{2},
    \end{align}
    (the left inequality is just a consequence of the definitions).
    We show that this policy can be approximated, in value, by a Lipschitz feedback policy.

    We now regularize the components of $\alpha$ backward in time. Set $\alpha^{(N)}:=\alpha$. Suppose that, for some $j\in\{0,\dots,N-1\}$, the feedbacks at times $j+1,\dots,N-1$ have already been replaced by Lipschitz feedbacks, while the previous ones are left unchanged. We denote the corresponding policy by
    \begin{align*}
        \alpha^{(j+1)}
        :=
        (\alpha_0,\dots,\alpha_j,\bar\alpha_{j+1},\dots,\bar\alpha_{N-1}), \qquad \text{where $\bar\alpha_{j+1},\dots,\bar\alpha_{N-1}$ are Lipschitz.}
    \end{align*}
    For a given $\bar \ma_j$ (to be defined later on), let us consider the policy obtained by replacing the $j$-th component by $\bar\alpha_j$:
    \begin{align*}
        \alpha^{(j)}
        :=
        (\alpha_0,\dots,\alpha_{j-1},\bar\alpha_j,
        \bar\alpha_{j+1},\dots,\bar\alpha_{N-1}).
    \end{align*}
    Since the policies $\alpha^{(j)}$ and $\alpha^{(j+1)}$ coincide up to time $t_j$, the corresponding discretized trajectories also coincide up to time $t_j$. The first difference between them is produced at the step from $t_j$ to $t_{j+1}$, and is controlled by \Cref{eq:controlled_diff_lip_control}.
    Since the feedback controls at times $j+1,\dots,N-1$ are Lipschitz, the discrete stability estimate for the Euler scheme propagates this error from $t_{j+1}$ to $T$. Consequently, using also the Lipschitz continuity of $f$ and $g$, there exists a finite constant $C_j\geq0$, depending only on the data, on $T$, and on the Lipschitz constants of $\bar\alpha_{j+1},\dots,\bar\alpha_{N-1}$ (and independent neither of $\bar \ma _j$ nor of $\ma_0,\dots,\ma_{j-1}$)
    such that
    \begin{align}
        \label{eq:intermed-0}
        \E_0[|
        \tilde J(0,\xi,\alpha^{(j)})
        -
        \tilde J(0,\xi,\alpha^{(j+1)})
        |]
        \leq
        C_j
        \Big(
        \E_0\E\Big[ \Big|
           \bar\alpha_j\big(\widetilde X_{t_j}^{\xi,\alpha^{(j+1)}}\big) - \alpha_j\big(\widetilde X_{t_j}^{\xi,\alpha^{(j+1)}}\big)
        \Big|^2 \Big] 
        \Big)^{1/2}.
    \end{align}
    Let $\mu_j = \mathcal{L}(\widetilde X_{t_j}^{\xi,\alpha^{(j+1)}}) = \mathcal{L}(\widetilde X_{t_j}^{\xi,\alpha})$ denote the law under $\PP_0\otimes\PP$.
    Applying \Cref{lemma:density_argument_law}, with
    $\mu=\mu_j$ and $\phi=\alpha_j$, we can choose a Lipschitz continuous map
    $\bar\alpha_j:\R^d\to A$ such that
    \begin{align}\label{eq:controlled_diff_lip_control}
        \Big(\int_{\R^d}
        |\bar\alpha_j(x)-\alpha_j(x)|^2\,\mu_j(\de x)
        \Big)^{1/2}
        \equiv
        \Big( \E_0\E\Big[
        \Big| \bar\alpha_j\big(\widetilde X_{t_j}^{\xi,\alpha^{(j+1)}}\big) - \alpha_j\big(\widetilde X_{t_j}^{\xi,\alpha^{(j+1)}}\big) \Big|^2
        \Big] \Big)^{1/2}
        \leq
        \frac{\eps}{2 C_j N}.
    \end{align}
    Then we have 
    \begin{align*}
        \E_0[| \tilde J(0,\xi,\alpha^{(j)}) - \tilde J(0,\xi,\alpha^{(j+1)}) |] \leq \frac{\eps}{2 N}.
    \end{align*}
    The procedure is applied backward, successively for $j=N-1,N-2,\dots,1,0$. 
    Suming up the bounds, and using the fact that $\ma^{(N)}\equiv \ma$, we obtain
    an approximate feedback control $\bar\alpha := (\bar\alpha_0,\dots,\bar\alpha_{N-1}) \in \Gcal_L^N$, 
    with $L:=\max_{0\leq k\leq N-1}[\bar\alpha_k]$, and such that
    \begin{align*}
          \E_0[| \tilde J(0,\xi,\bar\alpha) - \tilde J(0,\xi,\alpha) |] \leq \frac{\varepsilon}{2}.
    \end{align*}
    Therefore, together with \cref{eq:value_approx}, this leads to the desired estimate.
\end{proof}

\begin{remark}[Non-quantitative nature of the Lipschitz approximation]
\label{rem:non-quantitative-L}
The construction above is qualitative. It shows that, for every $\varepsilon>0$,
one may choose an $\varepsilon$-optimal feedback policy with $L$-Lipschitz
components, but it does not provide a quantitative bound on the corresponding
Lipschitz constant $L=L(\varepsilon)$, which may blow up as $\varepsilon\to0$. 
Consequently, the convergence result of \Cref{thm:convergence} does
not provide a rate in terms of the size of the approximation class; we refer to
\Cref{rem:lip_const_dependence} for structural situations in which $L$ can be
fixed a priori, allowing for quantitative bounds.
\end{remark}

\begin{proof}[Proof of Theorem~\ref{thm:convergence}]    
    Let $\delta>0$. Set $\varepsilon=\delta/4$. By \Cref{prop:intermediate_result_convergence}, there exist $L\geq0$ and an $\varepsilon$-suboptimal Lipschitz control $\bar\alpha \in \Gcal_L^N$.
    Since $\xi$ is compactly supported, there exists $R\geq0$ such that $\xi\in B_R$, $\PP_0$-a.s. Choose then $M>R$ large enough, such that
    \begin{align*}
        C_{T,L}\frac{\diam(A)}{M-R}(1+\E_0[|\xi|^2])^{1/2} \leq \frac{\delta}{4}.
    \end{align*}
    Set $\bar\eta := \delta/4$. By the approximation property of $(\hat\Gcal^\Theta_{L})_{\Theta\ge1}$, there exists $\bar\Theta\ge1$ such that, for every $\Theta\ge\bar\Theta$,
    \begin{align*}
        C_{T,L}
        \bigg(
        \tau\sum_{k=0}^{N-1}
        d^2_{\infty,B_M}
        (\bar\alpha_k,\hat\Gcal^\Theta_{L})
        \bigg)^{1/2}
        \leq
        \frac{\delta}{4}.
    \end{align*}
    Applying \Cref{thm:error_estimate}, for any $\Theta\geq\bar\Theta$, $\eta\in(0,\bar\eta]$,
    we obtain $\E_0[|\hat{V}_0^{\Theta,\eta}(\xi) - V_0(\xi)|] \leq \delta$,
    which proves the convergence result.
\end{proof}

\begin{remark}[Higher-order time discretizations]
\label{rem:higher-order-schemes}
   The convergence analysis has been presented for the Euler–Maruyama scheme to keep the notation simple, but the same approach extends to other explicit one-step discretization schemes. In particular, the discrete Gronwall estimate, the localization argument, and the policy-approximation estimate remain valid (up to different constants) for any scheme whose one-step map is stable with respect to both the state variable and the feedback control evaluations. Higher-order weak schemes could also be considered in this framework; see, e.g., \cite[Chs.~14-15]{kloeden_numerical_1992}, \cite{milstein2004stochastic}, and \cite{talay1990expansion}.
\end{remark}

\begin{remark}[Lipschitz-constrained architectures]
\label{rem:NN-lip-architecture}
    The approximation assumption \eqref{eq:nn-convergence-hypo} can be satisfied by known Lipschitz-constrained neural network architectures.
    In particular, norm-constrained GroupSort networks, introduced in \cite{anil2019sorting}, are universal approximators of Lipschitz functions on compact sets. Quantitative approximation results for Lipschitz functions by GroupSort networks are further developed in \cite{tanielian2021approximating}. Therefore, the class $\hat\Gcal_L$ appearing in the convergence theorem can be realized by existing neural-network architectures with an explicit (a priori) control of the Lipschitz constant.

    In the numerical experiments below, we consider standard feedforward neural networks with smooth Lipschitz activation functions, such as \texttt{sigmoid}, \texttt{tanh}, and \texttt{SiLU}. For fixed finite parameters, these networks define Lipschitz feedback maps, although their Lipschitz constants are not explicitly a priori constrained during training. This choice is made for computational convenience.
\end{remark}

\section{Neural-network algorithm}
\label{sec:numerical_algorithm}

We now describe the neural-network implementation used to approximate the time-discrete stochastic control problem associated with the value function $V_0$ introduced above.
The method is based on a direct parametrization of the feedback policy by a neural network and on the minimization of the corresponding simulated cost.

For a given feedback control $\alpha\in\Gcal^N$ and given $\xi \in L^2((\Omega_0,\Fcal_0,\PP_0);\R^d)$, compactly supported on a ball $B_R$ of radius $R>0$, we consider the stochastic process $\widetilde{X}_{t_k}^{\xi,\alpha}$ over the time grid $\{t_k:=k\tau\}_{k=0}^{N}$, for $N \in \N \backslash \{0\}$ and $\tau:=T/N$, generated by the time-discrete dynamics introduced in \Cref{eq:EM_scheme}. Here, for simplicity, we present the Euler--Maruyama scheme; the generalization to higher-order schemes follows according to \Cref{rem:higher-order-schemes}.
In view of the commutation formula in \Cref{prop:comm_formula}, the training objective is the randomized cost $\E_0[\tilde{J}(0,\xi,\alpha)]$.

The convergence analysis of \Cref{sec:convergence_analysis} applies to piecewise-constant feedback policies $\alpha=(\alpha_0,\ldots,\alpha_{N-1})\in\mathcal G^N$, independently of the neural parametrization used to represent them. One may either use a single time-state network $\alpha^\theta(t,x)$ and set $\alpha_k=\alpha^\theta(t_k,\cdot)$ (Examples~1 and~2), or a separate network per time step, $\alpha_k=\alpha_k^{\theta_k}$ (Example~3).
The former shares parameters across time steps, reduces the number of trainable parameters, and introduces an implicit regularization in time, which can be advantageous when the optimal feedback depends regularly on time. The latter yields a larger and more flexible approximation class, which can be preferable when the feedback or the admissible control constraints vary irregularly in time, at the price of a number of trainable parameters that typically grows linearly with the number of time steps.
In both cases, the error estimates apply to the induced discrete policy, provided it belongs to the approximation class of the convergence theorem. We stress that, even with time-step-specific networks, the training remains global in time: all feedback components are optimized simultaneously through the simulated total cost, rather than by a backward DPP recursion as in \cite{bachouch_deep_2022,hure_deep_2021}.

Let $\Theta$ denote the trainable parameters of a neural parametrization and set
\begin{align*}
    \alpha^\Theta=(\alpha_0^\Theta,\ldots,\alpha_{N-1}^\Theta), \qquad \alpha_k^\Theta:\mathbb R^d\to A.
\end{align*}
This notation covers both parametrizations discussed above 
(in the time-state network, $\alpha_k^\Theta(x)=\alpha^\theta(t_k,x)$ with $\Theta=\theta$, whereas in the time-step-specific network, $\alpha_k^\Theta(x)=\alpha_k^{\theta_k}(x)$ with $\Theta=(\theta_0,\ldots,\theta_{N-1})$).

The expectation in the training objective is estimated using $M_{\mathrm{batch}}>0$ simulated trajectories for the process $\widetilde{X}^{\xi,\alpha^\Theta}$. The Monte Carlo estimator associated with the cost functional for a given parametrized policy $\alpha^\Theta$ is then given by
\begin{align}\label{eq:losses}
    \widehat{J}(0,\xi,\alpha^\Theta) := \frac{1}{M_{\mathrm{batch}}}\sum_{m=1}^{M_{\mathrm{batch}}} \bigg[ \sum_{k=0}^{N-1} f(t_k, \wX^{\xi,\alpha^\Theta}_{t_k,m}, \alpha^\Theta_k(\wX^{\xi,\alpha^\Theta}_{t_k,m}))\tau + g(\wX^{\xi,\alpha^\Theta}_{t_N,m}) \bigg],
\end{align}
where $(\wX^{\xi,\alpha^\Theta}_{t_k,m})_{k=0}^{N}$ denotes the $m$-th sampled trajectory, starting at the point $\wX^{\xi,\alpha^\Theta}_{t_0,m} \in B_R$, sampled from the distribution $\mathcal{L}(\xi)$.

The training of the parametrized feedback policy $\alpha^\Theta$ updates the parameters $\Theta$ according to a stochastic gradient-based approximation method, minimizing the empirical loss \eqref{eq:losses}. The resulting training procedure is summarized in \Cref{algo_direct_policy_learning}.

\begin{algorithm}[!ht]
\caption{Direct neural-network policy learning}
\label{algo_direct_policy_learning}
    \SetKwInOut{Input}{Input}
    \SetKwInOut{Output}{Output}
    \Input{Dynamics $b$, $\sigma$; costs $f$, $g$; step-size $\tau$; $\mathcal{L}(\xi)$ s.t. $\operatorname{supp}(\mathcal L(\xi)) \subseteq B_R$ for $R\geq 0$}
    \Input{Control $\alpha^{\Theta}$; learning rate $\gamma$; $M_{\mathrm{batch}}$; $\mathit{N_{\mathrm{epoch}}}$}
    \For{$i = 1, \ldots, \mathit{N_{\mathrm{epoch}}}$}{
        Initialize $v^i = 0$ (cost of the batch)\\ 
        \For{$m = 1, \ldots, M_{\mathrm{batch}}$}{
            Sample $x^{i,m}$ in $B_R$ according to $\Lcal(\xi)$\\
            Generate i.i.d. standard Gaussian random vectors $ \{Z^{i,m}_k \}_{k=1}^{N}$\\
            Initialize $v^{i,m} = 0$ (cost of the trajectory)\\
            \For{$k = 0, \ldots, N-1$}{
                $t_k \leftarrow k \tau$\\
                $a^{i,m}_k \leftarrow \alpha^\Theta_k(x^{i,m})$\\
                $v^{i,m} \leftarrow v^{i,m} + f(t_k, x^{i,m}, a^{i,m}_k) \tau$\\
                $x^{i,m} \leftarrow x^{i,m} + b(t_k, x^{i,m}, a^{i,m}_k) \tau + \sigma(t_k, x^{i,m}, a^{i,m}_k) \sqrt{\tau} Z^{i,m}_{k+1}$
            }
        $v^{i,m} \leftarrow v^{i,m} + g(x^{i,m})$
        }
    $v^i \leftarrow \frac{1}{M_{\mathrm{batch}}}\sum_{m=1}^{M_{\mathrm{batch}}} v^{i,m}$\\
    $\Theta \leftarrow \Theta - \gamma \nabla_{\Theta} v^i$ (update parameters, where  $v^i=v^i(\Theta)$)\\
    }
    \Output{$\widehat{\Theta}$ trained parameters; $\widehat{\alpha}=\alpha^{\widehat{\Theta}}$ approximate feedback policy}
\end{algorithm}

\begin{remark}
    Although the theoretical formulation distinguishes between the auxiliary expectation $\E_0$, associated with the randomized initial condition $\xi$, and the expectation $\E$, associated with the noise of the controlled dynamics, the numerical implementation estimates the resulting product expectation directly. More precisely, at each training iteration we sample independent initial conditions $\xi^m\sim\mathcal L(\xi)$ and, conditionally on each $\xi^m$, independent noise increments driving the time-discrete dynamics. Thus, from \Cref{algo_direct_policy_learning}, $\widehat{J}(0,\xi,\alpha^\Theta)$ is a Monte Carlo estimator of $\E_0[\tilde{J}(0,\xi,\alpha^\Theta)]$.
    Moreover, in the theoretical analysis, the optimization error is represented by the parameter $\eta$. 
    In practice, finite-sample effects from the Monte Carlo approximation and the convergence of the stochastic optimization procedure both contribute to the observed training error.
\end{remark}

\section{Numerical examples}
\label{sec:numerical_examples}
We present three numerical examples illustrating the proposed methodology.
The experiments are designed to identify, and separate as far as possible, the
different error components: the time discretization of the controlled dynamics
(Examples~1 and~2, comparing the Euler--Maruyama scheme with the weak second-order
scheme of \Cref{appendix:KPscheme}), 
the restriction to piecewise-constant policies
(Example~3), and the neural-network approximation and optimization errors.
In all experiments, after training, the reported values are computed by re-evaluating the trained policy on a large independent sample, thereby reducing the Monte Carlo error.

We recall that the classical weak second-order convergence theory
(cf.\ \cite[Ch.~15]{kloeden_numerical_1992}) requires smooth coefficients,
including the feedback policy inserted into the dynamics. In our experiments the
trained network policy is smooth, but the optimal feedback may be discontinuous;
the second-order behavior observed below (in the first two examples) is therefore not covered by the classical
theory and should be regarded as empirical.

\paragraph{Computational setup.}
All numerical experiments were performed on a Linux GPU server using one \texttt{NVIDIA Tesla V100-PCIE-16GB} GPU per run. The server was equipped with three such GPUs and used NVIDIA driver version \texttt{535.309.01} and CUDA version \texttt{12.2}.

\subsection{Example 1: Stochastic radial target problem}
\label{sec:target_circle_example}
This first example is a benchmark target problem with a controlled SDE involving a degenerate diffusion.
We aim to observe the convergence order of time discretization schemes.

\paragraph{The model.}
Let $B$ be a standard real-valued Brownian motion. For $0 \leq t \leq T$, we consider the 2-dimensional process $X_s = (X^1_s, X^2_s)$, defined by the coupled system of stochastic differential equations
\begin{align*}
    \begin{cases}
        \de X^1_s &= - X^2_s \de B_s + \alpha^1_s \de s, \quad s\in(t,T]\\
        \de X^2_s &= \phantom{+} X^1_s \de B_s + \alpha^2_s \de s, \quad s\in(t,T]\\
        X_t &= x \in \R^2,
    \end{cases}
\end{align*}
where the control processes $\alpha_s = ( \alpha^1_s, \alpha^2_s ) = ( \alpha^1(s, X_s), \alpha^2(s, X_s))$ are given in feedback form. When needed, we will denote $X^{t,x,\alpha}_s=X_s$. Set $A:=B_M$ for a given positive constant $M>0$, to constrain the policies into the set $\Acal_M := \{ \alpha \in \Acal \,:\, |\alpha_s|\le M,\, \text{for all } s\in[t,T] \}$.
Given a target radius $r_0>0$, we define
\begin{align*}
    J(t,x,\alpha) := \E \big[ g(X^{t,x,\alpha}_T) \big], \qquad
    g(x) := (|x|-r_0)^2, \qquad V(t,x) = \inf_{\alpha\in\Acal_M} J(t,x,\alpha).
\end{align*}

This model can be interpreted as a target problem, in which we aim to reach the target radius $r_0$, within a finite time horizon $T>0$ and bounded control. Indeed, for $t\in[0,T]$, this defines a \emph{backward reachable set}, described as a \emph{zero-level set} for the value function, given by
\begin{align*}
    \Zcal_t := \{ x\in\R^2 \,:\, V(t,x)=0 \}.
\end{align*}
The condition $x\in\mathcal{Z}_t$ represents the states at time $t$ from which the terminal target can be reached exactly, namely those for which there exists an admissible control $\alpha$, with $|\alpha_s|\leq M$, for all $s\in[t,T]$, such that $|X_T^{t,x,\alpha}|=r_0$. See, for instance, \Cref{fig:trajectories_circle_ex}.

The value function is characterized, in the viscosity sense, by the HJB equation
\begin{align*}
    \begin{cases}
        - \partial_t V(t,x) &= \inf\limits_{|a| \leq M}\bigg\{ \langle a,\, \nabla V (t,x) \rangle + \dfrac{1}{2}\mathrm{Tr}[\sigma\sigma^\top(t,x) \nabla^2 V (t,x)] \bigg\}, \quad x\in\R^2,\ t\in[0,T] \\
        V(T,x) &= (|x|-r_0)^2, \quad x\in \R^2.
    \end{cases}
\end{align*}

An analytic expression for $V$, together with $\mathcal{Z}_t$, can be obtained, see \Cref{appendix:target_ex_computations}.
Notice that, whenever $\nabla V(t,x)\neq 0$, a possible optimal control satisfies
\begin{align}\label{eq:alpha_opt_circle}
   \alpha^*(t,x) = -M \frac{\nabla V (t,x)}{|\nabla V (t,x)|}.
\end{align}
In contrast, in the interior of the backward reachable set, $V$ is locally constant (hence $\nabla V=0$ wherever the classical gradient is defined), therefore the HJB does not select a unique optimal control there. In that region, the terminal target can be reached exactly, and any admissible control whose radial component steers the radius to $r_0$ is optimal; tangential components may be added as long as the constraint $|\alpha_s|\le M$ remains satisfied.

\paragraph{Numerical implementation.}
We consider the time horizon $T=0.5$, the target radius $r_0=1$, and the control bound $M=1$. The feedback control is parametrized by a deep feedforward neural network with input $(t,x_1,x_2)$, 3 hidden layers of $20$ neurons each, and using the $\mathtt{tanh}$ activation function. 
The output layer $\Pi_A(u):=\frac{u}{\max\{1,|u|\}}$ is used in order to generate controls in the unit ball. 

For each value of $N$ and for each time-discretization scheme, the policy is trained independently. During training, the initial condition is sampled uniformly on the centered ball $B_{2.5}$. We use the Adam optimizer with learning rate $\gamma=10^{-3}$, together with large Monte Carlo batches of size $10^5$, over $10^5$ training iterations. This choice is made in order to reduce the stochastic noise in the optimization and to make the error due to the time-discretization scheme clearly visible in the convergence study.

We perform simulations using both the Euler--Maruyama first order scheme and the explicit weak second order scheme of Platen. Since the Brownian noise is one-dimensional, we use the simplified version of the second-order scheme (see \eqref{2order_scheme_circle_example} in Appendix \ref{appendix:KPscheme}; see also \cite[Eq.~(15.1.1)]{kloeden_numerical_1992}).
After training, the value induced by the learned policy is recomputed pointwise, for each fixed initial condition, using $10^5$ independent Monte Carlo trajectories. This significantly reduces the Monte Carlo contribution in the final evaluation, so that the observed errors primarily reflect the approximation properties of the time-discretization scheme and, to a lesser extent, the neural-network approximation and optimization errors.

\medskip
\Cref{fig:trajectories_circle_ex} provides a first qualitative validation of the learned policy. The simulated trajectories show the expected radial behavior: initial states lying outside the backward reachable set are driven as far as possible toward the target circle, while those starting inside the reachable annulus can be steered so as to attain the target exactly within the terminal time.

The contour plot of the value induced by the trained neural-network policy is also in very good agreement with the theoretical structure of the problem, exhibiting an approximately rotationally invariant profile, with nearly concentric level sets. This is a nontrivial qualitative feature, since no symmetry is explicitly imposed in the neural-network architecture. The learned approximation therefore recovers not only the correct magnitude of the value function, but also its expected geometric structure.

\begin{figure}[!t]
    \centering
    \includegraphics[width=0.42\linewidth]{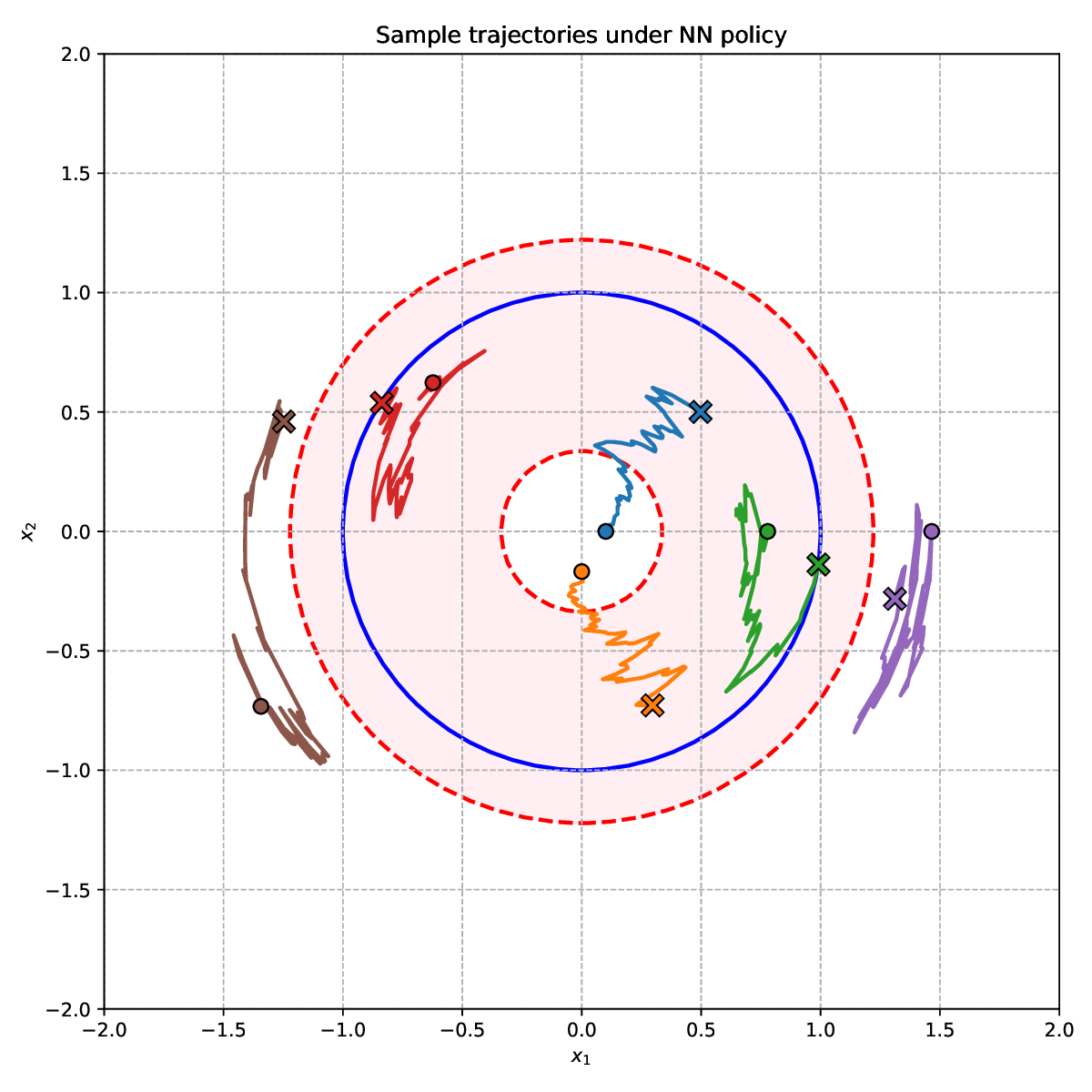}
    \hfill
    \includegraphics[width=0.42\linewidth]{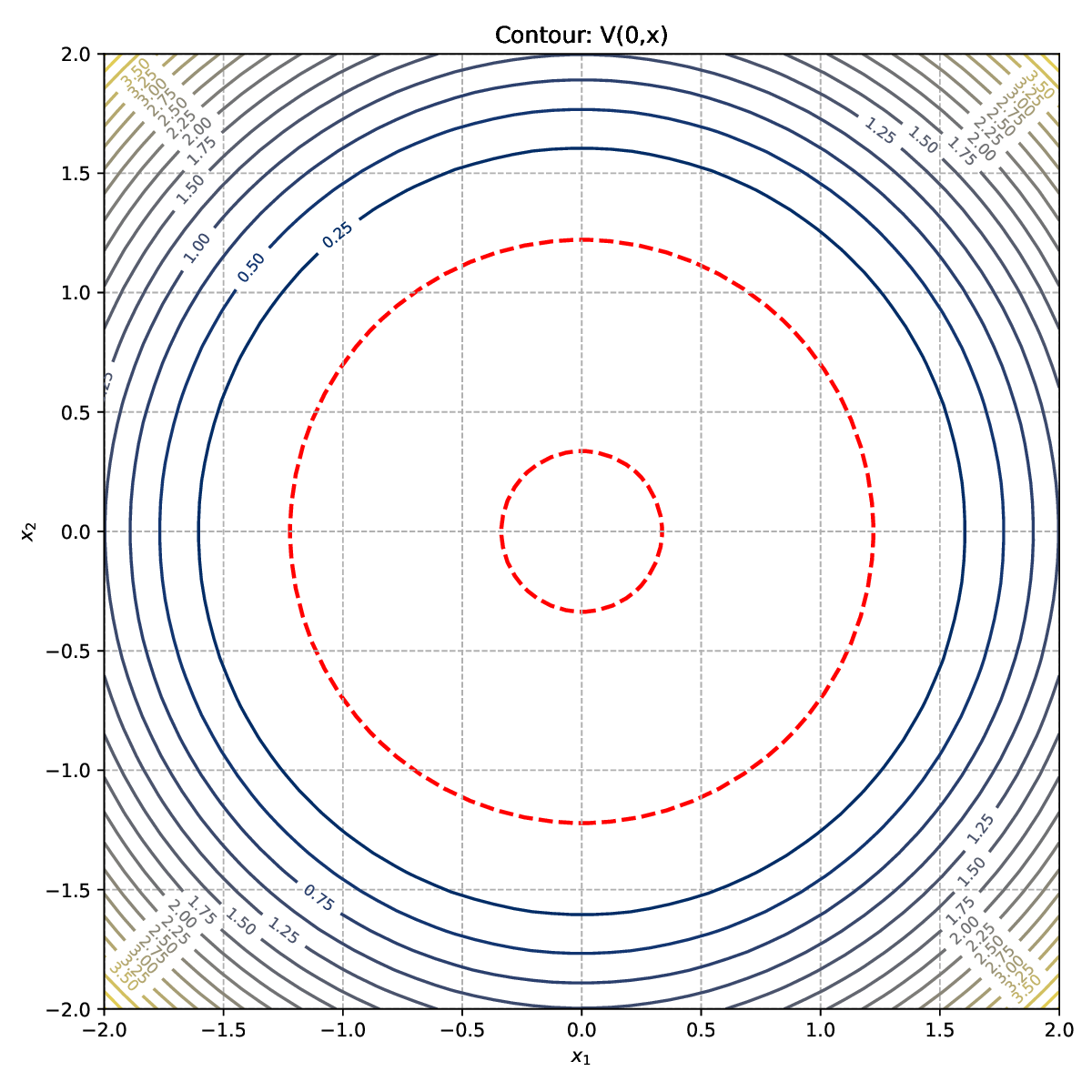}
    \caption{(Radial target problem) \emph{Left}: Simulated independent trajectories over the time horizon $[0,T=0.5]$ under the trained neural-network policy, with control constraint $M=1$. Dots and crosses represent the initial and terminal states, respectively. The blue circle is the target of radius $r_0=1$. The pink annular region represents the backward reachable set at time $t=0$, that is, $\mathcal{Z}_0$. \emph{Right}: Contour plot of the Monte Carlo estimate of the value induced by the trained neural-network policy at time $t=0$. The shape of the level sets is consistent with the rotational invariance of the underlying control problem.}
    \label{fig:trajectories_circle_ex}
\end{figure}

\begin{figure}[htbp]
    \centering
    \setlength{\tabcolsep}{3pt}

    \newcommand{\plotwidth}{0.305\textwidth}
    \newcommand{\plotfig}[1]{%
        \includegraphics[
            width=\plotwidth,
            valign=t
        ]{#1}%
    }

    \newcommand{\Nlabel}[1]{%
        \makebox[\plotwidth][c]{\small $#1$}%
    }

    \newcommand{\rowlabel}[1]{%
        \raisebox{-1.2\height}{\rotatebox[origin=c]{90}{\small #1}}%
    }

    \begin{tabular}{@{}c c c c@{}}

        {}
        & \Nlabel{N=4}
        & \Nlabel{N=16}
        & \Nlabel{N=64}
        \\

        \rowlabel{1-order scheme}
        &
        \plotfig{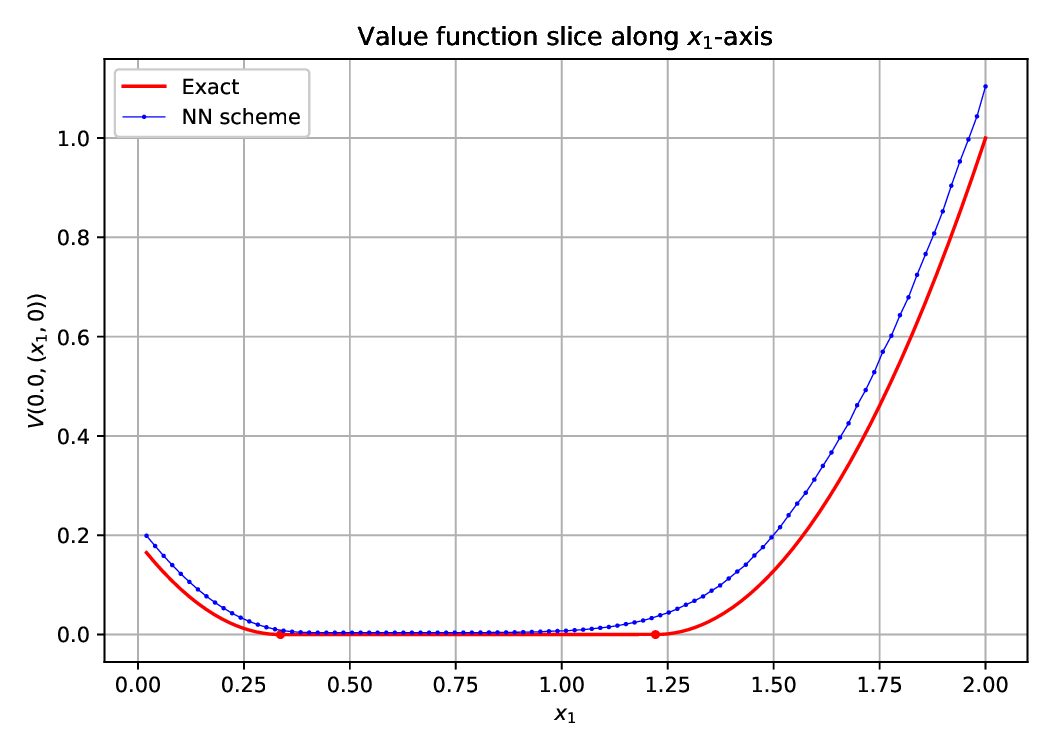}
        &
        \plotfig{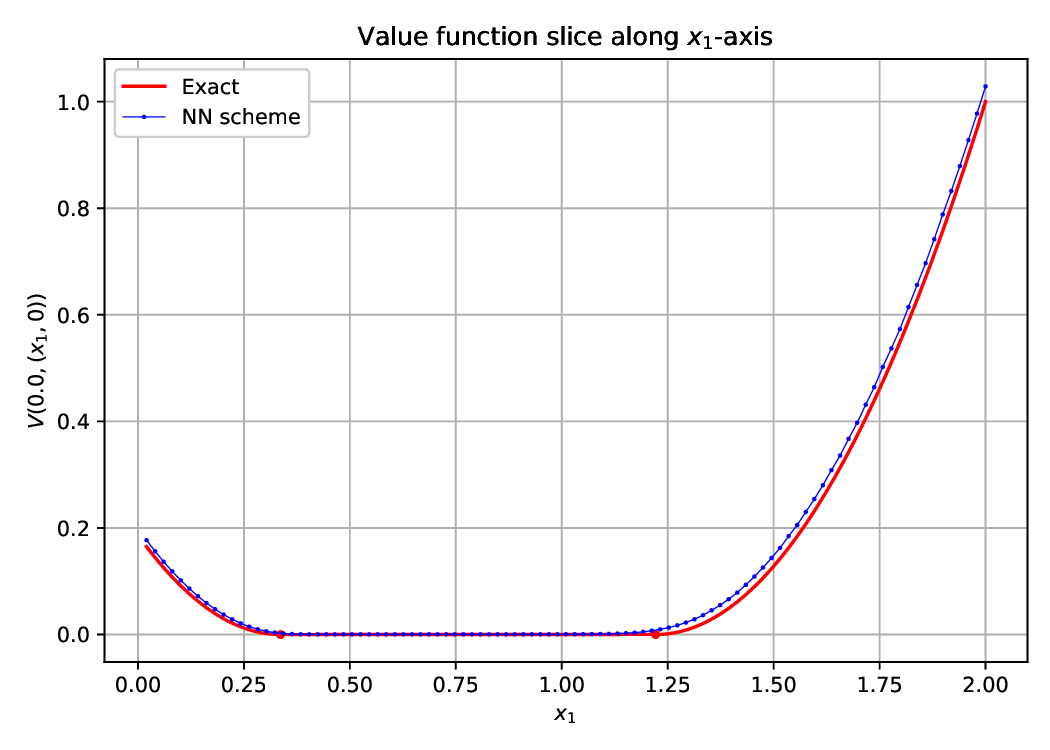}
        &
        \plotfig{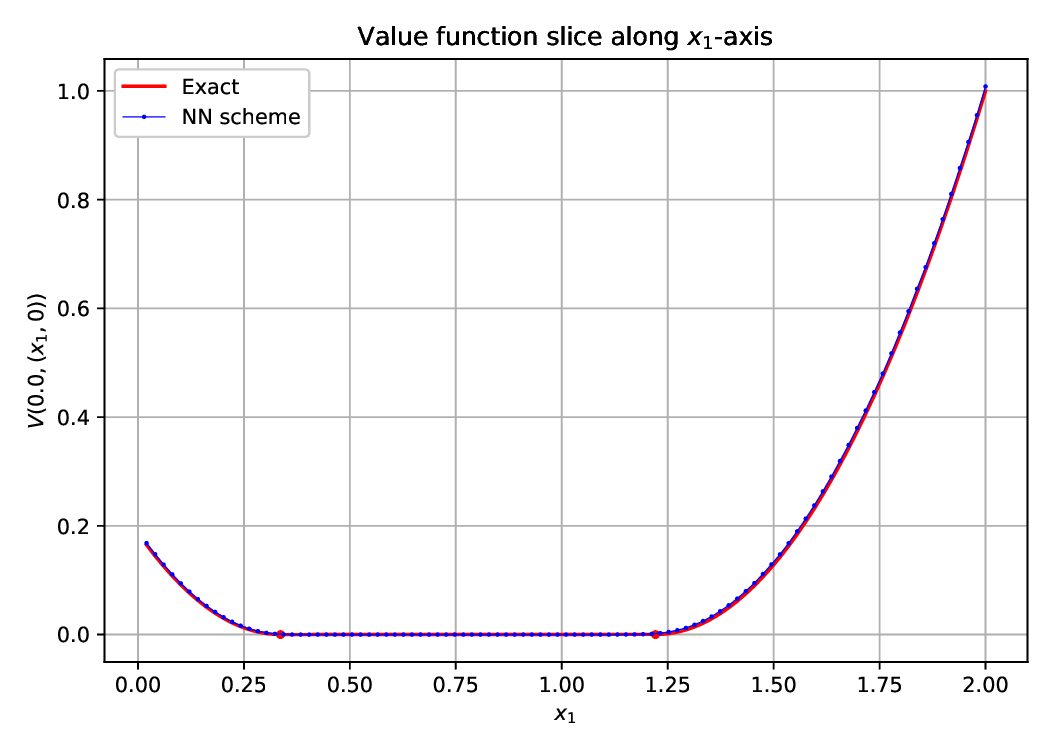}
        \\[0.8em]

        \rowlabel{2-order scheme}
        &
        \plotfig{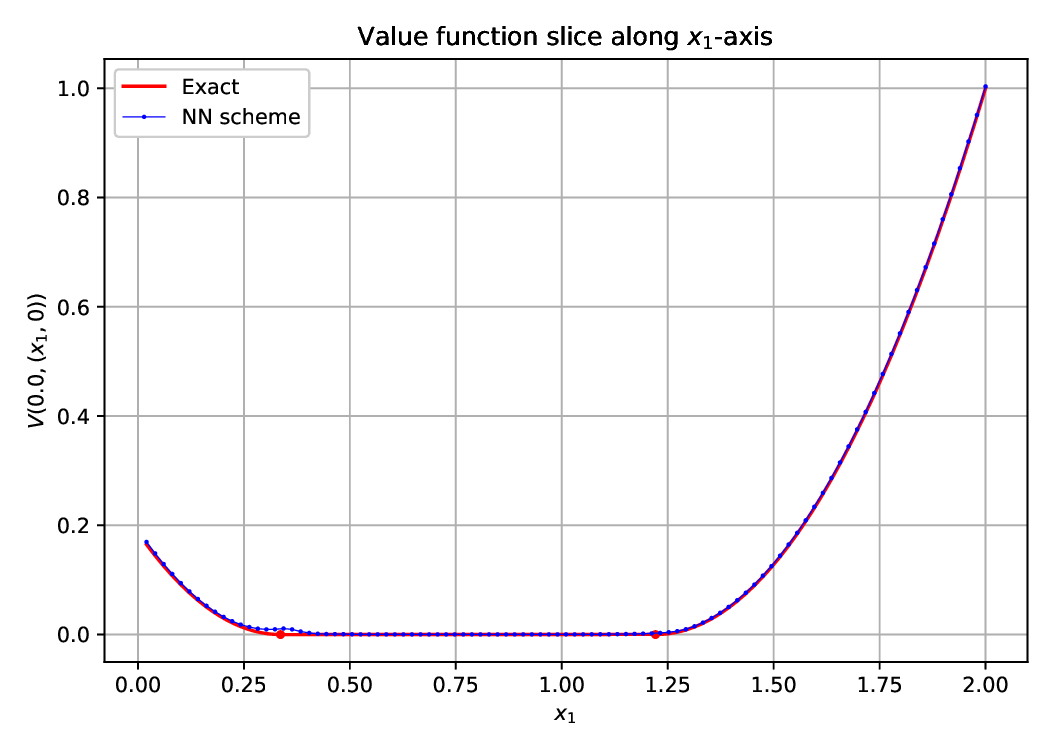}
        &
        \plotfig{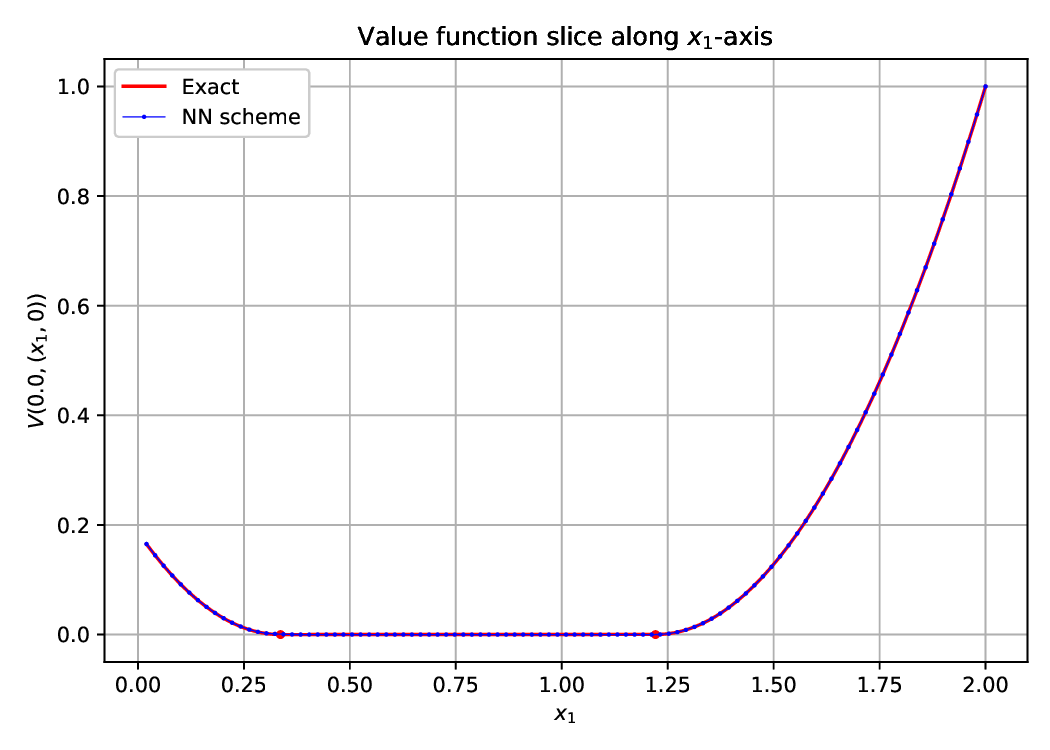}
        &
        \plotfig{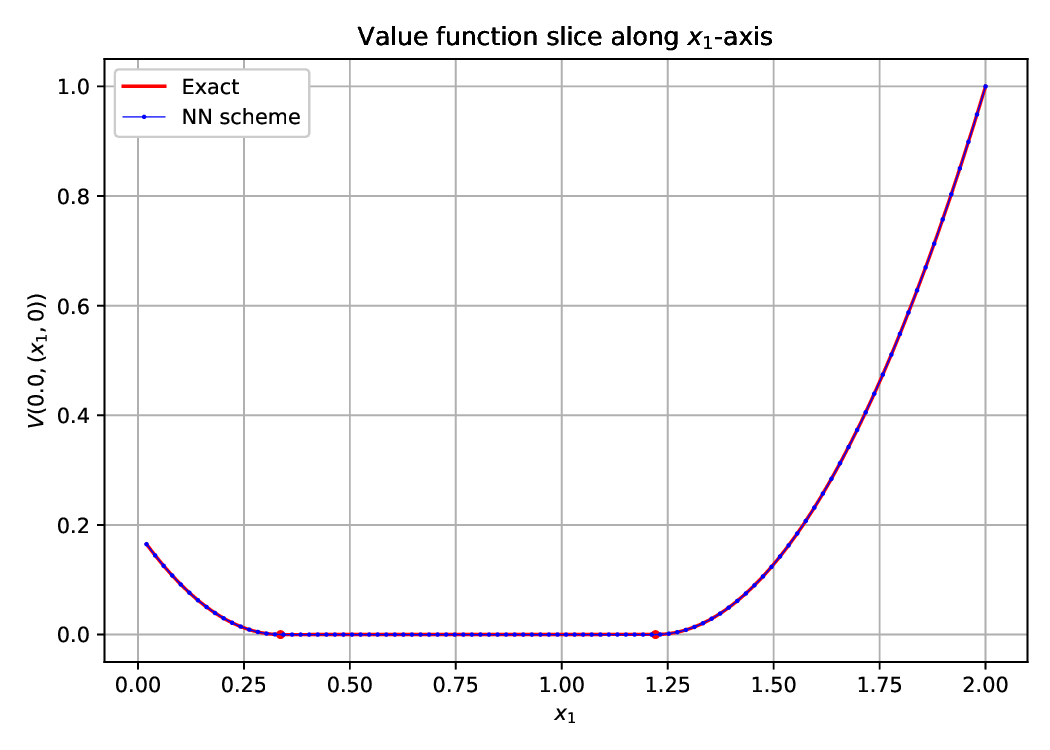}

    \end{tabular}

    \caption{(Radial target problem) Comparison between the exact value function $V(0,x)$ and the Monte Carlo estimate of the value induced by the trained neural-network policy, along the section $x=(r,0)$, $r\in(0,2]$. The value induced by the trained policy is recomputed pointwise using $10^5$ independent Monte Carlo trajectories. \emph{Top row:} Euler--Maruyama scheme. \emph{Bottom row:} Platen weak second order scheme.}
    \label{fig:VF_circle_ex}
\end{figure}

\medskip
The results reported in \Cref{fig:VF_circle_ex} show a clear and robust convergence pattern. The slices of the value function along the section $x=(r,0)$ already indicate that the second-order scheme is substantially more accurate than the first-order one, even for small values of $N$. In particular, for $N=4$, the approximation obtained with the Platen scheme is already very close to the exact value function, whereas the Euler--Maruyama approximation still exhibits a visible bias. As the time grid is refined, both schemes converge to the exact profile, with the second-order method remaining systematically more accurate.

This qualitative behavior is fully confirmed by \Cref{tab:circle_convergence,fig:loglog_convergence_circle_ex}. For the Euler--Maruyama scheme, the empirical convergence rates in the $L^1$ and $L^2$ norms are close to $1$, consistently with a first-order behavior in this benchmark. For the Platen weak second order scheme, the empirical rates are close to $2$, especially in the $L^1$ and $L^2$ norms.
The errors are evaluated along the section $x=(r,0)$, $r\in(0,2]$, by comparing the exact value function $V(0,x)$ with the Monte Carlo estimate induced by the trained policy. The reported $L^1$, $L^2$, and $L^\infty$ errors are computed on a uniform grid in $r$ and correspond to quadrature approximations of the continuous norms along this one-dimensional section.

The fact that the expected rates are recovered so accurately is particularly informative. It shows that, in the present experiment, the dominant contribution to the observed error is indeed the time-discretization error of the numerical scheme. In contrast, the neural-network approximation and optimization errors remain below this level across the tested range of discretizations. Moreover, since the final value is recomputed pointwise with $10^5$ independent Monte Carlo trajectories, the Monte Carlo contribution in the final evaluation is strongly reduced. The results also suggest that restricting to piecewise-constant feedback policies does not dominate the convergence behavior in this benchmark, as otherwise one would not observe the theoretical orders of the underlying schemes so clearly.

\begin{table}[htbp]
\centering
\caption{(Radial target problem) Error tables in discrete $L^1$, $L^2$, and $L^\infty$ norms along the section $x=(r,0)$, $r\in(0,2]$. The error is computed between the exact value function $V(0,x)$ and the Monte Carlo estimate of the value induced by the trained neural-network policy, recomputed pointwise with $10^5$ independent trajectories.}
\label{tab:circle_convergence}
\begin{tabular}{
c|
c
S[table-format=1.2e-1]
S[table-format=1.3, round-mode=places, scientific-notation=false]
S[table-format=1.2e-1]
S[table-format=1.3, round-mode=places, scientific-notation=false]
S[table-format=1.2e-1]
S[table-format=1.3, round-mode=places, scientific-notation=false]
}
\toprule
{} & {$N$} 
& {$\|e\|_{L^1}$} & {Rate} 
& {$\|e\|_{L^2}$} & {Rate} 
& {$\|e\|_{L^\infty}$} & {Rate} \\
\midrule
\multirow{6}{*}{\rotatebox[origin=c]{90}{1-order}}
& 2  & 1.380e-1 & {--}   & 1.298e-1 & {--}   & 1.744e-1 & {--} \\
& 4  & 7.659e-2 & 0.849  & 7.355e-2 & 0.819  & 1.033e-1 & 0.756 \\
& 8  & 4.063e-2 & 0.915  & 3.978e-2 & 0.887  & 5.769e-2 & 0.840 \\
& 16 & 2.116e-2 & 0.941  & 2.109e-2 & 0.916  & 3.019e-2 & 0.934 \\
& 32 & 1.071e-2 & 0.982  & 1.080e-2 & 0.966  & 1.638e-2 & 0.882 \\
& 64 & 5.403e-3 & 0.987  & 5.489e-3 & 0.976  & 8.709e-3 & 0.911 \\
\midrule
\multirow{6}{*}{\rotatebox[origin=c]{90}{2-order}}
& 2  & 2.194e-02 & {--}   & 2.037e-02 & {--}   & 5.263e-02 & {--} \\
& 4  & 4.124e-03 & 2.411  & 4.045e-03 & 2.332  & 1.108e-02 & 2.248 \\
& 8  & 9.789e-04 & 2.075  & 1.066e-03 & 1.924  & 3.514e-03 & 1.657 \\
& 16 & 1.699e-04 & 2.526  & 2.150e-04 & 2.310  & 9.558e-04 & 1.878 \\
& 32 & 3.956e-05 & 2.103  & 5.780e-05 & 1.895  & 3.030e-04 & 1.657 \\
& 64 & 9.599e-06 & 2.043  & 1.296e-05 & 2.157  & 5.964e-05 & 2.345 \\
\bottomrule
\end{tabular}
\end{table}

\begin{figure}[!t]
    \centering
    \includegraphics[width=0.325\linewidth]{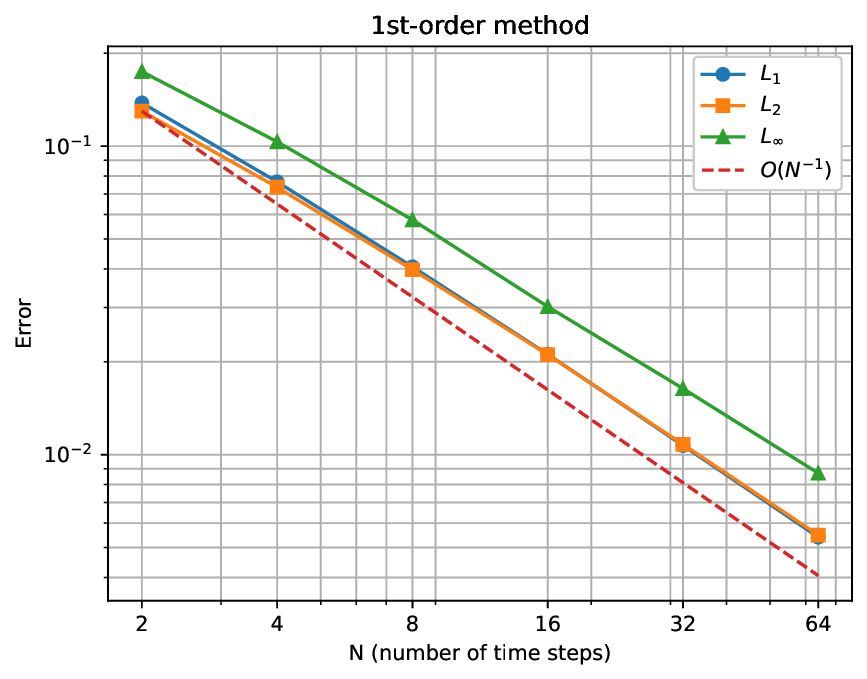}
    \includegraphics[width=0.325\linewidth]{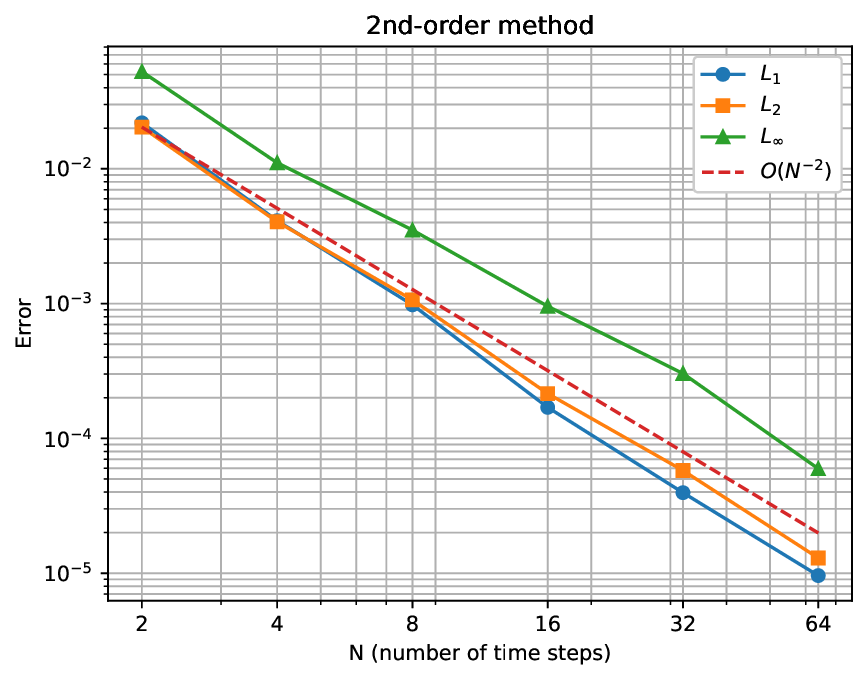}
    \includegraphics[width=0.325\linewidth]{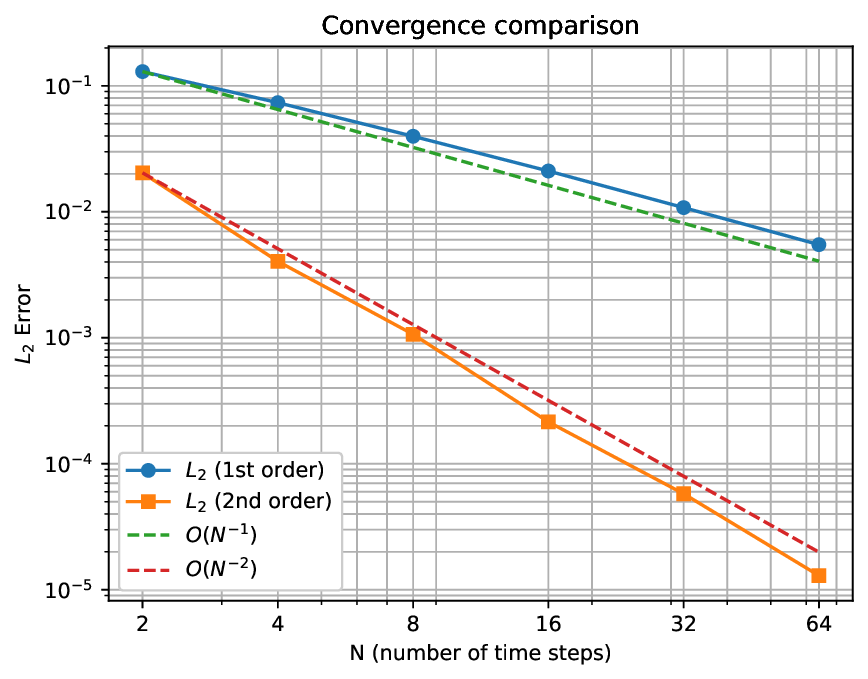}
    \caption{(Radial target problem) Log--log error for the value induced by the trained neural-network policy along the section $x=(r,0)$, $r\in(0,2]$. The dashed reference lines show a very clear first- and second-order convergence rates, as expected.}
    \label{fig:loglog_convergence_circle_ex}
\end{figure}

\subsection{Example 2: Hamilton--Jacobi--Bellman benchmark}
\label{sec:HJB_example}
The second example is a Hamilton--Jacobi--Bellman benchmark inspired by \cite{han-jen-e-18}, and illustrates the improved accuracy of the second-order weak scheme, as well as the regime in which the neural-network approximation error becomes dominant.

\paragraph{The model.}
We consider the HJB benchmark introduced in \cite[Equations (12)--(14)]{han-jen-e-18}, which we reformulate here as a stochastic optimal control problem.
The $d$-dimensional controlled dynamics are described by the following stochastic differential equation
\begin{equation*}
    \de X_t = 2 m_t \,\de t + \sqrt{2} \, \de W_t, \quad X_0 = x_0, \quad t\in[0,T],
\end{equation*}
where $m=(m_t)_{t\in[0,T]}$ represents the $d$-dimensional control process. The associated cost functional is given by
\begin{equation*}
    J\big(0,x_0,m\big) := \E\bigg[\int_0^T \| m_t \|^2 \, \de t + g(X_T) \bigg],
\end{equation*}
with $g(x):=\ln\big(\frac{1+\|x\|^2}{2}\big)$, and $\| \cdot \|$ denoting the Euclidean norm.
The infimum is taken over square-integrable progressively measurable controls.

\paragraph{Semi-analytic representation.}
The HJB equation of the problem is given by
\begin{align*}
\begin{cases}
    -\partial_t v (t,x) &=  \Delta v(t,x) + \inf_{m\in\R^d} \{ 2 \langle \nabla v(t,x), m \rangle + \|m\|^2 \}, \quad (t,x) \in [0,T) \times \R^d\\
    v(T,x) &= g(x), \quad x\in\R^d
\end{cases}
\end{align*}
with minimizer $m^*(t,x) = - \nabla v(t,x)$, which results in the following HJB equation
\begin{align*}
\begin{cases}
    -\partial_t v (t,x) &= \Delta v(t,x) -\| \nabla v(t,x) \|^2, \quad (t,x) \in [0,T) \times \R^d\\
    v(T,x) &= g(x), \quad x\in\R^d.
\end{cases}
\end{align*}
By applying the Hopf--Cole transformation $u=\exp(-v)$ (cf. \cite[Sec. 4.2]{chassagneux_richou_2016}), $u$ satisfies the backward heat equation
$\partial_t u + \Delta u = 0$, with terminal condition $u(T,x) = \exp\{-g(x)\}$,
and we deduce the following semi-analytic formulas for $v$ and $m^*$:
\begin{align}
    \label{eq:theoretical_vf_EHJ}
    v(t,x) &= - \ln\Big(\E\Big[\exp\Big\{-g(x + \sqrt{2}W_{T-t})\Big\}\Big]\Big),\\
    m^*(t,x) &= - \frac{\E[\exp\{-g(x + \sqrt{2}W_{T-t})\} \nabla g(x + \sqrt{2}W_{T-t})]}
    {\E[\exp\{-g(x + \sqrt{2}W_{T-t})\}]}. \label{eq:theoretical_control_EHJ}
\end{align}
We observe that $\|\nabla g\|\leq 1$, therefore, the unconstrained optimal feedback satisfies $\|m^*(t,\cdot)\|\leq 1$, $\forall t\in (0,T)$. Thus, restricting the admissible controls to the compact set $A=B_1$ does not change the value function and allows us to apply the theoretical framework developed in the previous sections.

\paragraph{Numerical implementation.}
In this example, we consider the time horizon $T=1$ in dimension $d=6$. The feedback control is parametrized by a deep feedforward neural network taking as input the $(d+1)$-dimensional vector $(t,x)$. The network has five hidden layers with $80$ neurons each and uses the $\mathtt{SiLU}$ activation function, which provides smoother control outputs with respect to $\mathtt{ReLU}$ for example. 
Since the optimal feedback satisfies $\|m^*(t,x)\|\le1$, the control constraint is enforced through the output activation
\begin{align*}
    \sigma_{\mathrm{out}}(z)
    =
    \begin{cases}
        \tfrac{z}{\|z\|}\sin^4 \big(\tfrac{\pi}{2}\|z\|\big),
        & 0<\|z\|<1,\\
        \tfrac{z}{\|z\|},
        & \|z\|\ge1,\\
        0,
        & z=0.
    \end{cases}
\end{align*}
This activation maps the network output into the unit ball and provides a smooth radial saturation of the control norm, having a flat derivative for both $\| z \| \to 0$ and $\| z \| \to 1$.
For each value of $N$ and for each time-discretization scheme, the policy is trained independently. The training uses the Adam optimizer with initial learning rate $\gamma=10^{-3}$, together with the adaptive learning-rate scheduler \texttt{ReduceLROnPlateau}, from \texttt{PyTorch}
\cite{pytorch_reduce_lr_on_plateau}, applied to the average training loss. The scheduler reduces the learning rate by a factor of $0.5$ after a plateau, with patience $1000$, cooldown $500$, and threshold $10^{-4}$, helping achieve better convergence. We train over $10^4$ iterations, using Monte Carlo batches of size $10^4$.
During training, the initial condition is sampled from a compactly supported isotropic distribution in $B_{2.5}$, by drawing a random direction uniformly on the sphere and an independent radius uniformly in $[0,2.5]$.
We perform simulations with both the Euler--Maruyama scheme and the Platen weak order $2.0$ scheme. Since the diffusion coefficient is constant, the second-order scheme reduces to the form reported in \eqref{2order_scheme_EHJ_example} from \Cref{appendix:KPscheme}. After training, the value induced by the learned policy is recomputed pointwise along the state-space diagonal, at time $t=0$, using $10^6$ independent Monte Carlo trajectories for each point.

\begin{figure}[htbp]
\centering

\begin{minipage}[t]{0.60\textwidth}
\vspace{7.5pt}
\centering
\small
\setlength{\tabcolsep}{4pt}

\begin{tabular}{
c|
S[table-format=1.3e-2]
S[table-format=1.3, round-mode=places, scientific-notation=false]
|
S[table-format=1.3e-2]
S[table-format=1.3, round-mode=places, scientific-notation=false]
}
\toprule
{} 
& \multicolumn{2}{c|}{1-order}
& \multicolumn{2}{c}{2-order} \\
\cmidrule(lr){2-3}
\cmidrule(lr){4-5}
{$N$}
& {$\|e\|_{L^2}$} & {Rate}
& {$\|e\|_{L^2}$} & {Rate} \\
\midrule
2  & 8.940e-02 & {--}  & 1.290e-02 & {--} \\
4  & 4.910e-02 & 0.865 & 3.340e-03 & 1.949 \\
8  & 2.570e-02 & 0.934 & 1.260e-03 & 1.406 \\
16 & 1.340e-02 & 0.940 & 9.720e-04 & 0.374 \\
32 & 7.060e-03 & 0.924 & 8.660e-04 & 0.167 \\
64 & 3.850e-03 & 0.875 & 8.230e-04 & 0.073 \\
\bottomrule
\end{tabular}

\end{minipage}
\hfill
\begin{minipage}[t]{0.38\textwidth}
\vspace{0pt}
\centering
\includegraphics[width=\textwidth]{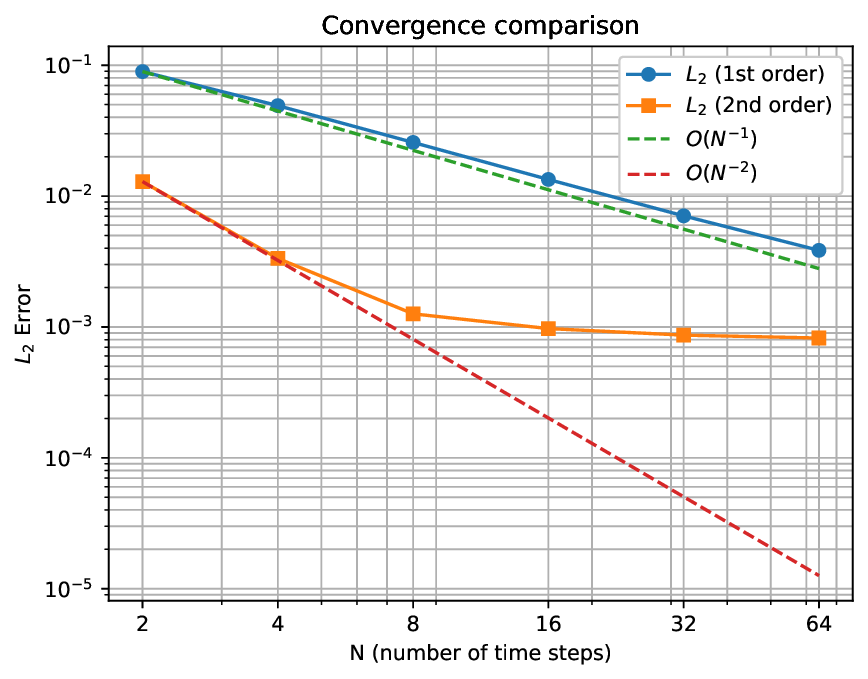}
\end{minipage}

\caption{(HJB benchmark problem) Error tables in discrete $L^2$ norm and corresponding log--log plot for the value induced by the trained neural-network policy. The error is computed along the state-space diagonal $x=(\xi,\dots,\xi)$, $\xi\in[0,2]$, by comparison with the semi-analytic value function. The value induced by the trained policy is recomputed using $10^6$ independent Monte Carlo trajectories for each point on the diagonal. The dashed lines indicate first- and second-order reference rates.}
\label{fig:l2_convergence_table_plot}

\end{figure}

\medskip
The convergence results are reported in \Cref{fig:l2_convergence_table_plot}. The error is computed along the state-space diagonal $x=(\xi,\dots,\xi)$ for $\xi\in[0,2]$,
by comparing the semi-analytic value function in \eqref{eq:theoretical_vf_EHJ} with the Monte Carlo estimate of the value induced by the trained policy. The reported $L^2$ errors are computed on a uniform grid in $\xi$, as quadrature approximations of the corresponding continuous norm along this one-dimensional section.
The Euler--Maruyama scheme exhibits a convergence rate close to 1 across the tested time steps. 
The Platen scheme is significantly more accurate already on coarse grids and initially shows a second-order decay of the error. However, after a few refinements, the second-order curve flattens and no longer follows the reference $O(N^{-2})$ slope. 
We think this behavior is consistent with the neural-network approximation error becoming dominant once the time-discretization error has been sufficiently reduced, suggesting the saturation of the capacity of the considered network.

This example therefore complements the radial target benchmark: previously, the neural network error was small enough to recover the expected orders of both schemes over the full range of discretizations. Here, instead, the second-order scheme rapidly reduces the discretization error to a level at which the residual error is mainly governed by the neural networks' approximation capacity and training accuracy.

\begin{figure}[htbp]
    \centering
    \setlength{\tabcolsep}{4pt}

    \newcommand{\plotwidth}{0.4\textwidth}
    \newcommand{\plotfig}[1]{%
        \includegraphics[
            width=\plotwidth,
            valign=t
        ]{#1}%
    }

    \newcommand{\Nlabel}[1]{%
        \makebox[\plotwidth][c]{\small $#1$}%
    }

    \newcommand{\rowlabel}[1]{%
        \raisebox{-1.6\height}{\rotatebox[origin=c]{90}{\small #1}}%
    }

    \begin{tabular}{@{}c c c@{}}

        {}
        & \Nlabel{N=4}
        & \Nlabel{N=16}
        \\[0.6em]

        \rowlabel{1-order scheme}
        &
        \plotfig{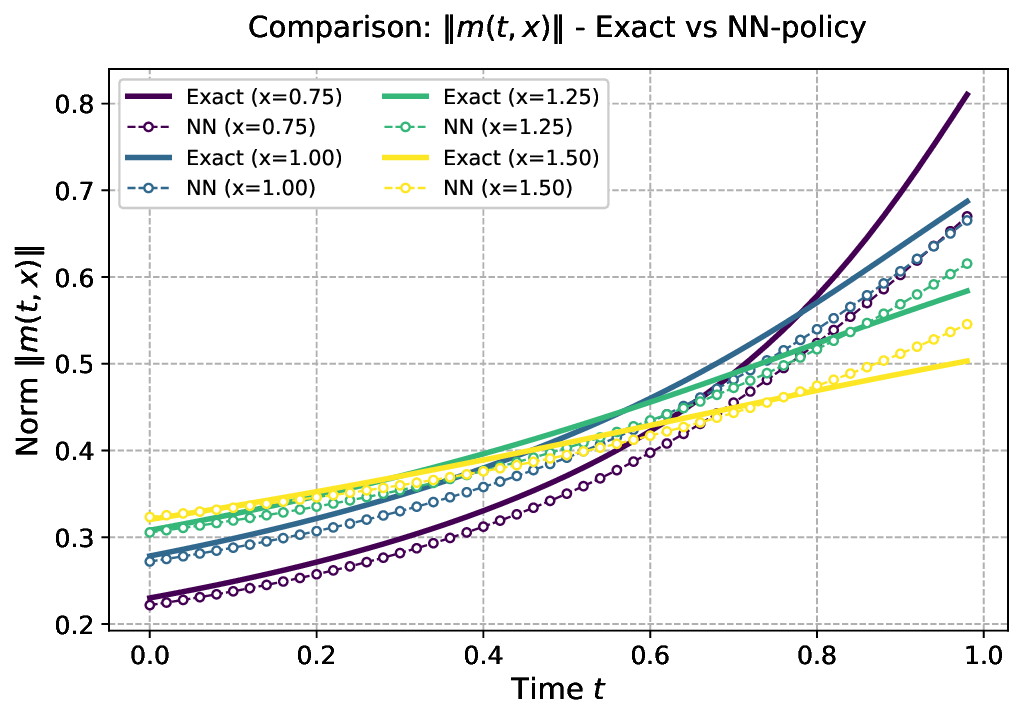}
        &
        \plotfig{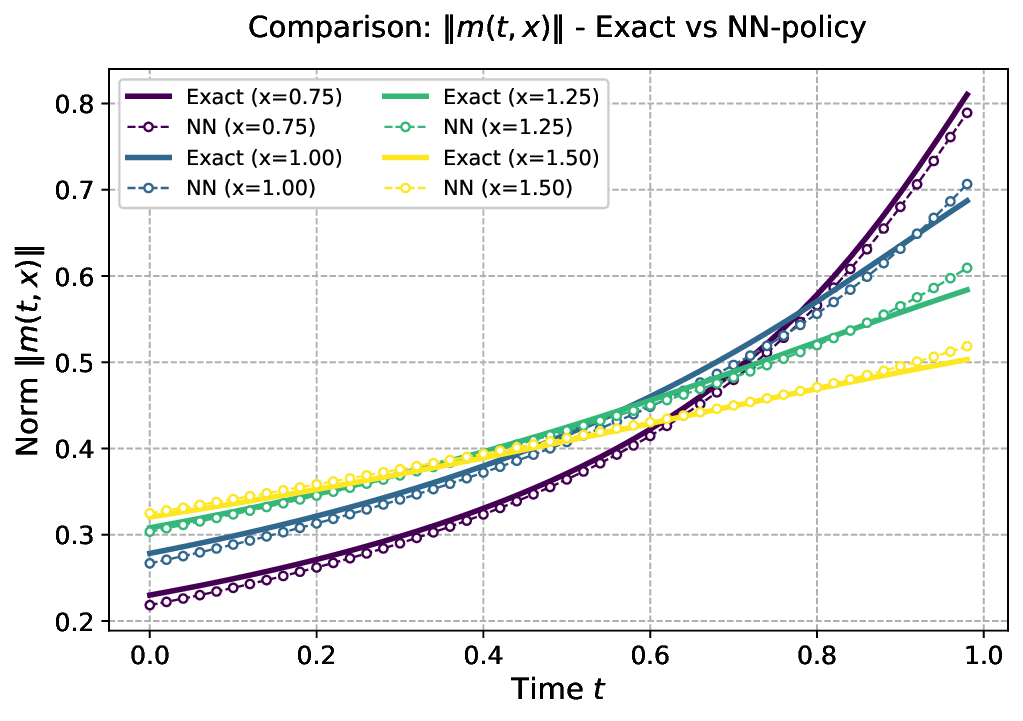}
        \\[0.8em]

        \rowlabel{2-order scheme}
        &
        \plotfig{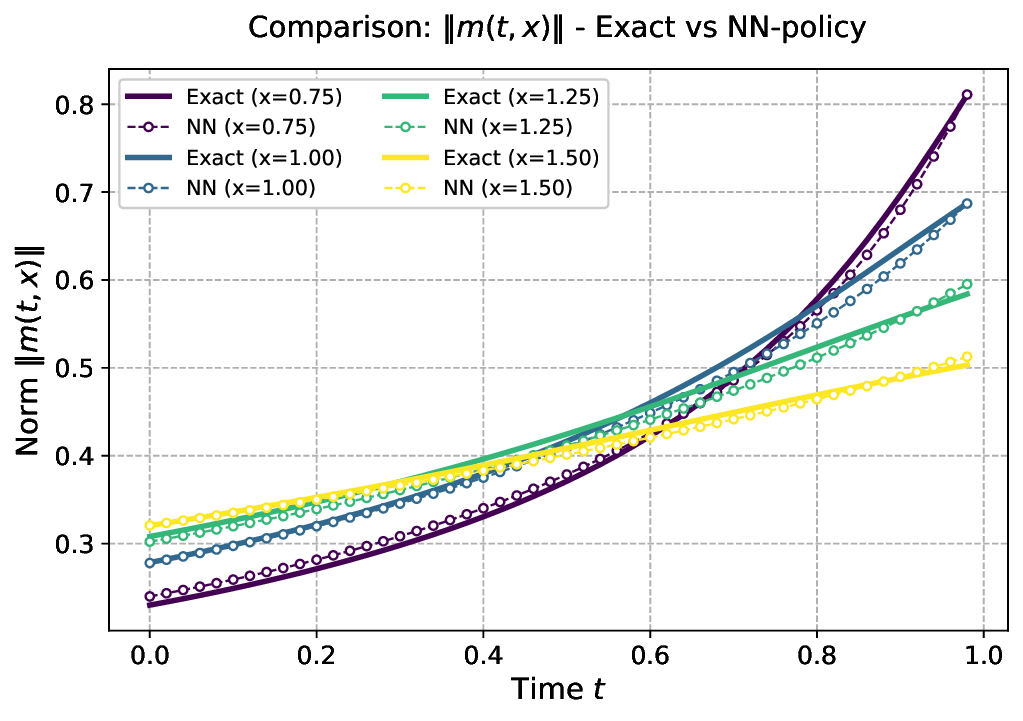}
        &
        \plotfig{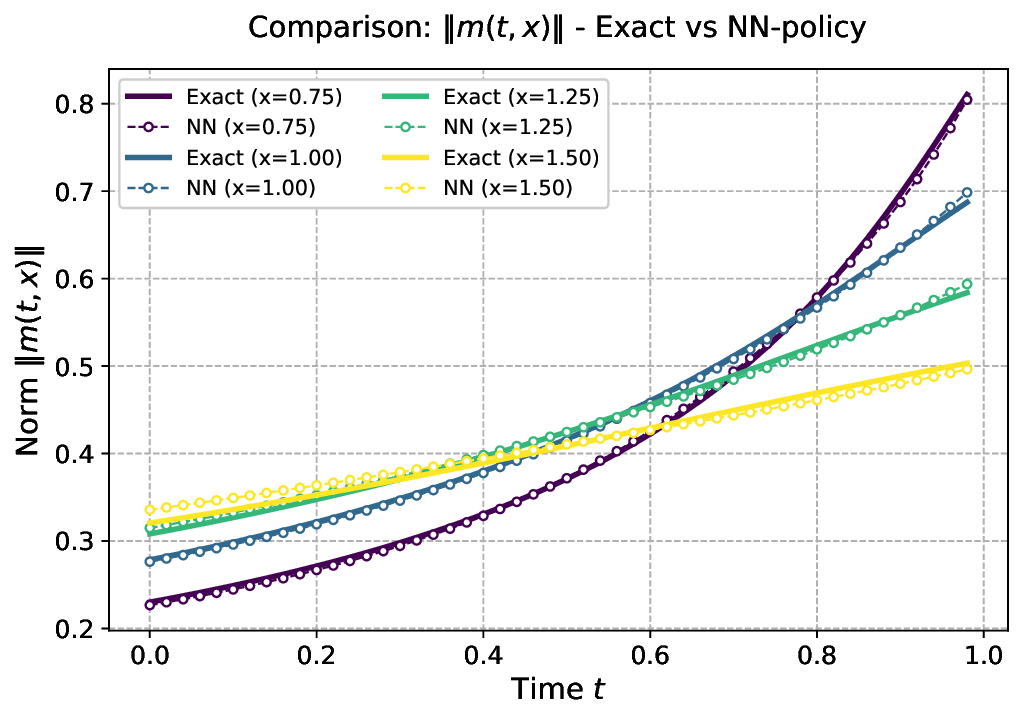}

    \end{tabular}

    \caption{(HJB benchmark problem) 
    Comparison between the norm of the semi-analytic optimal feedback and the norm of the learned neural-network feedback along the diagonal $x=(\xi,\dots,\xi)$, evaluated on a uniform grid of time points in $[0,1]$ and selected values $\xi\in\{0.75,1.0,1.25,1.5\}$.}
    \label{fig:control_norm_HJB_ex}
\end{figure}

\medskip
\Cref{fig:control_norm_HJB_ex} compares the norm of the learned feedback control with the semi-analytic optimal one along the diagonal $x=(\xi,\dots,\xi)$, for selected values of $\xi$ in the region typically explored during training. This comparison is possible in the present example because the optimal feedback is explicitly characterized by \eqref{eq:theoretical_control_EHJ}.
The plots show that the learned control captures the qualitative time profile of the optimal control. The agreement improves as the number of time steps used in training increases, especially in the region where the controlled trajectories are more frequently observed. Moreover, for the same value of $N$, the policy trained with the second-order scheme generally provides a closer approximation of the optimal control norm than the one trained with the Euler--Maruyama scheme. This is consistent with the value function errors from \Cref{fig:l2_convergence_table_plot}, where the second-order scheme reduces discretization bias, while the remaining discrepancy in the control reflects residual neural-network approximation and optimization errors.

\newcommand{\Qmin}{Q_{\min}}
\newcommand{\Qmax}{Q_{\max}}
\subsection{Example 3: Gas storage problem}
\label{sec:gas_storage_example}

The third example concerns a gas storage problem (cf. \cite{barrera2006numerical,carmona2010valuation,warin2012gas}) where prices are driven by Ornstein--Uhlenbeck processes. In this case, the exact simulation of the Ornstein--Uhlenbeck factors allows us to eliminate the time-discretization error in the exogenous stochastic factors and to focus on the effect of the piecewise-constant policy approximation with respect to the number of time steps.

\paragraph{The model.}
We examine a realistic gas storage problem in which the reservoir manager aims to maximize expected profit by optimally scheduling injections and withdrawals and purchasing the commodity when prices are low and selling it when prices are high. We consider a standard model for futures prices with short and long-term volatility factors. For each delivery maturity $\mathcal{T}>0$, the futures price $F(t,\mathcal{T})$, for $t\in[0,\mathcal{T}]$, follows the stochastic differential equation
\begin{align}\label{eq:futures}
    \frac{\de F(t,\mathcal{T})}{F(t,\mathcal{T})} = e^{-a(\mathcal{T}-t)} \sigma^S \de W^S_t + \sigma^L \de W^L_t,
\end{align}
for $a, \sigma^S, \sigma^L$ positive, with $W^S, W^L$ independent real-valued Brownian motions. The explicit solution to the linear SDE \eqref{eq:futures} is given by
\begin{align*}
    F(t,\mathcal{T}) = F(0,\mathcal{T}) \exp \bigg\{ \int_0^t e^{-a(\mathcal{T}-u)} \sigma^S \de W^S_u + \sigma^L W^L_t -\frac{1}{2} \int_0^t \Big( e^{-2a(\mathcal{T}-u)}(\sigma^S)^2 + (\sigma^L)^2 \Big) \de u \bigg\}.
\end{align*}
Denoting the spot price $S_t:=F(t,\mathcal{T}=t)$, we get
\begin{align}\label{eq:spot_price}
    S_t = F(0,t)\exp\bigg\{ \int_0^t e^{-a(t-u)}\sigma^S\de W^S_u + \sigma^L W^L_t -\frac{1}{2}\int_0^t \Big( e^{-2a(t-u)}(\sigma^S)^2 + (\sigma^L)^2 \Big) \de u \bigg\},
\end{align}
where we observe, by defining $\hat{W}^S_t := \int_0^t e^{-a(t-u)}\sigma^S\de W^S_u$, i.e., the solution to the Ornstein--Uhlenbeck process $\de \hat{W}^S_t = -a\hat{W}^S_t\de t + \sigma^S\de W^S_t$, and $\hat{W}^L_t := \sigma^L W^L_t$, that we are able to recover the Markovian setting of $S_t$ as a function of $(t, \hat{W}^S_t, \hat{W}^L_t)$.\\

The reservoir manager, according to the spot price $S_t$, aims to maximize the expected profit over a finite time horizon $T>0$, typically one year, given by $J(u) := - \E [ \int_0^T S_t u_t \de t ]$, where $(u_t)_{t\in[0,T]}$ represents the control process, i.e., injection ($u_t>0$) or withdrawal ($u_t<0$) rate. The optimization problem is subject to constraints concerning the characteristics of the storage:
\begin{itemize}
    \item $(Q_t)_{t\in[0,T]}$, the gas level in the storage, described by $\de Q_t = u_t \de t$, with initial inventory level $Q_0 = Q_{\mathrm{init}}$;
    \item withdrawal $C_W(Q)$ and injection $C_I(Q)$ rates, depending on the gas level in the storage $Q$;
    \item $Q_{\mathrm{min}}(t)$, $Q_{\mathrm{max}}(t)$, as the minimum and maximum levels of gas in the storage, at time $t$.
\end{itemize}
This results in a non-linear control problem, with flow constraints, for $t\in[0,T]$
\begin{align*}
    \begin{cases}
        -C_W(Q_t) \leq u_t \leq C_I(Q_t),\\
        Q_{\mathrm{min}}(t) \leq Q_t \leq Q_{\mathrm{max}}(t),
    \end{cases}
\end{align*}
where the first is a constraint on the control, and the second is a state constraint.

\paragraph{Time discretization and reduction to control constraints.}
The controlled Markov state is given by $(t,Q_t,\hat W^S_t,\hat W^L_t)_{t\in[0,T]}$, where the dynamics are degenerate in the inventory component $Q$. We measure time in days and consider the storage problem over one year, setting $T=365$, and $t_n=n\tau$, with $\tau=T/N$, for $n=0,\dots,N$.
The parameters appearing in the continuous-time futures model are rescaled accordingly. More precisely, if $a_{\mathrm{year}}$, $\sigma^S_{\mathrm{year}}$, and $\sigma^L_{\mathrm{year}}$ denote the annualized parameters, we set
\begin{align*}
a=\frac{a_{\mathrm{year}}}{365},
\qquad
\sigma^S=\frac{\sigma^S_{\mathrm{year}}}{\sqrt{365}},
\qquad
\sigma^L=\frac{\sigma^L_{\mathrm{year}}}{\sqrt{365}}.
\end{align*}
In the sequel, to simplify notation, we keep writing $a,\sigma^S,\sigma^L$ for these daily-scaled parameters.
The exogenous short-term Ornstein--Uhlenbeck factor and the long-term Brownian factor are simulated exactly on the grid, consistent with \eqref{eq:spot_price}, and the discrete-time dynamics are given by the recursive relation
\begin{align*}
    \begin{cases}
        \hat{W}^S_{n+1} &= e^{-a\tau}\hat{W}^S_{n} + \sigma^S\sqrt{\frac{1-e^{-2a\tau}}{2a}}Z^S_{n+1},\\
        \hat{W}^L_{n+1} &= \hat{W}^L_{n} + \sigma^L\sqrt{\tau}Z^L_{n+1},\\
        Q_{n+1} &= Q_{n} + U_{n},\\
        S_n &= F(0,t_n)\exp \Big\{ \hat{W}^S_{n} + \hat{W}^L_{n} - \frac{(\sigma^S)^2}{4a}(1-e^{-2at_n}) -\frac{1}{2}(\sigma^L)^2t_n \Big\},
    \end{cases}
\end{align*}
where $(Z^S_{n+1},Z^L_{n+1})$ are independent standard normal random variables.
The discrete control $U_n$ represents the net amount of gas injected into the storage during the period $[t_n,t_{n+1})$ and is related to the continuous-time rate by $U_n=\tau u_n$. Hence, the discretized reward functional is $\tilde{J}(U) = - \E \big[\sum_{n=0}^{N-1}S_n U_n \big]$, which corresponds to the expected trading profit associated with the discrete injection/withdrawal strategy.

The flow constraints take the form
\begin{align}
        \label{eq:discrete_constraints_control}
        &U_n \in [-C_W(Q_n)\tau ,\ C_I(Q_n)\tau] =: [U_W(Q_n) ,\ U_I(Q_n)], \qquad n=0,\dots,N-1,\\
        \label{eq:discrete_constraints_state_v1}
        &Q_n \in [\Qmin(t_n) ,\ \Qmax(t_n)], \qquad n=0,\dots,N,
\end{align}
where $U_W(Q)\le0$ denotes the withdrawal lower bound and $U_I(Q)\ge0$ the injection upper bound.
The parametrization of the control process, using neural networks, is introduced as follows. The state constraint \eqref{eq:discrete_constraints_state_v1} is equivalent to using
\begin{align}
    \label{eq:discrete_constraints_state_v2}
    U_n \in [\Qmin(t_{n+1}) - Q_n ,\ \Qmax(t_{n+1}) - Q_n ], \qquad n=0,\dots,N-1.
\end{align}
The intersection between \eqref{eq:discrete_constraints_control} and \eqref{eq:discrete_constraints_state_v2} yields a state-dependent admissible interval for the discrete control, which enforces both the flow and the inventory constraints at the next grid point:
\begin{equation}
\label{eq:discrete_constraints_final}
\begin{aligned}
    U_n \in [\underline U_n(Q_n) ,\ \overline U_n(Q_n)], \quad \text{where} \quad
    \begin{cases}
        \underline U_n(Q) &:= \max \{ \Qmin(t_{n+1}) - Q ,\ U_W(Q) \},\\
        \overline U_n(Q) &:= \min \{ \Qmax(t_{n+1}) - Q ,\ U_I(Q) \}.
    \end{cases}
\end{aligned}
\end{equation}
Since the problem is Markovian in
$(Q,\hat W^S,\hat W^L)$
at each time step, we restrict to feedback controls of the form
$U_n=\alpha_n(Q_n,\hat W^S_n,\hat W^L_n)$.
We introduce a network per time step, $\phi_n^{\theta_n}$ with parameters $\theta_n$, as an operator from $\R^3$ to $[0,1]$ such that, following \cite{warin2023reservoir}, the feedback policies are approximated by
\begin{align*}
   U_n^{\theta_n}
   = \underline U_n(Q_n) + \Big( \overline U_n(Q_n)-\underline U_n(Q_n) \Big) \phi_n^{\theta_n}(Q_n,\hat W^S_n,\hat W^L_n), \qquad n=0,\dots,N-1.
\end{align*}
Noting $\theta = (\theta_n)_{n=0}^{N-1}$, we approximate the optimal reservoir management by solving
\begin{align*}
    \theta^*= \argmax_{\theta} \tilde{J}(U^\theta) = \argmin_{\theta} \E\bigg[ \sum_{n=0}^{N-1} S_{n} U^{\theta}_n 
    \bigg].
\end{align*}

\paragraph{Numerical implementation.}
In this example, the exogenous price factors driving the spot price are simulated exactly on the time grid. Hence, unlike in the first two examples, there is no time-discretization error associated with the stochastic price factors. The observed error is therefore mainly related to the restriction to piecewise-constant policies, together with the neural-network approximation and optimization errors. 

The storage horizon runs from April 1, 2024, to April 1, 2025, corresponding to $T=365$ days. The annualized parameters used for the two-factor price model are $a_{\mathrm{year}}=1.9641$, $\sigma^S_{\mathrm{year}}=0.7877$, and $\sigma^L_{\mathrm{year}}=0.4968$.
In order to estimate the parameters $a_{\mathrm{year}}$, $ \sigma^S_{\mathrm{year}}$, and $\sigma^L_{\mathrm{year}}$, a maximum-likelihood method is used to fit the returns of the different products available on the market, namely spot prices and monthly products. The spot product, more precisely the day-ahead product, together with end-of-week, end-of-month, and monthly products, are observed on the market. Since these products correspond to non-overlapping delivery periods, a piecewise constant-in-time representation of $F(0,\cdot)$ on a daily grid is obtained from these quotations.

In order to isolate the effect of the time discretization of the control policy, the initial futures curve is replaced by a piecewise-constant approximation on the coarsest grid used in the experiments, corresponding to $N=5$.
Since the storage horizon is $T=365$ days, this gives a time step $\tau=73$ days. More precisely, we set
$F(0,t)
=
\sum_{i=0}^{4} \widehat F_i \mathbf 1_{I_i}(t)$, with
$I_i=[i\tau,(i+1)\tau)$ for $i=0,\dots,3$, and $I_4=[4\tau,T]$.
The constants $\widehat F_i$ are computed from the original market forward curve by averaging its values over the corresponding time window $I_i$.
The same input curve is then used for all finer grids. This avoids introducing additional deterministic variations in the futures curve as the number of decision dates increases, so that the convergence study primarily reflects the refinement of the piecewise-constant policy.
We use the simplified storage configuration, with no seasonal gate constraint (generally imposed by the storage owner when renting part of the capacity to third parties), and constant injection/withdrawal rates:
$Q_{\max}=10^6$,
$Q_{\min}=0$,
$Q_0=0$,
$C_I=\frac{Q_{\max}}{110}$, and
$C_W=\frac{Q_{\max}}{59}$.
Notice that, in this setting, the objective is affine with respect to the storage control, which leads to a bang--bang optimal feedback.
For each value of $N$, we implement one feedforward neural network per time step.
Each network takes as input the three-dimensional Markov state
$(Q_n,\hat W^S_n,\hat W^L_n)$
and outputs a scalar in $[0,1]$ through a final $\mathtt{sigmoid}$ activation. The hidden layers use the $\mathtt{tanh}$ activation function. In all simulations reported here, each network (per time-step) has two hidden layers with $10$ neurons each. The output of the network is then mapped into the admissible control interval according to \eqref{eq:discrete_constraints_final}.

\medskip
Since no closed-form solution is available for this constrained storage problem, we use the value obtained with $N=320$ as a reference benchmark, $v_{\mathrm{ref}}:=v_{320}$. This grid corresponds to a time step of
$\tau_{\mathrm{ref}}=\frac{365}{320}\simeq 1.14$
days, hence approximately one policy update per day.
We consider the time grids $N\in\{5,10,20,40,80,160,320\}$.
For $N=5,\dots,160$, the networks are trained for $8\times 10^4$ Adam iterations, using learning rate $\gamma=10^{-3}$ and Monte Carlo batches of size $10^5$. For the reference computation with $N=320$, the training is increased to $16\times10^4$ Adam iterations. During training, the current policy is periodically evaluated using batches of size $10^6$. After training, the reported value is recomputed with $10^8$ independent Monte Carlo trajectories, processed in batches of size $5\times10^6$.

\begin{figure}[!t]
    \centering

    \begin{subfigure}{0.45\textwidth}
        \centering
        \includegraphics[width=\linewidth]{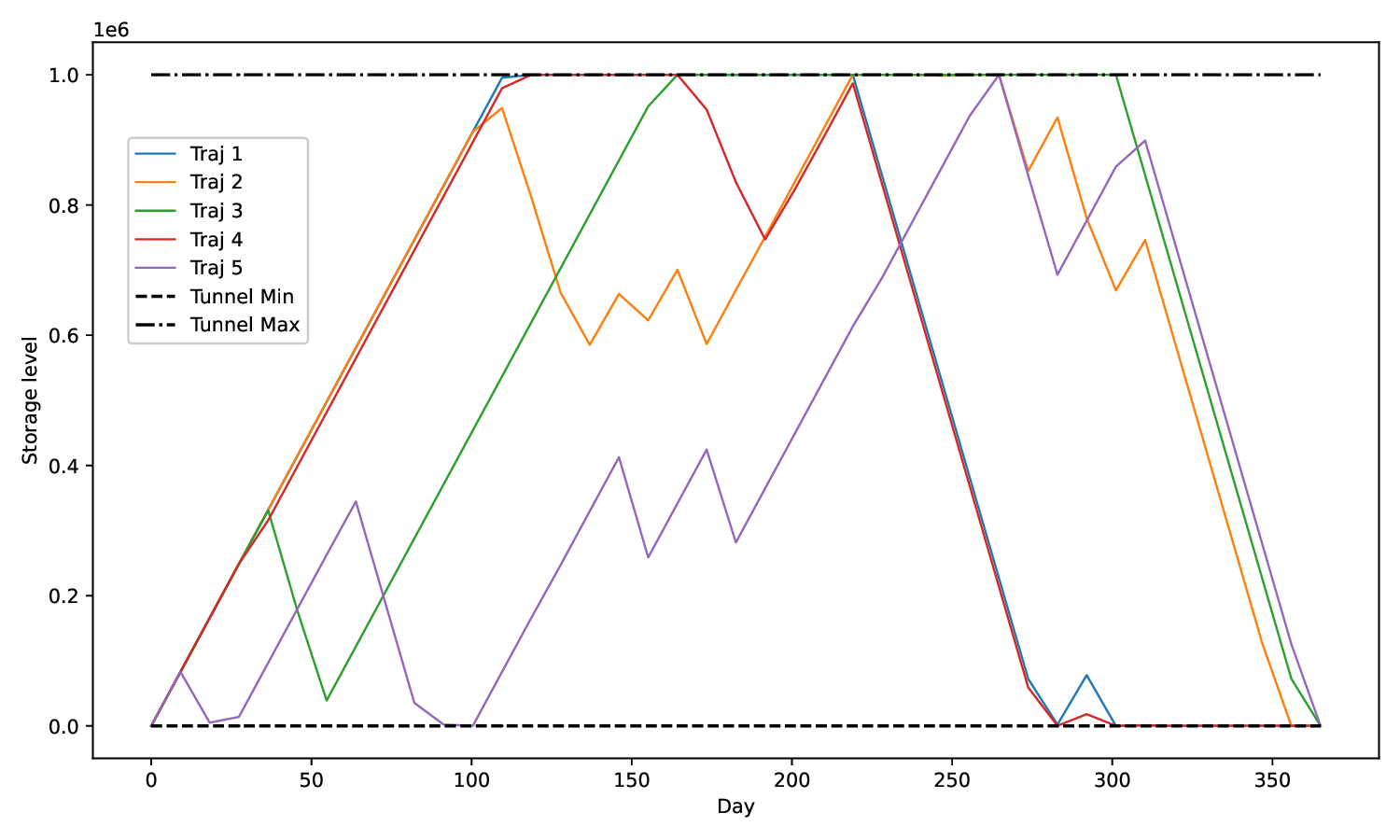}
        \caption{$N=40$}
    \end{subfigure}
    \hfill
    \begin{subfigure}{0.45\textwidth}
        \centering
        \includegraphics[width=\linewidth]{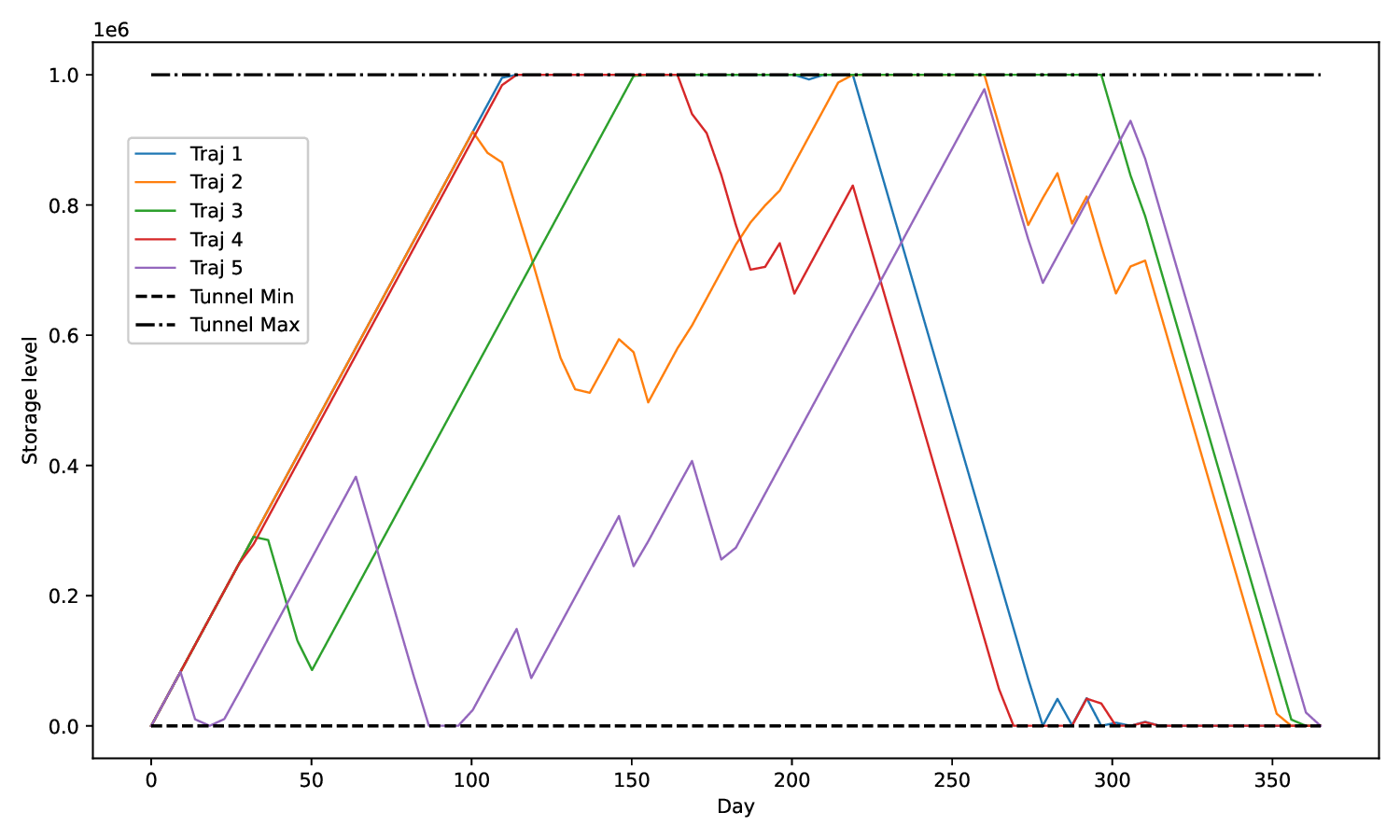}
        \caption{$N=80$}
    \end{subfigure}

    \medskip

    \begin{subfigure}{0.45\textwidth}
        \centering
        \includegraphics[width=\linewidth]{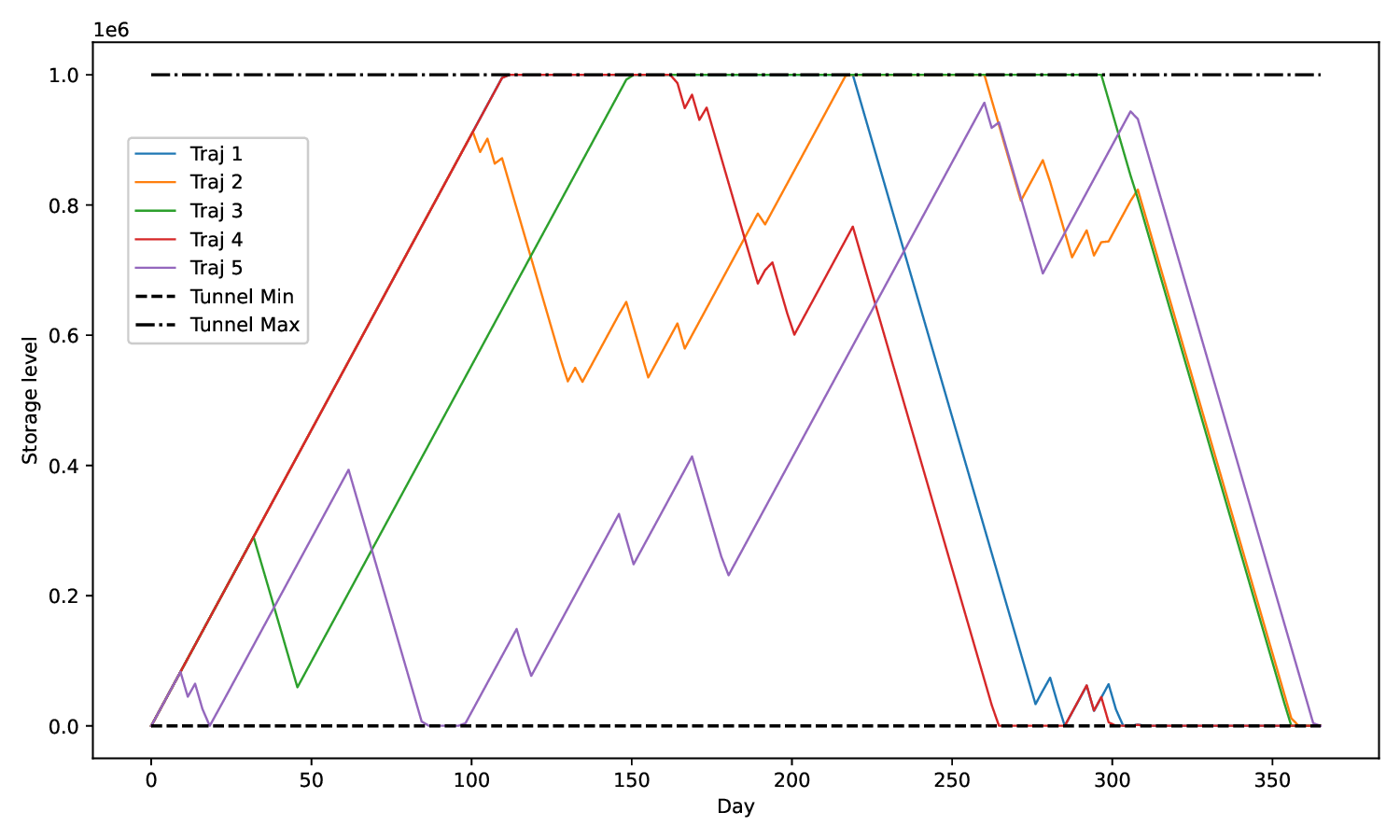}
        \caption{$N=160$}
    \end{subfigure}
    \hfill
    \begin{subfigure}{0.45\textwidth}
        \centering
        \includegraphics[width=\linewidth]{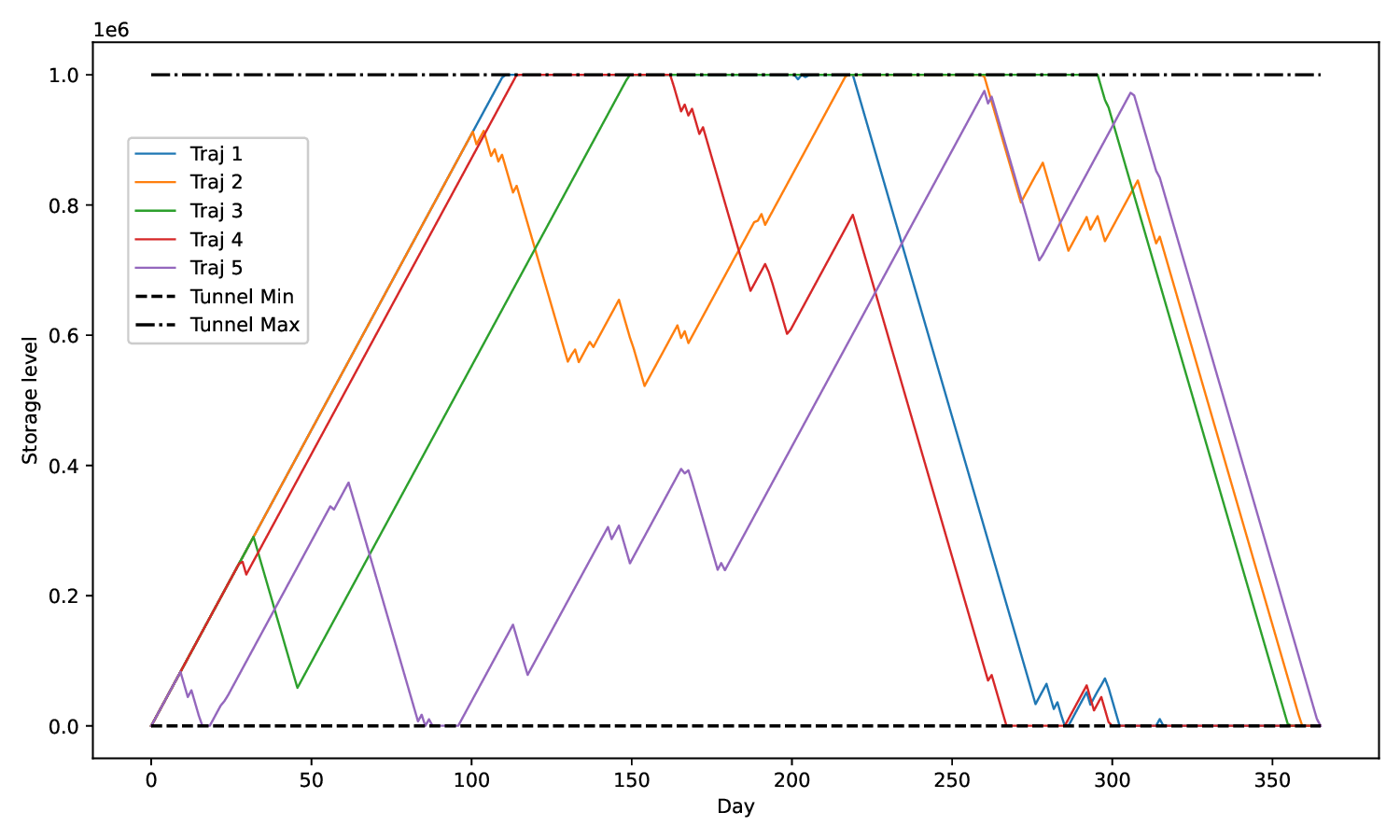}
        \caption{$N=320$ (reference benchmark)}
    \end{subfigure}

    \caption{(Gas storage problem) Simulated trajectories of the storage level for different time grids. Trajectories with the same color correspond to the same underlying realization of the driving Gaussian innovations. The innovations are generated once on the finest grid, $N=320$, and the innovations used on coarser grids are obtained by aggregating consecutive fine-grid increments over the corresponding time intervals. Thus, the different panels compare the policies under the same exogenous noise scenario, refined as the time grid becomes finer.}
    \label{fig:traj_gas_ex}
\end{figure}

\medskip
\Cref{fig:traj_gas_ex} provides a qualitative illustration of the storage policies obtained for different time grids. The same colors correspond to the same underlying realization of the exogenous noise, generated in a nested way across the different grids. Hence, the observed differences between the panels are mainly due to the refinement of the decision grid and the corresponding re-optimization of the neural-network policies under the same piecewise-constant input futures curve.
Several qualitative features can be observed. First, the inventory constraint is satisfied along all simulated trajectories, which confirms the effect of the constraint-preserving parametrization in \eqref{eq:discrete_constraints_final}. Second, the trajectories exhibit consistent storage-management behavior across different discretizations: the reservoir is filled and emptied in response to the evolution of the underlying price factors. As the number of time steps increases, the policies become more flexible in the timing of injection and withdrawal decisions. The refinement from $N=40$ to $N=320$ mainly affects the local timing of the operations, while the overall qualitative structure of the trajectories remains stable.

\medskip
The convergence study compares the values obtained on coarser decision grids with the reference value $v_{\mathrm{ref}}:=v_{320}$. For $N\in\{5,10,20,40,80,160\}$, we define the oriented relative error
$e_N := \frac{v_{\mathrm{ref}}-v_N}{v_{\mathrm{ref}}}$.
In the present experiment, the estimated values satisfy $v_N<v_{\mathrm{ref}}$, so that $e_N$ is positive. Let $\widehat v_N$ and $\widehat v_{\mathrm{ref}}$ denote the Monte Carlo estimators of $v_N$ and $v_{\mathrm{ref}}$, respectively. We estimate $e_N$ by $\widehat e_N = 1-\widehat r_N$, with $\widehat r_N := \frac{\widehat v_N}{\widehat v_{\mathrm{ref}}}$.
To construct confidence intervals for the relative error, we apply the linear delta method to the ratio $\widehat r_N$ (cf. \cite[Sec. 2.2]{delta_method}). Assuming that $\widehat v_N$ and $\widehat v_{\mathrm{ref}}$ are computed from independent Monte Carlo samples, we obtain
\begin{align*}
    \widehat{\mathrm{SE}}(\widehat{e}_N) = \widehat{\mathrm{SE}}(\widehat{r}_N) = \Bigg[ \frac{\big(\widehat{\mathrm{SE}}(\widehat{v}_N)\big)^2}{\big(\widehat{v}_{\mathrm{ref}}\big)^2} + \frac{\big(\widehat{v}_N\big)^2}{\big(\widehat{v}_{\mathrm{ref}}\big)^4}\big(\widehat{\mathrm{SE}}(\widehat{v}_{\mathrm{ref}})\big)^2 \Bigg]^{1/2},
\end{align*}
where $\widehat{\mathrm{SE}}(\cdot)$ denotes the standard error, which refers only to the Monte Carlo error in the final post-training evaluation. More precisely, once the trained policy is fixed, $\widehat v_N$ is computed as the empirical average of the gains over $M=10^8$ independent simulated trajectories, and $\widehat{\mathrm{SE}}(\widehat v_N)$ is the corresponding sample standard deviation divided by $\sqrt{M}$. The confidence intervals therefore quantify the statistical uncertainty of the Monte Carlo evaluation conditional on the trained policy.
Therefore, the two-sided confidence interval with confidence level $1-\alpha$ is given by
\begin{align*}
\mathrm{CI}_{1-\alpha}(\widehat e_N)
=
\Big[
\widehat e_N - z_{1-\alpha/2}\cdot\widehat{\mathrm{SE}}(\widehat e_N)
,\,
\widehat e_N + z_{1-\alpha/2}\cdot\widehat{\mathrm{SE}}(\widehat e_N)
\Big],
\end{align*}
where $z_{1-\alpha/2}$ denotes the $(1-\alpha/2)$-quantile of the standard normal distribution. In the plots and tables below, we use $\alpha=0.05$.

\begin{figure}[!t]
\centering
\begin{subfigure}{0.48\textwidth}
\centering
\includegraphics[width=\linewidth]{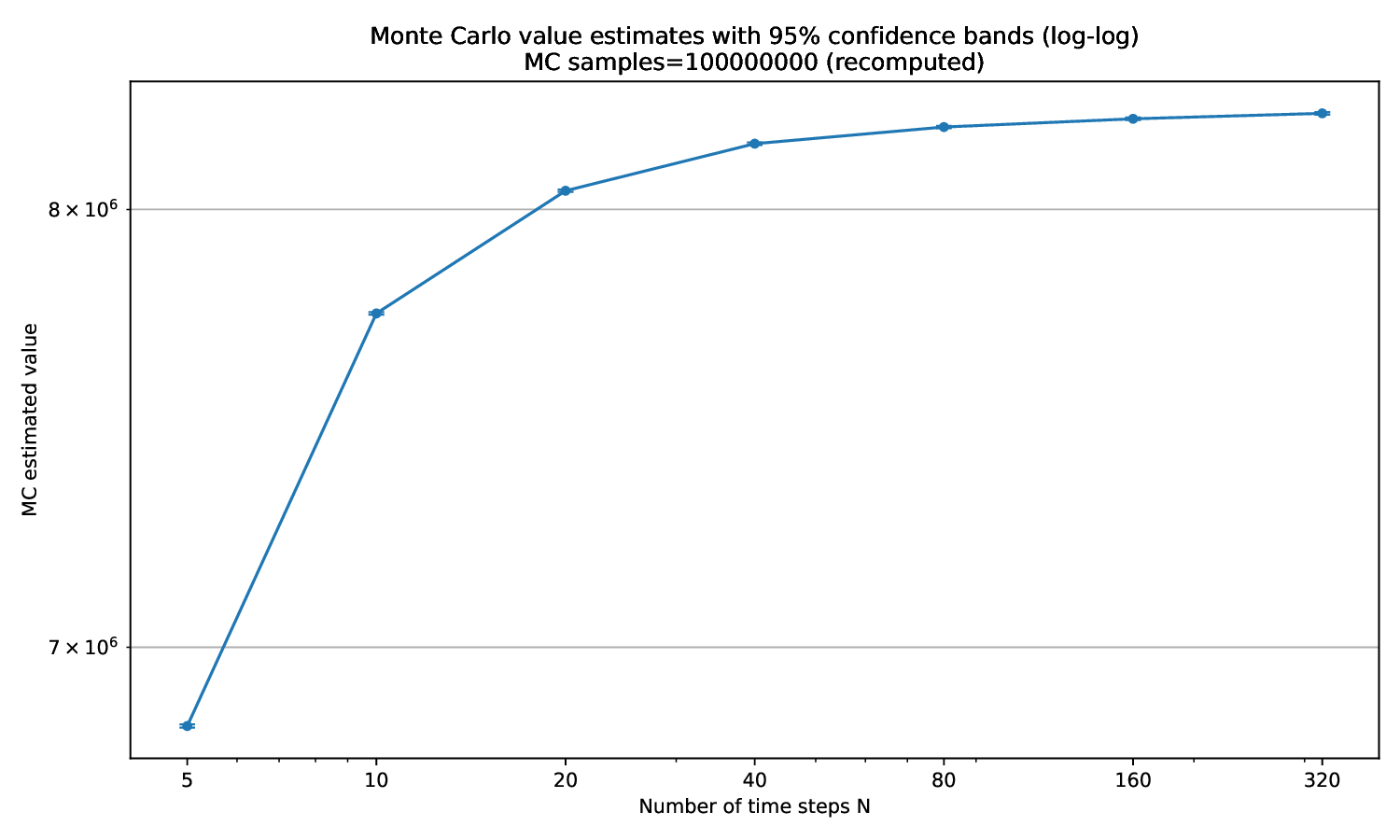}
\caption{Estimated optimized gains.}
\label{fig:gains_gas_ex}
\end{subfigure}
\hfill
\begin{subfigure}{0.48\textwidth}
\centering
\includegraphics[width=\linewidth]{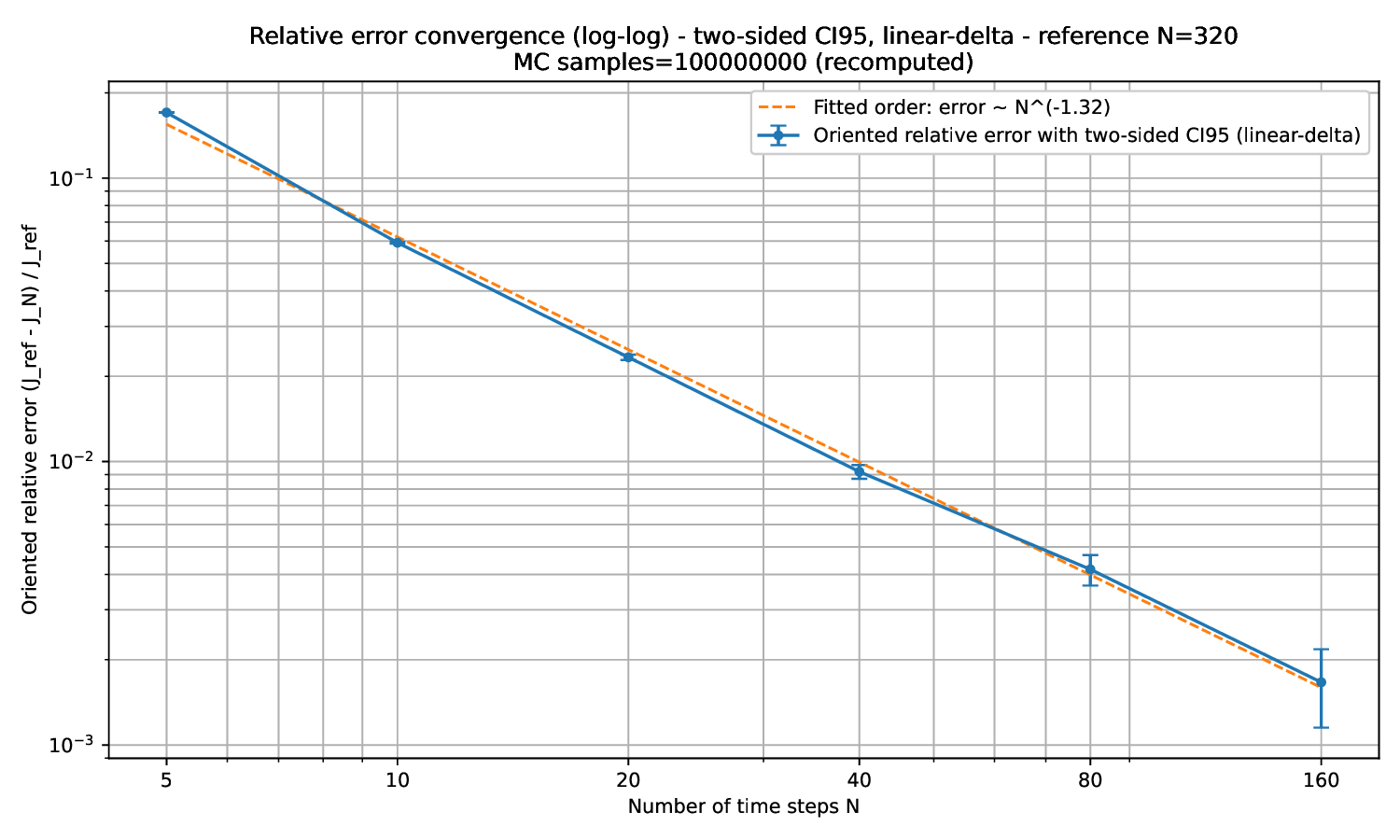}
\caption{Relative error with respect to $N=320$.}
\label{fig:relative_error_gas_ex}
\end{subfigure}
\caption{(Gas storage problem) \emph{Left:} Monte Carlo estimates of the optimized gain for different time grids, with $95\%$ confidence intervals, in log--log scale. \emph{Right:} Oriented relative error $(v_{320}-v_N)/v_{320}$ in log--log scale, with confidence intervals computed by the linear delta method and fitted algebraic decay.}
\label{fig:convergence_gas_ex}
\end{figure}

\medskip
\Cref{fig:convergence_gas_ex} summarizes the quantitative convergence results. The left panel shows that the optimized gains increase monotonically with the number of time steps. This is consistent with the fact that finer grids provide a richer class of piecewise-constant policies and therefore allow for more flexible injection and withdrawal decisions. Although such monotonicity is not enforced by the neural-network parametrization and training procedure, it is clearly observed in the numerical results, suggesting that the approximation and optimization errors are sufficiently controlled in this experiment.
The right panel reports the relative errors $e_N$ in log--log scale. The error decreases regularly as the time grid is refined. The dashed line is obtained by a least-squares linear regression in log--log scale over the points $N=5,\dots,160$, yielding the empirical behavior $e_N \simeq C N^{-1.32}$.
This rate should be interpreted as an empirical convergence rate with respect to the numerical reference value $v_{320}$, rather than as an exact theoretical order. Nevertheless, the regular decay confirms that refining the decision grid effectively reduces the loss induced by the piecewise-constant policy restriction. The confidence intervals are barely visible in the gain plot, since the final values are recomputed with $10^8$ independent Monte Carlo trajectories; in the log--log error plot, the confidence intervals appear asymmetric because they are displayed on a logarithmic scale.

\medskip
The corresponding numerical values are reported in \Cref{tab:storage_convergence}, together with the Monte Carlo confidence intervals, the relative errors, and the observed local rates.
\begin{table}[htbp]
\centering
\caption{(Gas storage problem) Values $\widehat v_N$, errors $\widehat e_N$ and estimated order of convergence.
The estimated values $\widehat v_N$ are recomputed using $10^8$ independent Monte Carlo trajectories. The reference value is $\widehat v_{320}$. The reported confidence intervals are half-widths. For instance, the entry
$\widehat v_N=6.8338 \times 10^6$ with $\mathrm{CI}_{95\%}(\widehat v_N)=3.356 \times 10^3$
corresponds to
$\widehat v_N=(6.8338\pm3.356\times10^{-3})\times10^6$.
\label{tab:storage_convergence}
}
\begin{tabular}{c|c c c c c}
\toprule
{$N$}
& {$\widehat v_N$}
& {$\mathrm{CI}_{95\%}(\widehat v_N)$}
& {$\widehat e_N$}
& {$\mathrm{CI}_{95\%}(\widehat e_N)$}
& {Rate} \\
\midrule
5
& $6.8338 \times 10^6$
& $3.356 \times 10^3$
& $1.7045 \times 10^{-1}$
& $5.067 \times 10^{-4}$
& {--} \\
10
& $7.7503 \times 10^6$
& $3.088 \times 10^3$
& $5.9196 \times 10^{-2}$
& $5.073 \times 10^{-4}$
& $1.526$ \\
20
& $8.0457 \times 10^6$
& $3.022 \times 10^3$
& $2.3331 \times 10^{-2}$
& $5.104 \times 10^{-4}$
& $1.343$ \\
40
& $8.1621 \times 10^6$
& $3.005 \times 10^3$
& $9.2102 \times 10^{-3}$
& $5.125 \times 10^{-4}$
& $1.341$ \\
80
& $8.2036 \times 10^6$
& $2.998 \times 10^3$
& $4.1670 \times 10^{-3}$
& $5.131 \times 10^{-4}$
& $1.144$ \\
160
& $8.2242 \times 10^6$
& $2.993 \times 10^3$
& $1.6652 \times 10^{-3}$
& $5.134 \times 10^{-4}$
& $1.323$ \\
320
& $8.2379 \times 10^6$
& $2.993 \times 10^3$
& {--}
& {--}
& {--} \\
\bottomrule
\end{tabular}
\end{table}
The table quantitatively confirms the behavior observed in \Cref{fig:convergence_gas_ex}. The Monte Carlo confidence intervals are small compared with the differences between consecutive discretizations, especially for coarse and intermediate grids. The local rates are relatively stable and mostly close to the fitted order observed in the log--log plot, with some expected variability due to the use of a numerical reference value and the presence of neural-network approximation and optimization errors. Since the exogenous stochastic factors are simulated exactly on the grid, the observed convergence directly reflects the discretization error induced by piecewise-constant policies. Overall, the results show a clear and consistent reduction of the error as the decision grid is refined.

\bibliographystyle{abbrv}
\bibliography{ref}

\appendix
\section{Technical results}
\label{appendix:technical_results}

\subsection{Proof of Lemma~\ref{lemma:3}}\label{appendix:subsec-1}
\begin{proof}
    We show the proof for $p=2$; the general case $p\ge2$ follows from the same argument, using the discrete-time Burkholder-Davis-Gundy inequality in $L^p$. This is a standard moment estimate (see, for instance, \cite[Theorem 4.5.4]{kloeden_numerical_1992}), and we report an adapted proof for the discrete-time setting, to make explicit the constant $C$ that appears in the estimate. By considering the recursion for the process $x_k := \widetilde{X}^{x,\alpha}_{t_k}$, we get
    \begin{align*}
        \E\bigg[ \max_{0 \leq k \leq n} |x_k - x|^2 \bigg] \leq 2\E\bigg[ \max_{0 \leq k \leq n} \Big| \sum_{i=0}^{k-1} b_i(x_i,\alpha_i(x_i))\tau \Big|^2 \bigg] + 2\E\bigg[ \max_{0 \leq k \leq n} \Big| \sum_{i=0}^{k-1} \sigma_i(x_i,\alpha_i(x_i))\sqrt{\tau}Z_{i+1} \Big|^2 \bigg].
    \end{align*}
    For the first term on the right-hand side, we get
    \begin{align*}
        \E\bigg[ \max_{0 \leq k \leq n} \Big| \sum_{i=0}^{k-1} b_i(x_i,\alpha_i(x_i))\tau \Big|^2 \bigg] &\leq \tau n \sum_{i=0}^{n-1} \E[|b_i(x_i,\alpha_i(x_i))|^2]\tau\\
        &\leq \tau n \sum_{i=0}^{n-1} \big(2[b]_3^2 + 2[b]_1^2\E[|x_i|^2]\big)\tau,
    \end{align*}
    where $[b]_3$ is a positive constant such that
    $\sup_{t\in[0,T]}|b(t,0,a)|\leq[b]_3$ for all $a\in A$; such a constant exists
    by \Cref{ass:ex_uniq_SDE}.
    For the second term, instead, we take into account the Burkholder-Davis-Gundy inequality in the discrete time setting, and similarly to before, we obtain
    \begin{align*}
        \E\bigg[ \max_{0 \leq k \leq n} \Big| \sum_{i=0}^{k-1} \sigma_i(x_i,\alpha_i(x_i))\sqrt{\tau}Z_{i+1} \Big|^2 \bigg] &\leq 4 \sum_{i=0}^{n-1} \E[|\sigma_i(x_i,\alpha_i(x_i))|^2]\tau\\
        &\leq 4 \sum_{i=0}^{n-1} \big( 2[\sigma]_3^2 + 2[\sigma]_1^2 \E[|x_i|^2] \big)\tau,
    \end{align*}
    where $[\sigma]_3>0$ is such that $|\sigma(0,a)|\leq[\sigma]_3$, for all $a\in A$.
    
    Define now $\Gamma_n := \E [ \max_{0 \leq k \leq n} |x_k - x|^2 ]$, and observe that $\E[|x_i|^2] \leq 2\Gamma_i + 2|x|^2$. Then we have obtained the following
    \begin{align*}
        \Gamma_n &\leq 4[b]_3^2(\tau n)^2 + 16[\sigma]_3^2(\tau n) + 8( [b]_1^2 \tau n + 4[\sigma]_1^2 )\sum_{i=0}^{n-1}\Gamma_i \tau + 8( [b]_1^2 \tau n + 4[\sigma]_1^2 )|x|^2 \tau n\\
        &\leq \big( ( 4[b]_3^2 T + 16[\sigma]_3^2 ) + ( 8[b]_1^2 + 4[\sigma]_1^2 ) |x|^2 \big) \tau n + 8 ( [b]_1^2 T + 4[\sigma]_1^2 ) \sum_{i=0}^{n-1}\Gamma_i \tau.
    \end{align*}
    Finally, by a discrete Gronwall's lemma, we get
    \begin{align*}
        \Gamma_n \leq (C_1 + C_2|x|^2)\tau n \cdot e^{C_3 \tau n},
    \end{align*}
    where $C_1 := 4[b]_3^2 T + 16[\sigma]_3^2$, $C_2 := 8[b]_1^2 + 4[\sigma]_1^2$ and $C_3 := 8( [b]_1^2 T + 4[\sigma]_1^2 )$. We conclude, taking $C := \max \{ C_1, C_2, C_3 \}$ that
    \begin{align*}
        \E\bigg[ \max_{0 \leq k \leq n} |x_k - x|^2 \bigg] \leq C (1 + |x|^2)T e^{C T}.
    \end{align*}
\end{proof}

\subsection{Density argument with respect to Radon measures}
\begin{lemma}\label{lemma:density_argument_law}
    Let $\mu$ be a finite Radon measure on $\R^d$ and let $A\subset\R^k$ be non-empty, compact, and convex. Then, for every Borel measurable function $\phi:\R^d\to A$ and every $\delta>0$, there exists a Lipschitz continuous function $\phi^\delta:\R^d\to A$ such that
    \begin{align*}
        \int_{\R^d}|\phi^\delta(x)-\phi(x)|^2\,\mu(\de x)<\delta.
    \end{align*}
\end{lemma}
\begin{proof}
    Since $A$ is compact, $\phi$ is bounded and therefore belongs to $L^2(\mu;\R^k)$, where
    \begin{align*}
        L^2(\mu;\R^k) := \Big\{ \varphi:\R^d\to\R^k \,\text{measurable}\, \ \Big| \ \|\varphi\|^2_{L^2(\mu)} := \int_{\R^d}|\varphi(x)|^2\mu(\de x) < \infty \Big\}.
    \end{align*}
    By the density of continuous compactly supported functions in $L^2(\mu;\R^k)$ (cf. \cite[Prop.~7.9]{Folland1999}), there exists $h\in C_c(\R^d;\R^k)$ such that
    \begin{align*}
        \int_{\R^d}|h(x)-\phi(x)|^2\,\mu(\de x)<\frac{\delta}{4}.
    \end{align*}
    In addition, by a standard mollification argument, we can choose $h^\delta\in C_c^\infty(\R^d;\R^k)$, hence globally Lipschitz, such that
    \begin{align*}
        \int_{\R^d}|h^\delta(x)-h(x)|^2\,\mu(dx)<\frac{\delta}{4}.
    \end{align*}
    To conclude the proof, it is enough to observe that, $A$  being closed and convex, the projection $\Pi_A:\R^k\to A$ is well-defined and $1$-Lipschitz. Therefore, $\phi^\delta:=\Pi_A\circ h^\delta$ is Lipschitz continuous and takes values in $A$. Moreover, since $\phi(x)\in A$ for every $x\in\R^d$, we have
    \begin{align*}
        |\phi^\delta(x)-\phi(x)|
        =
        |\Pi_A(h^\delta(x))-\Pi_A(\phi(x))|
        \leq
        |h^\delta(x)-\phi(x)|.
    \end{align*}
    Consequently,
    \begin{align*}
        &\int_{\R^d}|\phi^\delta(x)-\phi(x)|^2\,\mu(\de x)
        \leq
        \int_{\R^d}|h^\delta(x)-\phi(x)|^2\,\mu(\de x) \\
        &\hspace{2cm} \leq
        2\bigg( \int_{\R^d}|h^\delta(x)-h(x)|^2\,\mu(\de x)
        +
        \int_{\R^d}|h(x)-\phi(x)|^2\,\mu(\de x) \bigg) < \delta.
    \end{align*}
\end{proof}

\section{Explicit computations for Example~\ref{sec:target_circle_example}}
\label{appendix:target_ex_computations}

\subsection{Radial reduction and exact solution.}
Define the radius $R_s := |X_s|$. By It\^o's formula, we get a finite variation evolution for the square of the radius,
\begin{align}\label{eq:ex2_dyn_radius_squared}
    \de R^2_s = 2\langle X_s,\, \alpha_s \rangle \de s + R^2_s \de s, \quad R^2_t = |x|^2.
\end{align}
Thus, let $r:=|x|$ and define $v(t,r)$ as the radial reduction of the HJB. For $r=|x|>0$, the radial derivatives satisfy
\begin{align*}
    \nabla V (t,x) &= \frac{\partial v}{\partial r}\nabla r (t,x)= v_{r}(t,|x|)\frac{x}{|x|} = v_{r}(t,|x|)\frac{x}{r},\\
    \nabla^2 V(t,x) &= v_{rr}(t,|x|)\frac{xx^\top}{r^2} + \frac{v_r}{r}(t,|x|)\bigg(I_2 - \frac{xx^\top}{r^2} \bigg),
\end{align*}
where $v_r$ and $v_{rr}$ denote the partial derivatives of $v$ with respect to the radial variable. Also, $\mathrm{Tr}[\sigma\sigma^\top(t,x)\nabla^2 V(t,x)] = r v_r (t,|x|)$ and we derive the HJB equation for $v$, in a viscosity sense,
\begin{align*}
\begin{cases}
    \partial_t v(t,r) + \frac{1}{2}r v_r (t,r) - M|v_r(t,r)| = 0, \quad (t,r) \in [0,T)\times\R_+\\
    v(T,r) = (r - r_0)^2.
\end{cases}
\end{align*}
Whenever $x\neq0$ and $v_r(t,|x|)\neq0$, the optimal feedback selection of \eqref{eq:alpha_opt_circle} is then given by $\alpha^*(t,x) = -M\frac{x}{|x|}\mathrm{sgn}(v_r(t,|x|))$.
Inside the zero-level set, $v$ is locally flat, and the optimal controls are generally non-unique.

\subsection{Backward reachable set.}
We derive an explicit characterization of the boundary curves of the backward reachable set and the associated value function. According to \eqref{eq:ex2_dyn_radius_squared}, $R_s := |X_s| = \sqrt{R^2_s}$ satisfies the finite-variation dynamics
\begin{align*}
    \de R_s = \bigg( \frac{R_s}{2} + u_s \bigg)\de s, \qquad u_s := \bigg\langle \alpha_s, \, \frac{X_s}{|X_s|} \bigg \rangle,
\end{align*}
where $u_s$ represents the radial component of the control. All expressions involving $X_s/|X_s|$ are meant for $X_s\neq0$, and the case $X_s=0$ is treated by continuity in the radial variable.

The boundary curves of the backward reachable set are obtained by imposing the extremal admissible radial controls. The lower boundary corresponds to the maximal outward radial control $u \equiv M$, while the upper boundary corresponds to the maximal inward radial control $u \equiv -M$. The two curves are then expressed as solutions of the following linear ODEs:
\begin{align*}
\begin{cases}
    \dot{\rho}_{-}(t) = \frac{1}{2}\rho_{-}(t) + M, \quad \rho_{-}(T) = r_0, \qquad\qquad &\text{thus} \quad \rho_{-}(t) = -2M + (r_0+2M)e^{(t-T)/2};\\
    \dot{\rho}_{+}(t) = \frac{1}{2}\rho_{+}(t) - M, \quad \rho_{+}(T) = r_0, \qquad\qquad &\text{thus} \quad \rho_{+}(t) = 2M + (r_0-2M)e^{(t-T)/2}.
\end{cases}
\end{align*}
It follows that the zero-level set boundaries are explicitly characterized through $\rho_{-}(t)$ and $\rho_{+}(t)$,
\begin{align*}
    \Zcal_t = \big\{ x\in\R^2 \,:\, \max \{ 0, \rho_{-}(t) \} \leq |x| \leq \rho_{+}(t) \big\},
\end{align*}
and the explicit formula for the value function is
\begin{align*}
    V(t,x) = 
    \begin{cases}
        e^{T-t} ( \rho_{-}(t) - |x| )^2, \qquad &\, \rho_{-}(t)>0, \quad 0 \leq |x| < \rho_{-}(t),\\
        0, \qquad &\, x \in \mathcal{Z}_t,\\
        e^{T-t} ( |x| - \rho_{+}(t) )^2, \qquad &\, |x| > \rho_{+}(t).\\
    \end{cases}
\end{align*}
Moreover, outside the zero-level set, the sign of $v_r$ is determined by the relative position of $r$ with respect to the reachable interval. More precisely, $v_r(t,r)<0$ if $r<\rho_-(t)$, and $v_r(t,r)>0$ if $r>\rho_+(t)$.
Hence, whenever $x\notin\mathcal{Z}_t$ and $x\neq0$, an optimal feedback is given by
\begin{align*}
\alpha^*(t,x)=
\begin{cases}
+M\frac{x}{|x|}, & |x|<\max\{0,\rho_-(t)\},\\
-M\frac{x}{|x|}, & |x|>\rho_+(t),
\end{cases}
\end{align*}
i.e., if $|x|>\rho_+(t)$, the control acts inward, decreasing the radius as much as possible;
if $|x|<\rho_-(t)$, it acts outward, increasing the radius as much as possible.

\section{Explicit weak order 2.0 scheme}
\label{appendix:KPscheme}
Let $(W_t)_{t\geq0}$ an $m$-dimensional standard Brownian motion, we consider a $d$-dimensional controlled stochastic differential equation of the form
\begin{align*}
    \de X_t = b(t,X_t,\alpha_t)\de t + \sigma(t,X_t,\alpha_t)\de W_t,\quad X_0=\xi,\quad t\in[0,T], 
\end{align*}
for which we assume the control $\alpha_t=\alpha(t,X_t)$ in feedback-form. For notational simplicity, we denote the discretized version, corresponding to our discrete-time setting in \cref{sub:discrete_time_setting}, by $(Y_n)_{n=0}^{N}$, where the size of the time step is $\Delta t:= T/N$ and $\Delta W_n:= \sqrt{\Delta t}Z_n$, with $Z_n \, \mathrm{i.i.d.} \, \sim \mathcal{N}(0_m,\mathrm{I}_m)$, representing the Brownian increment.

We consider explicit scheme proposed by Kloeden and Platen and which was originally formulated for autonomous SDEs 
(\cite[eq. (15.1.3)]{kloeden_numerical_1992}).  Since the optimal control problems under consideration in this work are inherently non-autonomous, the coefficients may be time-dependent, and the feedback control policy $\alpha_t = u(t, X_t)$ introduces an explicit dependence on the temporal variable $t$. Hence we treat time as an additional state variable by defining an augmented state vector $\bar{X}_s = (s, X_s)^\top$. 
This leads to the following expressions,
for Example~1 (\Cref{sec:target_circle_example}):
\be
    Y_{n+1} 
     &= & Y_n + \frac{1}{2}\Big(b(t_n, Y_n,\ma_n) + b(t_{n+1},\bar{Y}_{n+1},\bar\ma_{n+1}) \Big)\Delta t \nonumber \\
     &  & \hspace{10ex}+ \frac{1}{4}\Big( \sigma(Y^+_{n+1}) + \sigma(Y^-_{n+1}) + 2\sigma(Y_n) \Big)\Delta W_{n+1} \nonumber \\
     &  & \hspace{10ex}+ \frac{1}{4}\Big( \sigma(Y^+_{n+1}) - \sigma(Y^-_{n+1}) \Big) \big( (\Delta W_{n+1})^2 - \Delta t \big) \Delta t^{-1/2}
  \label{2order_scheme_circle_example}
\ee
with
$$
\begin{aligned}
    &  \bar{Y}_{n+1}  := Y_n + b(t_n,Y_n,\ma_n)\Delta t + \sigma(Y_n)\Delta W_{n+1},\\
    &  Y^\pm_{n+1}   := Y_n + b(t_n,Y_n,\ma_n)\Delta t  \pm \sigma(Y_n)\sqrt{\Delta t},\\
    &  \ma_n   := \ma(t_n,Y_n), \qquad \bar\ma_{n+1} :=\ma(t_{n+1},\bar{Y}_{n+1})
\end{aligned}
$$
and, for Example~2, with $\ms$ constant (\Cref{sec:HJB_example}):
\begin{equation}
\label{2order_scheme_EHJ_example}
\begin{aligned}
    & Y_{n+1} = Y_n + \frac{1}{2}\Big(b(\alpha(t_{n+1},\bar{Y}_{n+1})) + b(\alpha(t_n, Y_n))\Big)\Delta t + \sigma \Delta W_{n+1},\\
    & \bar{Y}_{n+1} = Y_n + b(\alpha(t_n, Y_n))\Delta t + \sigma \Delta W_{n+1}.
\end{aligned}
\end{equation}

\end{document}